%% file: main.tex
\documentclass[reqno]{amsart}
\usepackage[margin = 1in]{geometry}
\usepackage{amsmath, amssymb, amsthm, fancyhdr, verbatim, graphicx}
\usepackage{enumerate}
\usepackage[all]{xy}
\usepackage[dvipsnames]{xcolor}
\usepackage{mathrsfs}
\usepackage{tikz-cd}
\usepackage{framed}
\usepackage[hidelinks, pagebackref]{hyperref}
\hypersetup{colorlinks=true, linkcolor=magenta,citecolor=cyan} 
\usepackage[titletoc]{appendix}
\usepackage{bbm}
\usepackage{lipsum}
\usepackage{adjustbox}
\usepackage{mathfyz}
\usepackage{mathzyw}
\usepackage{mathsrl}
\usepackage{envi}
\usepackage{stmaryrd}
\usepackage{leftindex}
\usepackage[english]{babel}

\usepackage{refsrl} 

\numberwithin{equation}{section}

\makeatletter
\def\th@remark{%
  \thm@headfont{\bfseries}%
  \normalfont 
  \thm@preskip \thm@preskip 
  \thm@postskip\thm@preskip
}
\def\imod#1{\allowbreak\mkern5mu({\operator@font mod}\,\,#1)}
\makeatother

\numberwithin{equation}{section}

\title{Higher Period Integrals and Derivatives of $L$-functions}
\author{Shurui Liu}
\address{Stanford University, Department of Mathematics, Building 380, Stanford, CA 94305, USA}
\email{srliu@stanford.edu}
\author{Zeyu Wang}
\address{Massachusetts Institute of Technology, Department of Mathematics, 77 Massachusetts Avenue, Cambridge, MA 02139, USA}
\email{wangzeyu@mit.edu}

\begin{document}

\begin{abstract}

  We propose a geometric framework to produce a formula relating higher period integrals to higher central derivatives of $L$-functions over function fields, extending the framework \cite{BZSV} to higher derivatives. For a strongly tempered affine smooth $G$-variety $X$, we give a geometric construction of the action of $L$-observables on the geometric period integral $\int_{X}\mathbb{L}_\s$ of a Hecke
  eigensheaf $\mathbb{L}_\s$ on $\Bun_G$. By taking a suitable version of Frobenius trace of this action, we recover higher central derivatives of the $L$-function attached to the dual symplectic representation. As an application, in the Rankin--Selberg case $(\GL_n\times\GL_{n-1},\GL_{n-1})$, we obtain a formula for higher derivatives of the Rankin--Selberg $L$-function. This provides a conceptual generalization of the higher Gross--Zagier formula of \cite{YZ1} to higher-dimensional spherical varieties.

\end{abstract}

\maketitle

\tableofcontents

\input{newintro.tex}

\section{Notations and conventions}
In this section, we summarize the notations and conventions used throughout the article.
\subsection{General}\label{gen}
For each $n\in\Z_{\geq 1}$, define $[n]=\{1,2,\dots,n\}$. We use $S_n$ to denote the symmetric group of the set $[n]$.

\subsection{Categories}\label{cat}
By a category, we mean an $(\infty,1)$-category. For a category $\sC$, we use $\sC^{\om}$ to denote its full subcategory of compact objects. For a stable $\infty$-category (or a triangulated category) with t-structure, we use $\sC^{\he}$ to denote its heart. We use $\sC^{>-\infty}$ to denote the full subcategory of objects bounded on the left.

We use $\rrr{Cat}^{\rrr{Idem}}_k$ to denote the $\infty$-category of small idempotent complete $k$-linear categories. We use $\LinCat_k$ to denote the $\infty$-category of presentable $k$-linear categories with continuous morphisms (i.e., morphisms preserving colimit).

\subsection{Algebraic geometry}\label{ag}
In this article, we by default work with schemes/(\'etale) stacks over $\F_q$ on the automorphic side, and algebraic stacks over $k=\Ql$ or a sufficiently large finite extension of $\Ql$ on the spectral side. To $X$ a stack over $\F_q$, we will attach in \autoref{sheaftheory} a category of (ind-constructible) \'etale sheaves on $X_{\overline{\F}_q}$ which we denote by $\Shv(X_{\overline{\F}_q})$ or sometimes we write it as $\Shv(X)$. For $\cF,\cG\in\Shv(X)$, we write $\Hom(\cF,\cG):=\Hom_{\Shv(X)}(\cF,\cG)$.

\subsection{Super linear algebra}\label{superlinear}
In this paper, we use $k$ to denote a sufficiently large field with $\mathrm{char}(k)=0$, which is usually taken to be $\Ql$ (or a finite extension of $\Ql$) for $l\neq p$. In this section, we introduce our conventions on super vector spaces. Although these conventions are standard, the sign involved can be subtle, so we fix the sign rule here.

We use $\Vect$ to denote the rigid symmetric monoidal stable $\infty$-category of (complexes of) super vector spaces over $k$, and any linear category would be linear over $\Vect$. We use $\sw:V\otimes W\isom W\otimes V$ to denote the commutative constraint of $\Vect$. We use $\Vect^0\sub\Vect$ to denote the full subcategory of even vector spaces. Regarding $\Vect$ as $\Z/2\Z$-graded objects in $\Vect^0$, we get a natural forgetful functor $\oblv:\Vect\to\Vect^0$. There is another functor $(\bullet)_0:\Vect\to \Vect^0$ defined as taking the even part. Moreover, the category $\Vect$ is equipped with an endomorphism $\Pi:\Vect\to \Vect$ given by reversing the parity, which gives rise to a symmetric monoidal endomorphism $\langle 1\rangle=\Pi[1]:\Vect\to \Vect$. We obtain another functor $(\bullet)_1=(\bullet)_0\circ\Pi:\Vect\to \Vect^0$ that takes the odd part of a super vector space.

For any linear category $\sC$ and objects $c,d\in\sC$, we use $\Hom_{\sC}(c,d)\in \Vect$ to denote the Hom space enriched over $\Vect$. We write $\Hom^0_{\sC}(c,d):=H^0(\Hom_{\sC}(c,d))\in\Vect$. We often drop $\sC$ if $\sC$ is clear from the context.

The subcategory of compact objects $\Vect^{\om}$ consists of bounded complexes of finite-dimensional super vector spaces and hence, coincides with dualizable objects in $\Vect$. For each $V\in\Vect^{\om}$, its (right) dual is denoted by $V^*$. There are natural unit and counit maps \[\ev:V\otimes V^*\to k\] \[\coev:k\to V^*\otimes V.\]

For every morphism $f:V\to W$ between $V,W\in\Vect^{\om}$, the dual homomorphism $f^*:W^*\to V^*$ is defined as \[f^*:W^*\xrightarrow{\coev_V\otimes \id}V^*\otimes V\otimes W^*\xrightarrow{\id\otimes f\otimes \id}V^*\otimes W\otimes W^*\xrightarrow{\id\otimes \ev_W} V^*.\]

Note that we have \[W^*\otimes V^*\isom (V\otimes W)^*\] where the unit and counit maps are given by \[\ev_{V\otimes W}:(V\otimes W)\otimes (W^*\otimes V^*)\xrightarrow{\id\otimes\ev_W\otimes\id}V\otimes V^*\xrightarrow{\ev_V} k\] \[\coev_{V\otimes W}:k\isom k\otimes k \xrightarrow {\coev_W\otimes\coev_V}W^*\otimes W\otimes V^*\otimes V\xrightarrow{\id\otimes\sw\otimes\id}(W^*\otimes V^*)\otimes (V\otimes W).\]

From now on, we restrict our discussion to $\Vect^{\he}\sub\Vect$.

For $V\in \Vect^{\he}$, an \emph{(even) element} in $V$ is a linear map $k\to V$, and an \emph{odd element} in $V$ is a linear map $\Pi k\to V$. By an element $v\in V$, we mean an element $v\in\oblv(V)$. It is obvious that an even (resp. odd) element of $V$ is an element of $V_0$ (resp. $V_1$) and hence can be regarded as an element of $V$. We call an even or odd element of $V$ a pure element. For each pure element $v\in V$, we define $|v|$ to be its parity, which is zero if $v$ is even and one if $v$ is odd. Similarly, one defines (even, odd, pure) \emph{coelements} as functionals on $\oblv(V)$.

For any $F\in\End(V)$, we use $\tr(F)\in k$ to denote the (super)trace
\begin{equation}\label{tracedef}
    \tr(F):k\xrightarrow{\coev}V^*\otimes V\xrightarrow{\id\otimes F} V^*\otimes V\xrightarrow{\sw} V\otimes V^*\xrightarrow{\ev} k.
\end{equation}
In concrete terms, the (super)trace is related to the usual trace of endomorphisms of vector spaces by $\tr(F)=\tr(F|_{V_0})-\tr(F|_{V_1})$.

For $F\in\Aut(V)$, we define the (super)determinant of $F$ (also known as Berezinian in the literature) to be\begin{equation}\label{detdef}
    \det(F):=\det(F|_{V_0})\det(F|_{V_1})^{-1}\in k.
\end{equation}

By a bilinear form between $V, W\in\Vect^{\he}$, we mean an even coelement of $V\otimes W$. Note that we have a $S_n$-action on $V^{\otimes n}$ induced by $\sw$. We use \begin{equation}\label{swap} \sw_{i,i+1}\in\End(V^{\otimes n})\end{equation} to denote the action of the permutation $(i,i+1)\in S_n$. The action of $S_n$ on $V^{\otimes n}$ gives a direct sum decomposition \[V^{\otimes n}=\bigoplus_{\x\in\Irr(S_n)} (V^{\otimes n})_{\x}.\] We define $\Sym^n V:=(V^{\otimes n})_{\triv}$ and $\bigwedge^n V:=(V^{\otimes n})_{\sgn}$. A bilinear form on $V$ is called symplectic (resp. orthogonal) if it comes from a coelement of $\bigwedge^2V$ (resp. $\Sym^2 V$). For a bilinear form $b: V\otimes W\to k$, we write $b(v,w)=b(v\otimes w)$ for $v\in V$ and $w\in W$, which are pure elements (get zero if they do not have the same parity). 

For $v_1,\dots,v_n\in V$, we write
\[v_1v_2\cdots v_n:= \frac{1}{n!}\sum_{g\in S_n}g(v_1\otimes v_2\otimes\cdots\otimes v_n)\in \Sym^n V\sub V^{\otimes n}\] and \[v_1\wedge v_2\wedge\cdots\wedge v_n:=\frac{1}{n!}\sum_{g\in S_n}\sgn(g)g(v_1\otimes v_2\otimes\cdots\otimes v_n)\in \Sym^n V\sub V^{\otimes n}.\] Under this convention, the natural quotient map $V^{\otimes n}\to \Sym^n V$ is $v_1\otimes\cdots\otimes v_n\mapsto v_1\cdots v_n$, and the natural quotient map $V^{\otimes n}\to \bigwedge^n V$ is $v_1\otimes\cdots\otimes v_n\mapsto v_1\wedge\cdots\wedge v_n$.

For each $n\in \Z_{\geq 0}$, We have $\Sym^n V^*\isom (\Sym^n V)^*$ where the evaluation map is given by \begin{equation}\label{symev}\ev_{\Sym^n V}: \Sym^n V\otimes \Sym^n V^*\xrightarrow{\id\otimes n!} (V^{\otimes n})\otimes (V^{*\otimes n})\isom(V^{\otimes n})\otimes (V^{\otimes n})^*\xrightarrow{\ev_{V^{\otimes n}}} k.\end{equation} There is a similar construction for $\bigwedge^n V^*\isom (\bigwedge^n V)^*$. One should pay attention to the factor $n!$ here.

For any $V\in\Vect^{\he}$, we have a canonical symplectic pairing $\om_{\can}$ on $V\oplus V^*$ defined by 
\begin{equation}\label{StdSymp}
    \om_{\can}:(V\oplus V^*)\otimes (V\oplus V^*) \to V\otimes V^*\oplus V^*\otimes V \xrightarrow{\ev - \ev\circ\sw} k.
\end{equation}

We introduce the temporary notation $k_i:=\Pi^i k\in\Vect$.

For each pure element $v^*\in V^*$, we sometimes identify it with the pure coelement $V\to k_{|v^*|}$ which is the composition \[v^*:V\xrightarrow{\otimes v^*} V\otimes \Pi^{|v^*|}V^*\xrightarrow{\ev} k_{|v^*|}.\] Conversely, given a pure coelement $v^*:V\to k_{|v^*|}$, we get an element $v^*\in V^*$ via \[v^*:k_{|v^*|}\xrightarrow{\coev} V^*\otimes \Pi^{|v^*|}V\xrightarrow{\id\otimes v^*} V^*.\] It is easy to check that these two operations are inverse to each other.

For a bilinear form which is a linear map $b:V\otimes W\to k$, we sometimes identify it with the linear map \begin{equation}\label{pairingchange}b:W\xrightarrow{\coev\otimes\id} V^*\otimes V\otimes W\xrightarrow{\id\otimes b} V^*.\end{equation} Conversely, suppose we are given a linear map $b:W\to V^*$, we obtain a bilinear form via \[b:V\otimes W\xrightarrow{\id\otimes b} V\otimes V^*\xrightarrow{\ev} k.\] It is easy to check that these two operations are inverse to each other.

From a non-degenerate bilinear form $b:V\otimes W\to k$, we can also identify it with an element $\coev_b:k\to W^*\otimes V^*$ by \begin{equation}\label{coevb} \coev_b:k\xrightarrow{\coev_W} W^*\otimes W\xrightarrow{\id\otimes b} W^*\otimes V^*,\end{equation} which is also the dual map of $b$.

Given two bilinear forms $b_V:V^{\otimes 2}\to k$ and $b_W:W^{\otimes 2}\to k$, we get a bilinear form  $b_{V\otimes W}:(V\otimes W)^{\otimes 2}\to k$ via \begin{equation}\label{tensorb} b_{V\otimes W}:(V\otimes W)\otimes(V\otimes W)\xrightarrow{\id\otimes\sw\otimes\id}V\otimes V\otimes W\otimes W\xrightarrow{b_V\otimes b_W} k\otimes k\isom k.\end{equation} 

We also get bilinear forms \begin{equation}\label{symb} b_{\Sym^n V}:(\Sym^n V)^{\otimes 2}\to k\end{equation} and $b_{\bigwedge^n V}:(\bigwedge^n V)^{\otimes 2}\ \to k$ by restricting $n!\cdot b_{V^{\otimes n}}$ to the corresponding subspace. One should pay attention to the factor $n!$ here.

When $V$ is an odd vector space, in concrete terms, we have 
\begin{equation}\label{tensorbformula}
    b_{V^{\otimes n}}(v_1\cdots v_n,v'_1\cdots v'_n)=(-1)^{n(n-1)/2}b_V(v_1,v_{1}')\cdots b_V(v_n,v_{n}')
\end{equation} and
\begin{equation}\label{symbformula}
    b_{\Sym^n V}(v_1\cdots v_n,v'_1\cdots v'_n)=\sum_{g\in S_n}(-1)^{n(n-1)/2}\sgn(g)b_V(v_1,v_{g(1)}')\cdots b_V(v_n,v_{g(n)}')
\end{equation} for $v_1,\cdots,v_n,v_1',\cdots,v_n'\in V$.

Given a bilinear form $b:V\otimes W\to k$, there is a canonical symplectic form $\om_{\can}:(V\oplus W)^{\otimes 2}\to k$ defined by 
\begin{equation}\label{symcanb}
    \om_{\can}(v,w)=b(v,w), \om_{\can}(w,v)=(-1)^{1+|v||w|}b(w,v), \om_{\can}(V,V)=\om_{\can}(W,W)=0,
\end{equation}
for $v\in V$ and $w\in W$.

\subsection{Lie theory}\label{lie}
We use $G$ to denote a connected split reductive group over $\F_q$. We use $T\sub B$ to denote the maximal torus and Borel subgroup. We get a sequence of Lie algebras $\frt\sub\frb\sub\frg$. We use $X_*(T), X^*(T)$ to denote the cocharacter lattice and character lattice, and $X_*(T)^+$, $X^*(T)^+$ to denote the subset of dominant elements. We use $W_G$ to denote the Weyl group of $G$. We use $\r_G\in X^*(T)$ to denote the half sum of all positive roots of $G$.

For each $\l\in X_*(T)$, we use $\l^+$ to denote the unique element in $W_G\l\cap X_*(T)^+$ and similarly for elements in $X^*(T)$. For $\l\in X_*(T)$, we denote $d_{G,\l}:=\langle 2\r_G,\l\rangle$. We use $d\l\in \frt$ to denote the differential of $\l$.

We sometimes use $H\sub G$ to denote a connected split reductive subgroup. We use $T_H$ to denote a maximal torus of $H$ such that $T_H\sub T$. We regard $X_*(T_H)\sub X_*(T)$. For each $\l\in X_*(T)$, we always use $\l^+$ to denote the $G$-dominant $W_G$-conjugate of $\l$ (rather than the $H$-dominant conjugate).

We use $\Gc$ to denote the Langlands dual of $G$ defined over $k$.

When we talk about representations of $G$, we always mean \emph{right} $G$-modules. When we talk about representations of $\Gc$, we always mean \emph{left} $\Gc$-modules.

\subsection{Varieties with group action}\label{Gvar}
We always use $X$ to denote a variety with \emph{right} $G$-action. We are mostly interested in the case that $T^*X$ is a \emph{hyperspherical} $G$-Hamiltonian space. For its complete definition and properties, we refer to \cite[\S3]{BZSV}. We merely summarize some points that we will use in our article.

In this article, we only care about polarized hyperspherical $G$-varieties on the automorphic side (and denote it by $T^*X$). When we talk about a hyperspherical $G$-variety $T^*X$, we mean an affine smooth $G\times \Ggr$-variety $X$ such that the $G\times\Ggr$-variety $T^*X$ is hyperspherical in the sense of \cite[\S3.5.1]{BZSV} (See below for the $\Ggr$-action on $T^*X$). Here $\Ggr$ is a copy of $\Gm$ and the action of $G\times\Ggr$ is on the right.

Over an algebraically closed field in characteristic zero, the hypersphericity of $T^*X$ implies that $X$ is spherical \cite[Proposition\,3.7.4]{BZSV} such that the generic $B$-stabilizer is connected, and can be written as $X=Y\times^HG$ for some connected reductive subgroup $H\sub G$ and an $H$-representation $Y$. The requirement that $T^*X$ is \emph{neutral} for $T^*X$ to be hyperspherical in \textit{loc.cit} implies that the $\Ggr$-action on $X$ is trivial on $H\bs G$ and of weight one on the fibers of $X\to H\backslash G$ (i.e scaling action on $Y$). Taking this normalization, the $\Ggr$-action on $T^*X=(T^*Y\oplus (\frg/\frh)^*)\times^H G$ is given by weight one action on $T^*Y$ and weight two action on $(\frg/\frh)^*$, trivial on the base $H\bs G$. In particular, if $X$ is $G$-homogeneous, then the $\Ggr$-action on $X$ is trivial.

When we talk about hyperspherical $T^*X$, as in \cite[\S3.8]{BZSV}, we always assume there exists and fix a $G$-eigenmeasure $\om$ on $X$ with eigencharacter $\eta:G\to\Gm$. Such an eigenmeasure is automatically an eigenmeasure for $G\times\Ggr$ such that \begin{equation}\label{gammax}
    (g,\l)^*\om=\eta(g)\l^{\gamma_X}\om
\end{equation}
for $g\in G$ and $\l\in\Ggr$. Here $\gamma_X\in \ZZ$ is an invariant of the variety $X$ (independent of $\om$) and one can easily computes that \begin{equation}\label{gammavalue}
    \gamma_X=\dim(Y).
\end{equation}
We define \begin{equation}\label{betavalue}
    \b_X:=\dim(G)+\gamma_X-\dim(X)=\dim(H).
\end{equation}

\subsection{Hyperspherical duality}\label{hdual}
When $T^*X$ is hyperspherical, we use $\Mc$ to denote the dual $\Gc$-hyperspherical variety for $T^*X$ constructed in \cite[\S3]{BZSV}. We use $\Gc_X$ to denote the dual group of $X$. Let $P(X)$ be the stabilizer of the open $B$-orbit in $X$.
We say $T^*X$ is \emph{strongly tempered} if $\Mc$ is a $\Gc$-linear representation (or equivalently, the dual group $\Gc_X=\Gc$ and $P(X)=B$). We say that $T^*X$ is \emph{tempered} if $\Mc$ is of the form $\Mc=(S\oplus (\frgc/\frgc_X)^*)\times^{\Gc_X}\Gc$ (or equivalently, the parabolic $P(X)=B$) where $S$ is a symplectic linear $\Gc_X$-representation with weight one $\Ggr$-action.

\subsection{Intersection complex}\label{icsheaf}
On a variety $X$, we use $\IC_X$ to denote its intersection complex, which is normalized in the following way: For a smooth variety $X$, we have $\IC_X:=\uk_X\langle \dim(X)\rangle$. That is to say: the complex $\IC_X$ is perverse and has the same parity as $\dim(X)$ and is pure of weight zero whenever weight makes sense.

\subsection{Super geometric Langlands}\label{supergeometriclanglands}
To normalize the (super) Geometric Langlands equivalence \begin{equation}\label{glc}
    \Shv_{\Nilp}(\Bun_G)\cong \IndCoh_{\Nilp}(\Loc_{\Gc}^{\res}),
\end{equation} we fix \begin{equation}\label{halfspin}
    \Om^{1/2}\in\Pic(C)
\end{equation} and an isomorphism $(\Om^{1/2})^{\otimes 2}\isom \Om$. We also need to fix $\psi:\F_q\to k^{\times}$, which gives us an Artin-Schreier sheaf \begin{equation}\label{assheaf}\AS\in\Shv(\A^1)\end{equation} lying in (naive) cohomological degree 0 and corresponds to $\psi$ under the sheaf-function correspondence.              

We define the normalized Whittaker sheaf $\cP_{\Whit}^{\norm}$ as follows. Consider diagram \begin{equation}\label{bunndiagram}
    \begin{tikzcd}
        \Bun_{N,\rc(\Om)} \ar[r, "q"] \ar[d, "\pi"]  & \A^1 \\
        \Bun_{G} & 
    \end{tikzcd}
\end{equation} 
where $\Bun_{N,\rc(\Om)}$ is defined via the Cartesian diagram
\begin{equation}
\begin{tikzcd}
    \Bun_{N,\rc(\Om)} \ar[r]\ar[d] & \Bun_B \ar[d] \\
    \pt \ar[r, "\rc(\Om)"] & \Bun_T\rcart.
    \end{tikzcd}
\end{equation}
We define
\begin{equation}\label{whitsheaf}
    \cP_{\Whit}^{\norm}:=\pi_!q^*\AS\langle \dim(\Bun_{N,\rc(\om)}) \rangle
\end{equation}
where $\dim(\Bun_{N,\rc(\Om)})=(g-1)(\dim(N)-\langle2\r_G,2\rc_{G}\rangle)$.

We adopt the normalization of the (conjectural) Geometric Langlands equivalence in \cite[\S12.2.1]{BZSV} such that under the equivalence \eqref{glc}, the object $\mathrm{P}_{\Nilp}*\cP_{\Whit}^{\norm}\in \Shv_{\Nilp}(\Bun_G)$ corresponds to the (shifted) dualizing sheaf $\om_{\Loc_{\Gc}^{\res}}\langle -(g-1)\dim(G) \rangle\in\QCoh(\Loc_{\Gc}^{\res})\sub\IndCoh(\Loc_{\Gc}^{\res})$, where $\mathrm{P}_{\Nilp}*:\Shv(\Bun_G)\to\Shv_{\Nilp}(\Bun_G)$ is the Beilinson's spectral projector defined in \cite[\S13.4.1]{arinkin2022stacklocalsystemsrestricted}, which is the \emph{right} adjoint of the natural inclusion $\Shv_{\Nilp}(\Bun_G)\to\Shv(\Bun_G)$.

\subsection{Geometric class field theory}\label{gcft}
Consider the map \begin{equation}\AJ_d:C^d\to\Pic^d\end{equation} defined by \begin{equation} \AJ_d(c_1,\dots,c_d)=\cO(c_1+\cdots+c_d)\end{equation} and the map \begin{equation} \lh:C\times \Pic\to\Pic\end{equation} such that \begin{equation}\lh(c,\cL)=\cL(c).\end{equation} For each $\s\in\Loc_{\Gm}^{\res}(k)$, we define $\LL_{\s}$ to be the unique perverse sheaf on $\Pic$ such that \begin{equation}
    \AJ^*_d\LL_{\s}\isom \s^{\boxtimes d}\langle g-1\rangle
\end{equation} and \begin{equation}
    \lh^*\LL_{\s}\isom \s\boxtimes \LL_{\s}.
\end{equation}

\subsubsection{Unramified classical class field theory}\label{cft}
For each $\s\in\Loc_{\mathbb{G}_m}^{\arith}(k)$, we define the associated Hecke character to be 
\begin{equation}
    \chi_{\s}:=\tr(\Frob, \LL_{\s}\langle-(g-1)\rangle))\in\Hom(\Pic(\F_q),k^{\times})
\end{equation}

\subsection{Cohomology of classifying stacks}\label{cohbg}
We choose $\cO(-1)\in\Pic(B\G_m)$ to be the tautological line bundle on $B\Gm$ such that for any line bundle $\cL\in\Pic(X)$ defining a map $f_{\cL}:X\to B\Gm$, we have $f_{\cL}^*\cO(-1)=\cL$. This choice gives us an isomorphism $\Gamma(B\Gm,\uk)\isom k[\hbar]$ where $\hbar=c_1(\cO(-1))$.

\subsection{Shearing}\label{shear}
For a vector space $V$ equipped with a $\Ggr=\Gm$-action, we get a weight decomposition $V=\bigoplus_{i\in\ZZ} V_i$. We define $V^{\shear}:=\bigoplus_{i\in\ZZ} V_i\langle i\rangle$.

\section{Linear algebra: Kolyvagin system}\label{linearalgebra}
In this section, we introduce the linear algebra machinery used to study higher derivatives of $L$-functions of a symplectic local system over the curve $C$.
\begin{itemize}
    \item In \autoref{clifalgsection}, we recall the Clifford algebra and fix conventions.
    \item In \autoref{pinstructuresection}, we introduce the Pin structure, which serves as a preparation for defining the Kolyvagin system.
    \item  In \autoref{koly}, we introduce the Kolyvagin system associated to a Pin structure, and prove the key result \thref{kolylfun}.
    \item   In \autoref{modulekoly}, we introduce and study the Kolyvagin system associated to a module of the Clifford algebra, which will be the Kolyvagin system occurring in \thref{higherformulaglnfirst}, \thref{higherformulageneral}.
\end{itemize}

Throughout the section, we work over $k=\Qlbar$ or a finite extension of $\Ql$ such that all the groups and polynomials split. Assume we are in the following setting:
\begin{setting}\thlabel{symp}
We fix $(M, \om)$ to be an odd symplectic vector space of dimension $(0|2n)$ for $n\in\Z_{\geq 0}$, where the symplecticity of the pairing means that it is non-degenerate and for any $m_1,m_2\in M$ we have \[\om(m_1,m_2)=\om(m_2,m_1).\] 
\end{setting}

For our conventions on super vector spaces, see \autoref{superlinear}.

\subsection{Conventions on the Clifford algebra}\label{clifalgsection}
We start by recalling some basic notions on the Clifford algebra.

\subsubsection{Basic structures}\label{linearalgebrasetup}
\begin{defn}\thlabel{clif}
    We define the \emph{Clifford algebra} associated to $(M,\om)$ to be \[\Cl(M):=\Cl(M,\om):=M^{\otimes}/I\] where $M^{\otimes}$ is the free tensor algebra on $M$ and the two-sided ideal $I\sub M^{\otimes}$ is generated by elements of the form \[m_1\otimes m_2+m_2\otimes m_1-\om(m_1,m_2)\] for all $m_1,m_2\in M$.\footnote{It is more common to take $I$ to be generated by $m_1\otimes m_2+m_2\otimes m_1-2\om(m_1,m_2)$. We discard the factor $2$ to get a cleaner formula in \thref{kolylfun}. Moreover, the Clifford algebras we are going to consider (as in \thref{eigenkolyvagin}) are closer to our convention.} 
 \end{defn}

Since $\Cl(M)\in\Vect$, we have its parity decomposition $\Cl(M)=\Cl(M)^{\even}\oplus \Cl(M)^{\odd}$.

There is an increasing filtration $\{G_i\Cl(M)\}_{i\in\Z_{\geq 0}}$ defined by \begin{equation} \label{cliffil}G_i\Cl(M):=\sum_{0\leq j\leq i}\im(M^{\otimes j}\to\Cl(M)).\end{equation} 
The graded pieces with respect to the filtration $G_{\bullet}$ is $\Gr^G_{\bullet} \Cl(M)\isom \Sym^{\bullet}M$. In particular, the top graded piece $\Gr^G_{\top} \Cl(M)=\Sym^{\top}M$ is an one-dimensional even vector space. From $\om:M^{\otimes 2}\to k$, we get an induced orthogonal pairing $\om_{\top}:(\Sym^{\top} M)^{\otimes 2}\to k$ as in \eqref{symb}.

\begin{defn}\thlabel{parity}
    A \emph{parity} of $(M,\om)$ is a linear map $\om_{\top}^{1/2}:\Sym^{\top}M\to k$ such that \[(\om_{\top}^{1/2})^{\otimes 2}=\om_{\top}.\]
\end{defn}

\begin{remark}
    The name parity is justified by \thref{whyparity}.

\end{remark}

A choice of parity $\om^{1/2}_{\top}$ gives a linear map \[\tr:\Cl(M)\to k\] defined as the composition \begin{equation}\label{tr}\tr:\Cl(M)\to \Gr^G_{\top} \Cl(M)\isom \Sym^{\top}M\xrightarrow{\om^{1/2}_{\top}} k.\end{equation}

Note that for each $M$ as above, there are precisely two choices of parity $\om_{\top}^{1/2}$ differing by a sign.

\subsubsection{Pin groups}
We use $\Cl(M)^{\times}$ to denote the group of invertible elements in $\Cl(M)$. There are several important subgroups  \begin{equation}\label{pingroup}
    \begin{tikzcd}
        \Spin(M) \ar[r,hook] \ar[d,hook] & \Pin(M) \ar[d, hook] & \\
        \GSpin(M) \ar[r, hook] & \GPin(M) \ar[r,hook] & \Cl(M)^{\times}
    \end{tikzcd}.
\end{equation}

The group $\GPin(M)$ is defined as \[\GPin(M)=\{x\in \Cl(M)^{\times}|xMx^{-1}\sub M\},\] and the group $\GSpin(M)$ is defined by \[\GSpin(M):=\GPin(M)\cap \Cl(M)^{\even}.\]

The group $\Pin(M)$ can be defined as \[\Pin(M):=\{v_1v_2\cdots v_k\in\Cl(M)^{\times}|v_1,v_2,\dots,v_k\in M~ \textup{such that}~ \om(v_i,v_i)=2\}\] and the group $\Spin(M)$ is defined by \[\Spin(M):=\Pin(M)\cap \Cl(M)^{\even}.\]

One can easily check that $\Pin(M)$ and $\Spin(M)$ are subgroups by noting that \[v^{-1}=\frac{2v}{\om(v,v)}\] for every $v\in M\sub\Cl(M)$ such that $\om(v,v)\neq 0$. Moreover, it is well-known that one has
\[\GPin(M) = \Pin(M)\cdot k^{\times}\] and \[\GSpin(M) = \Spin(M)\cdot k^{\times}.\]

We have a natural group homomorphism \begin{equation}\label{pindet} \det:\GPin(M)\to\{\pm 1\}\end{equation} defined by \[\det(g)=(-1)^{|g|},\]where $|g|$ is the parity of $g\in\Cl(M)$.\footnote{Every element in $\GPin(M)$ is pure.} Moreover, there is a canonical map \begin{equation}\label{pincover}
    R:\GPin(M)\to \Or(M)
\end{equation}
\[g\mapsto R_g\]
defined as 
\begin{equation}
    R_g(m) = \det(g)gmg^{-1}\in M\sub \Cl(M) ~\textup{for $g\in\GPin(M)$ and $m\in M\sub\Cl(M)$}.
\end{equation}
We also have a map \begin{equation}\label{spinornorm}
    \l:\GPin(M)\to \GPin(M)/\Pin(M)\isom k^{\times}
\end{equation}
which satisfies $\l(cg)=c^2$ for $c\in k^{\times}\sub\GPin(M)$ and $g\in\Pin(M)$.

For each element $v\in M$ satisfying $\om(v,v)\neq 0$, it is easy to see that for every $m\in M$ we have
\begin{equation}
R_v(m)=m-\frac{2\om(m,v)}{\om(v,v)}v
\end{equation} which coincides with the reflection along $v$.

\subsubsection{Clifford module}
It is well-known that there is a unique up-to-parity (super)representation $S$ of $\Cl(M)$ such that the representation map $\Cl(M)\to \End(S)$ is an isomorphism. The representation $S$ is called the \emph{Clifford module}. We now construct a specific model of the Clifford module.
\begin{const}\thlabel{cliffordmodule}
     Choose a polarization $M=L\oplus L^*$ such that $\om=\om_{\can}$ as defined in \eqref{StdSymp}.
\begin{itemize}
    \item Define $S=\Sym^{\bullet} L^*$.
    \item The action of $L^*\sub\Cl(M)$ on $S$ is given by the linear map $L^*\otimes \Sym^{\bullet} L^*\to \Sym^{\bullet} L^*$ defined as multiplication on the left;
    \item The action of $L\sub\Cl(M)$ on $S$ is given by the linear map $L\otimes \Sym^{\bullet} L^*\to \Sym^{\bullet} L^*$ defined via \[v\otimes v_1^*v_2^*\cdots v_n^*\mapsto  \sum_{i=1}^n (-1)^{i+1}\om(v,v_i^*) v_1^*\cdots \widehat{v_i^*}\cdots v_n^*.\]
    \item The actions above uniquely extends to an action of $\Cl(M)$ on $S$.
\end{itemize}
\end{const}

Note that the polarization $M=L\oplus L^*$ induces a parity $\om^{1/2}_{\top,L}$ given by \[\om^{1/2}_{\top,L}:\Sym^{\top}M\isom \Sym^{\top}L\otimes \Sym^{\top} L^*\isom \Sym^{\top}L\otimes (\Sym^{\top}L)^*\xrightarrow{\ev} k\] where the second isomorphism is defined in \eqref{symev}.

An easy computation shows the following:
\begin{prop}\thlabel{tr=cliffordmodule}
    The map $\tr:\Cl(M)\to k$ induced by $\om^{1/2}_{\top,L}$ as in \eqref{tr} coincides with the map $\Cl(M)\to\End(S)\xrightarrow{\tr} k$, where the (super)trace $\tr:\End(S)\to k$ is defined in \eqref{tracedef}.
\end{prop}

\begin{remark}\thlabel{whyparity}
    From \thref{tr=cliffordmodule}, one sees that a choice of parity is equivalent to a choice of parity of the Clifford module.
\end{remark}

We also note the following obvious property of $\tr$:
\begin{prop}
    For pure elements $x,y\in \Cl(M)$, we have $\tr(xy)=(-1)^{|x||y|}\tr(yx)$. 
\end{prop}

\subsection{Pin structures}\label{pinstructuresection}
From now on, we assume we are in the following setting:
\begin{setting}\thlabel{Fsymp}
    Let $(M,\om)$ be as in \thref{symp}. Moreover, we fix $F\in\Or(M)$ and get a triple $(M,\om,F)$.
\end{setting}

\begin{defn}\thlabel{pinstructure}
    A \emph{Pin structure} (resp. \emph{GPin structure}) of $(M,\om,F)$ is an element $\tilF\in\Pin(M)$ (resp. $\tilF\in\GPin(M)$ such that \[R_{\tilF}=F,\] where $R:\GPin(M)\to \Or(M)$ is defined in \eqref{pincover}.
\end{defn}

Note that the lifting $\tilF$ in a Pin structure (resp. GPin structure) is unique up to $\pm 1$ (resp. $k^{\times}$). We use \begin{equation}\label{kolyspace}\Koly_F\sub\Cl(M)\end{equation} to denote the one-dimensional subspace spanned by all GPin structures of $(M,\om, F)$. We make the following definition
\begin{defn}\thlabel{kolystructure}
    A \emph{Kolyvagin structure} of $(M,\om,F)$ is an element $\tilF\in\Koly_F$.
\end{defn}

Note that the map $\l:\Koly_F\cap \GSpin(M)\to k$ defined in \eqref{spinornorm} can be extended to a map \begin{equation}\label{kolynorm}
    \l:\Koly_F\to k
\end{equation}
by setting $\l(0)=0$. The map $\l$ is quadratic.

The name Kolyvagin will be justified in  \autoref{koly}.

\subsubsection{Semisimple case}\label{clifexp}
We are now going to give explicit descriptions of the Pin structure $\tilF$ in \thref{pinstructure} when $F\in\Or(M)$ is semisimple. This section is only used in the proof of \thref{kolylfun} and \thref{kolynormstandard}. It can be safely skipped during a first reading.

We first consider the case $\det(F)=1$. In this case, there exists a polarization $M=L\oplus L^*$ such that $\om=\om_{\can}$ as in \eqref{StdSymp} and both $L$ and $L^*$ are $F$-stable. We choose $F$-eigenbasis $(e_1,e_2,\cdots,e_n)$ of $L$ and get dual basis $(e_1^*,e_2^*,\cdots,e_n^*)$ of $L^*$ satisfying $\om(e_i,e_j^*)=\d_{ij}$. Suppose we have $Fe_i=\a_ie_i$ for each $i$, we get $Fe_i^*=\a_i^{-1}e_i^*$. Note the following description of $\tilF$:

\begin{prop}\thlabel{despin0}
    The Pin structure $\tilF$ is given by \begin{equation}\label{pineven}\tilF=\pm \prod_{i=1}^n(\a_i^{1/2}e_ie_i^*+\a_i^{-1/2}e_i^*e_i)\in\Spin(M).\end{equation}
\end{prop}

\begin{proof}
    Note that \[\a_i^{1/2}e_ie_i^*+\a_i^{-1/2}e_i^*e_i=(e_i+e_i^*)\cdot(\a_i^{-1/2}e_i+\a_i^{1/2}e_i^*)\] and we have \[\om(e_i+e_i^*,e_i+e_i^*)=\om(\a_i^{-1/2}e_i+\a_i^{1/2}e_i^*,\a_i^{-1/2}e_i+\a_i^{1/2}e_i^*)=2,\] therefore we know that right hand side of \eqref{pineven} lies in $\Pin(M)$. To prove the identity, it suffices to check that $R_{e_i+e_i^*}\circ R_{\a_i^{-1/2}e_i+\a_i^{1/2}e_i^*}$ acts by identity on vectors $e_j,e_j^*$ for $j\neq i$ and has eigenvalue $\a_i$ on $e_i$ and $\a_i^{-1}$ on $e_i^*$. The first claim is obvious. The second claim follows from the following direct computation:
    \[\begin{split}
        R_{e_i+e_i^*}\circ R_{\a_i^{-1/2}e_i+\a_i^{1/2}e_i^*}(e_i)&= R_{e_i+e_i^*}(e_i-\om(e_i,\a_i^{-1/2}e_i+\a_i^{1/2}e_i^*)(\a_i^{-1/2}e_i+\a_i^{1/2}e_i^*))\\
        &=R_{e_i+e_i^*}(-\a_ie_i^*)\\
        &=-\a_ie_i^*+\om(\a_ie_i^*,e_i+e_i^*)(e_i+e_i^*)\\
        &=\a_ie_i
    \end{split}.\]
The computation for $e_i^*$ is similar.

\end{proof}
Let us write $F_L=F|_L$, $F_{L^*}=F|_{L^*}$ and $F_{\Sym^{\bullet}L^*}$, $F_{\Sym^{\bullet}L}$ to be the automorphisms on $\Sym^{\bullet}L^*$, $\Sym^{\bullet}L$. We have the following immediate observations:

\begin{cor}
\thlabel{actpin0}
    Let us take \[\tilF=\prod_{i=1}^n(\a_i^{1/2}e_ie_i^*+\a_i^{-1/2}e_i^*e_i)\in\Spin(M),\] then the action of $\tilF$ on $S=\Sym^{\bullet}L^*$ is given by \[\tilF\cdot e_{i_1}^*\cdots e_{i_k}^*=\e(L,F)^{1/2}\a_{i_1}^{-1}\cdots\a_{i_k}^{-1}e_{i_1}^*\cdots e_{i_k}^*,\] where $\e(L,F):=\det(F_L)^{-1}=\a_1\cdots\a_n$ is the root number. In other words, the action of $\tilF$ on $S=\Sym^{\bullet}L^*$ equals $\e(L,F)^{1/2}F_{\Sym^{\bullet}L^*}$.
\end{cor}
\begin{remark}
    The above choice of $\tilF$ depends on a choice of square root $\e(L, F)^{1/2}$ (rather than choices of square roots of individual $\a_i$). The last claim in \thref{actpin0} is also valid when $F$ is not semisimple but admits the $F$-stable polarization $M=L\oplus L^*$.
\end{remark}

Now we turn to the case $\det(F)=-1$. In this case, we always have $F$-stable decomposition $M=W\oplus W^*\oplus k\cdot x\oplus k\cdot y$ such that $W\oplus W^*$, $k\cdot x$, $k\cdot y$
are orthogonal to each other (recall that we are assuming $F$ semisimple), in which $\om|_{W\oplus W^*}$ is the canonical symplectic pairing, and $\om(x,x)=\om(y,y)=2$, $Fx=x$, $Fy=-y$. Let us set $a=(x+y\sqrt{-1})/2$ and $a^*=(x-y\sqrt{-1})/2$. Take $L^*=W^*\oplus k\cdot a^*$ and $L=W\oplus k\cdot a$. This gives a polarization $W=L\oplus L^*$, hence we can use $S=\Sym^{\bullet}L^*$ as the model for the Clifford module. We further pick $F$-eigenbasis $e_1,\cdots,e_{n-1}$ of $W$ as before such that $Fe_i=\a_i e_i$, and obtain dual basis $e_1^*,\cdots,e_{n-1}^*$. We have the following proposition parallel to \thref{despin0}:

\begin{prop}
    \thlabel{despin1}
    We have \[\tilF=\pm (\prod_{i=1}^{n-1}(\a_i^{1/2}e_ie_i^*+\a_i^{-1/2}e_i^*e_i))y\in\Pin(M).\]
\end{prop}
The proof is completely analogous to the proof of \thref{despin0}, so we omit it. We also have the following Corollary parallel to \thref{actpin0}:

\begin{cor}
    \thlabel{actpin1}
    Let us take \[\tilF=(\prod_{i=1}^{n-1}(\a_i^{1/2}e_ie_i^*+\a_i^{-1/2}e_i^*e_i))y\in\Pin(M),\] then the action of $\tilF$ on $S=\Sym^{\bullet}L^*$ is given by \[\tilF\cdot e_{i_1}^*\cdots e_{i_k}^*=(-1)^{k+1/2}\e(W,F)^{1/2}\a_{i_1}^{-1}\cdots\a_{i_k}^{-1}e_{i_1}^*\cdots e_{i_k}^*a^*\] and \[\tilF\cdot e_{i_1}^*\cdots e_{i_k}^*a^*=(-1)^{k-1/2}\e(W,F)^{1/2}\a_{i_1}^{-1}\cdots\a_{i_k}^{-1}e_{i_1}^*\cdots e_{i_k}^*.\]
\end{cor}

\subsection{Kolyvagin systems}\label{koly}

In this section, we introduce a sequence of elements attached to $(M,\om, F)$ after choosing a Kolyvagin structure (see \thref{kolystructure}) and parity (see \thref{parity}). This sequence of elements was first considered in \cite{YZ3}. In comparison with the original definition, we provide a new construction of these elements that is considerably simpler and more conceptual. Since this sequence of elements shares some features with the Kolyvagin system over number fields defined in \cite{mazur2004kolyvagin}, it can be regarded as a function field analog of the Kolyvagin system.

\subsubsection{Definition}
\begin{defn}\thlabel{kolypolarization}
    Given $(M,\om,F)$ as in \thref{Fsymp}, a \emph{Kolyvagin polarization} of $(M,\om,F)$ is a quadruple \[(M,\om,\tilF,\om_{\top}^{1/2})\] where
    \begin{itemize}
        \item The element $\tilF\in \Koly_F$ is a Kolyvagin structure for $(M,\om,F)$ defined in \thref{kolystructure};
        \item The element $\om_{\top}^{1/2}$ is a parity for $(M,\om)$ defined in \thref{parity}.
    \end{itemize}
    We say a Kolyvagin polarization is GPin (or Pin, Spin, GSpin) if $\tilF\in\GPin(M)$ (or $\tilF$ lies in $\Pin(M)$, $\Spin(M)$, $\GSpin(M)$).
\end{defn}

\begin{defn}\thlabel{kolydefn}
We define the \emph{Kolyvagin system} attached to a Kolyvagin polarization $(M,\om,\tilF,\om_{\top}^{1/2})$ to be the sequence of elements \[\{z_{r}\in M^{*\otimes r}\}_{r\in\Z_{\geq 0}}\] where the element $z_r$ is defined as the coelement \[z_r:M^{\otimes r}\to k\] such that  
\begin{equation}\label{kolycons}
    z_r(m_1\otimes m_2\otimes\cdots\otimes m_r)=\tr( m_1m_2\cdots m_r\tilF)
\end{equation}
for $m_1,\dots,m_r\in M$, and the map $\tr$ is defined in \eqref{tr}.
    We say a Kolyvagin system is GPin (or Pin, Spin, GSpin) if the Kolyvagin polarization is GPin (or Pin, Spin, GSpin).
\end{defn}

\begin{remark}\thlabel{tensordualconvention}
    We remind the reader that by the convention in \autoref{superlinear}, for $m_r^*\otimes\cdots\otimes m_1^*\in M^{*\otimes r}$, the corresponding coelement of $m_r^*\otimes\cdots\otimes m_1^*:M^{\otimes r}\to k$ is \begin{equation}\label{tensordualformula}
        (m_r^*\otimes\cdots\otimes m_1^*)(m_1\otimes\cdots\otimes m_r)=m_1^*(m_1)\cdots m_r^*(m_r).
    \end{equation} for $m_1\otimes\cdots\otimes m_r\in M^{\otimes r}$. 
\end{remark}

\begin{remark}
    We sometimes extend the Kolyvagin system $\{z_{r}\in M^{*\otimes r}\}_{r\in\Z_{\geq 0}}$ to a sequence $\{z_{r}\in M^{*\otimes r}\}_{r\in\Z}$ by setting $z_{r}=0$ for $r<0$.
\end{remark}

\subsubsection{Characterizing properties}\label{kolycharchapter}
In this section, we describe a characterizing property of the Kolyvagin system $\{z_r\}$, which served as the original definition in \cite{YZ3}. While this perspective will not be used elsewhere in the present article, we include it here for completeness.

Consider the following morphisms:
\begin{itemize}
    \item For each $1\leq i\leq r+1$, we have a map $\coev_{i,i+1}:M^{*\otimes r}\to M^{*\otimes r+2}$ defined as the composition \[\coev_{i,i+1}:M^{*\otimes r}\isom M^{*\otimes i-1}\otimes k\otimes M^{*\otimes r-i+1}\xrightarrow{\id\otimes \coev_{\om}\otimes\id}M^{*\otimes r+2}\] where the map $\coev_{\om}:k\to M^*\otimes M^*$ is introduced in \eqref{coevb}.
    \item For each $1\leq i\leq r-1$, we have a map $\sw_{i,i+1}:M^{*\otimes r}\to M^{*\otimes r}$ introduced in  \eqref{swap}.
    \item We have a morphism $\pF:M^{*\otimes r}\to M^{*\otimes r}$ defined by \[\pF(m_1^*\otimes m_2^*\otimes\cdots\otimes m_r^*)=(-1)^{r-1}F(m_r^*)\otimes m_1^*\otimes\cdots\otimes m_{r-1}^*\] where $F=(F^*)^{-1}:M^*\to M^*$.
\end{itemize}

\begin{prop}\thlabel{kolychar}
    For the Kolyvagin system $\{z_{r}\in M^{*\otimes r}\}_{r\in\Z_{\geq 0}}$, we have the following facts:
    \begin{enumerate}
    \item \label{kolyiter}
            For each $r\in \Z_{\geq 0}$ and $1\leq i\leq r-1$, we have
            \begin{equation*}
                z_r-\sw_{i,i+1}(z_r)=\coev_{i,i+1}(z_{r-2}).
            \end{equation*}
    \item \label{kolypf} For each $r\in \Z_{\geq 0}$, we have
            \begin{equation*}
                \pF(z_r)=z_r.
            \end{equation*}
    \item \label{kolyuni} The properties in \eqref{kolyiter} and \eqref{kolypf} uniquely characterizes $\{z_{r}\in M^{*\otimes r}\}_{r\in\Z_{\geq 0}}$ up to scalar. Equivalently speaking, after fixing the parity in \thref{parity}, the construction \eqref{kolycons} defines an isomorphism between the one-dimensional vector space $\Koly_F$ defined in \thref{kolystructure} and the vector space of sequences of elements $\{z_{r}\in M^{*\otimes r}\}_{r\in\Z_{\geq 0}}$ satisfying \eqref{kolyiter} and \eqref{kolypf}.
    \end{enumerate}
\end{prop}

\begin{proof}
    We first prove the following lemma:
    \begin{lemma}\thlabel{dualclif}
        A sequence $\{z_{r}\in M^{*\otimes r}\}_{r\in\Z_{\geq 0}}$ satisfies \itemref{kolychar}{kolyiter} if and only if there exists some element $T\in \Cl(M)$ such that the sequence $\{z_{r}\in M^{*\otimes r}\}_{r\in\Z_{\geq 0}}$ satisfies \begin{equation}\label{kolycons1}
    z_r(m_1\otimes m_2\otimes\cdots\otimes m_r)=\tr( m_1m_2\cdots m_rT).
\end{equation} for any $r\in\Z_{\geq 0}$ and $m_1,\dots,m_r\in M$.
    \end{lemma}
\begin{proof}[Proof of lemma]
Regard the sequence $\{z_{r}\in M^{*\otimes r}\}_{r\in\Z_{\geq 0}}$ as a single element  $z_{\bullet}\in \prod_{r\in \Z_{\geq 0}} M^{*\otimes r}\isom (M^{\otimes r})^*$. The condition in \itemref{kolychar}{kolyiter} is equivalent to say that the element $z_{\bullet}$ regarded as a functional $M^{\otimes}\to k$ vanishes on the linear subspace $I$ spanned by elements of the form \[m_1\otimes\cdots\otimes m_k\otimes (x\otimes y+y\otimes x -\om(x,y))\otimes m_{k+1}\otimes\cdots\otimes m_l.\] This is equivalent to say that $z_{\bullet}$ comes from a function $\Cl(M)=M^{\otimes}/I\to k$. Since the bilinear form $\Cl(M)\otimes \Cl(M)\to \Cl(M)\xrightarrow{\tr} k$ is non-degenerate, we know that there exists $F\in\Cl(M)$ satisfying the desired property.
    
\end{proof}
    Note that \itemref{kolychar}{kolyiter} follows directly from the lemma above.
    
    To prove \itemref{kolychar}{kolypf}, one notes for every $m_1\otimes m_2\otimes\cdots\otimes m_r\in M^{\otimes r}$, we have 
    \begin{equation*}
    \begin{split}
        z_r(m_1\otimes m_2\otimes\cdots\otimes m_r) &=\tr(m_1m_2\cdots m_r\tilF)\\
        &=(-\det(F))^{r-1}\tr(m_r\tilF m_1m_2\cdots m_{r-1} )\\
        &=(-1)^{r-1}\det(F)^{r-1}\tr(\tilF\tilF^{-1}m_r\tilF m_1m_2\cdots m_{r-1} )\\
        &=(-1)^{r-1}\tr(F^{-1}(m_r)m_1m_2\cdots m_{r-1}\tilF)\\
        &=(-1)^{r-1}z_r(F^{-1}(m_r)\otimes m_1\otimes m_2\otimes\cdots\otimes m_{r-1})\\
        &=\pF(z_r)(m_1\otimes m_2\otimes\cdots\otimes m_r).
    \end{split}
    \end{equation*}
    This concludes the proof of \itemref{kolychar}{kolypf}.
    
    To prove \itemref{kolychar}{kolyuni}, we choose a Pin structure $\tilF$ defined in \thref{pinstructure} and suppose we have a system $\{z'_{r}\in M^{*\otimes r}\}_{r\in\Z_{\geq 0}}$ satisfying conditions in \iitemref{kolychar}{kolyiter}{kolypf}. By \thref{dualclif}, there exists an element $T\in\Cl(M)$ such that the system $\{z'_{r}\in M^{*\otimes r}\}_{r\in\Z_{\geq 0}}$ is constructed from $T$ via \eqref{kolycons1}. Then the condition \itemref{kolychar}{kolypf} tells us that for every $m_1\otimes m_2\otimes\cdots\otimes m_r\in M^{\otimes r}$, one has \[\tr(m_1m_2\cdots m_{r-1}m_r T)=\tr(m_1m_2\cdots m_{r-1}(T_0-T_1)F^{-1}(m_r))\] where $T_0,T_1$ are the even and odd part of $T$, respectively.
    Since this identity holds for any $r$, we know that for any $m\in M$ one has $mT=(T_0-T_1)F^{-1}(m)$, which implies that $T\tilF^{-1}$ lies in the (super)center of $\Cl(M)$ which is well-known to be $k\sub\Cl(M)$. This implies that $T$ and $\tilF$ differ by a scalar, and we finish the proof of \itemref{kolychar}{kolyuni}.
    
\end{proof}

\subsubsection{Relation to the Selmer rank}
This section is not used elsewhere in the article. It discusses an interesting property of the Kolyvagin system, first observed in \cite{YZ3}, which also explains its name. We include it here to emphasize this property and to provide an independent proof.

Fix a GPin Kolyvagin polarization $(M,\om,\tilF,\om^{1/2}_{\top})$ as in \thref{kolypolarization} and construct the associated Kolyvagin system $\{z_{r}\in M^{*\otimes r}\}_{r\in\Z_{\geq 0}}$ as in \thref{kolydefn}. We have the following facts:
\begin{prop}\thlabel{ord=selmer}
For the GPin Kolyvagin system $\{z_{r}\in M^{*\otimes r}\}_{r\in\Z_{\geq 0}}$, define its \emph{order} to be \begin{equation}
        \ord(z_{\bullet}):=\min_{z_r\neq 0}r \in\Z_{\geq 0}.
    \end{equation}
    Then there is an equality \[\ord(z_{\bullet})=\dim M^{F=1}.\]
\end{prop}

\begin{proof}
    Note that $z_r\neq 0$ is equivalent to $\tr((G_r\Cl(M))\tilF)\neq0$, where the filtration $G_{\bullet}\Cl(M)$ is defined in \eqref{cliffil}. Therefore, we know that $\ord(z_{\bullet})$ is the minimal integer $r$ such that $\tr(G_r\Cl(M)\tilF)\neq 0$. Suppose $\tilF\in G_k\Cl(M)\bs G_{k-1}\Cl(M)$, we know that \begin{equation}\label{selmer1}\ord(z_{\bullet})=2n-k.\end{equation} On the other hand, the number $k$ is also the minimal integer $k$ such that $F\in \Or(M)$ can be written as a composition of $k$ reflections. This description gives  \begin{equation}\label{selmer2}k=2n-\dim M^{F=1}.\end{equation} Combining \eqref{selmer1}\eqref{selmer2}, we get $\ord(z_{\bullet})=\dim M^{F=1}$.
\end{proof}

\subsubsection{Relation to $L$-functions}\label{kolylfunsection}
Fix a Kolyvagin polarization $(M,\om,\tilF,\om^{1/2}_{\top})$ as in \thref{kolypolarization} and construct the associated Kolyvagin system $\{z_{r}\in M^{*\otimes r}\}_{r\in\Z_{\geq 0}}$ as in \thref{kolydefn}. We are now going to give an expression of central derivatives of $L$-functions via the Kolyvagin system. To do this, we first recall the definition of $L$-functions and root number in this setting.

\begin{defn}\thlabel{lfuncrootnumber}
Suppose we are given $(V,F)$ where $V$ is a finite dimensional (super) vector space and $F\in\GL(V)$,
   \begin{enumerate}
       \item \label{lfunc} We define the \emph{$L$-function} of the pair $(V,F)$ to be \[L(V,F,s):=\det(1-q^{-s}F)^{-1};\] 
       \item \label{rootnumber} We define the \emph{root number} of $(V,F)$ to be \[\e(V,F)=\det(F)^{-1}.\]
   \end{enumerate} where the (super)determinant $\det$ is defined in \eqref{detdef}.
\end{defn}

Recall that $\om$ is a symplectic pairing on $M$, which can be regarded as an isomorphism $\om: M\isom M^*$ as in \eqref{pairingchange}, hence gives us a symplectic form on $M^*$ by transporting $\om$ via the isomorphism. This further induces bilinear forms $\om_r$ on $M^{*\otimes r}$.

\begin{remark}\thlabel{pairingconvention}
    We recall our convention for the bilinear form $\om_r$ given in \eqref{tensorb}. Under the isomorphism $\om:M\isom M^*$, the bilinear forms $\om_r$ on $M^{*\otimes r}\isom M^{\otimes r}$ is given by \begin{equation}\label{tensorbilienarformula} \om_r(m_1\otimes\cdots\otimes m_r,m_1'\otimes\cdots\otimes m_r')=(-1)^{r(r-1)/2}\om(m_1,m_1')\cdots\om(m_r,m_r').\end{equation} 
\end{remark}

\begin{warning}
    One should pay attention to the difference in ordering elements between formula \eqref{tensorbilienarformula} and the formula \eqref{tensordualformula}. This difference is \emph{essential} for the following \thref{kolylfun} to be true.
\end{warning}

The main result in this section is the following:
\begin{thm}\thlabel{kolylfun}
For the Kolyvagin system $\{z_{r}\in M^{*\otimes r}\}_{r\in\Z_{\geq 0}}$ associated to a Kolyvagin polarization $(M,\om,\tilF,\om^{1/2}_{\top})$ introduced in \thref{kolydefn}, the following identity holds for any $r\in \Z_{\geq 0}$:
\begin{equation}\label{kolylfunformula}
    \om_r(z_r,z_r)=\l(\tilF)\e_{n,r}(\log q)^{-r}\left(\frac{d}{ds}\right)^{r}\Big|_{s=0}(q^{ns}L(M,F,s))
\end{equation} where $\e_{n,r}=(-1)^{r(r-1)/2+n}$.
Here, the map $\l:\Koly_F\to k$ is defined in \eqref{kolynorm}.
\end{thm}

\begin{remark}
    It is easy to see that when the Kolyvagin system in \thref{kolylfun} is Pin as defined in \thref{kolydefn}, both sides of \eqref{kolylfunformula} are independent of the choice of Pin Kolyvagin polarization (i.e., changing $\tilF$ by a sign changes neither side of \eqref{kolylfunformula}).
\end{remark}

\begin{remark}
    When the Kolyvagin system is GPin (i.e., non-zero), the minimal number $r$ such that the right-hand-side of \eqref{kolylfunformula} is non-zero is $r=\dim(M^{(F-1)^{\infty}=0})$ (i.e., the dimension of the \emph{generalized} eigenspace of $M$ with eigenvalue $1$). Comparing with \thref{ord=selmer}, one sees that if the Kolyvagin system satisfies $\om_r(z_r,z_r)\neq 0$ whenever $z_r\neq 0$ (for example, this holds if the pairing $\om_r$ has some positivity), then $M^{(F-1)^{\infty}=0}=M^{F=1}$. This is used in \cite{YZ3}, in which the positivity of the pairing is provided by the Hodge index theorem.
\end{remark}

\begin{proof}
    For each $r\in\Z_{\geq 0}$, the formula \eqref{kolylfunformula} regarded as an equation for $F\in\Or(M)$ defines a Zariski closed subset of $\Or(M)$. Therefore, we only need to prove \eqref{kolylfunformula} for semisimple $F$. 

    We first show that the identity \eqref{kolylfunformula} is stable under direct sum in $M$. Indeed, suppose $M=M_1\oplus M_2$ is an $F$-stable (orthogonal) direct sum decomposition as symplectic odd vector spaces. We define $\om_{M_i}=\om|_{M_i}$. We can assume that the Kolyvagin system $\{z_r\in M^{*\otimes r}\}_{r\in\ZZ_{\geq 0}}$ comes from an $F$-equivariant $\Cl(M)$-module $S$ (see \autoref{modulekoly} for the definition), which admits an $F$-equivariant decomposition $S=S_1\otimes S_2$ where $S_i$ is an $F$-equivariant $\Cl(M_i)$-module. We use $\{z_{M_i,r}\in M_i^{*\otimes r}\}_{r\in\ZZ_{\geq 0}}$ to denote the associated Kolyvagin system for $S_i$. Note that we have a decomposition $M^{\otimes r}=\bigoplus_{0\leq r_1\leq r} (M_1^{\otimes r_1}\otimes M_2^{\otimes r-r_1})^{\oplus \binom{r}{r_1}}$ such that $z_{r}|_{M_1^{\otimes r_1}\otimes M_2^{\otimes r-r_1}}=\pm z_{M_1,r_1}\otimes z_{M_2,r_2}$ for each direct summand $M_1^{\otimes r_1}\otimes M_2^{\otimes r-r_1}\sub M^{\otimes r}$. Also note that $(-1)^{r(r-1)/2}\om_{r}|_{M^{*\otimes r}}=(-1)^{r_1(r_1-1)/2+(r-r_1)(r-r_1-1)/2}\om_{M_1,r_1}\otimes \om_{M_2,r_2}$. By Leibnitz's rule, a similar formula holds for the right-hand side of \eqref{kolylfunformula}. Therefore, the identity \eqref{kolylfunformula} holds for $M$ if it holds for $M_1$ and $M_2$.

    Now, we can assume $\dim M=2$ and apply the computational result in \autoref{clifexp}. We only do the case $\det(F)=1$ and leave the case $\det(F)=-1$ to the readers. We can assume $\lambda(\widetilde{F})=1$. In this case, we can do the computation in the setting of \thref{actpin0}. Note that $S=k\cdot 1\oplus k\cdot e_1^*$ in this case. The element $\widetilde{F}$ acts on $S$ by $\widetilde{F}\cdot 1=\alpha_1^{1/2}\cdot 1$ and $\widetilde{F}\cdot e_1^*=\alpha_1^{-1/2} e_1^*$. Since both sides of \eqref{kolylfunformula} are zero if $r$ is odd, we can assume $r$ is even.

    We first compute the left-hand side of \eqref{kolylfunformula}. When $r=0$, we have $z_0=\alpha_1^{1/2}-\alpha_1^{-1/2}$, hence, $\om_0(z_0,z_0)=(\alpha_1^{1/2}-\alpha_1^{-1/2})^2=-(2-\alpha_1-\alpha_1^{-1})$. When $r>0$, since $(e_1\otimes e_1^*)^{\otimes r/2}\cdot 1=1$ and $(e_1\otimes e_1^*)^{\otimes r/2}\cdot e_1^*=0$, we have $z_{r}((e_1\otimes e_1^*)^{\otimes r/2})=\alpha_1^{1/2}$. Similarly, we have $z_{r}((e_1^*\otimes e_1)^{\otimes r/2})=-\alpha_1^{-1/2}$. Since $z_{r}$ takes zero value on other basis vectors of $M^{\otimes r}$, we know $\om_r(z_r,z_r)=(-1)^{r(r-1)/2+1}2$. 

    Now we compute the right-hand side of \eqref{kolylfunformula}. Note that $q^{-s}L(M,F,s)=q^s+q^{-s}-\alpha_1-\alpha^{-1}_1$. Therefore, we have \[\e_{n,r}(\log q)^{-r}\left(\frac{d}{ds}\right)^{r}\Big|_{s=0}(q^{ns}L(M,F,s))= \left\{\begin{aligned}
        -(2-\alpha_1-\alpha_1^{-1}) & ,r=0 \\
        (-1)^{r(r-1)/2+1}(1+(-1)^r) & ,r>0
    \end{aligned} \right.\] Comparing the computational results above, the identity \eqref{kolylfunformula} is proved.
\end{proof}

\begin{remark}
    One can give an alternative proof of \thref{kolylfun} by interpreting the pairing value $\omega_r(z_r,z_r)$ as a single trace: Consider the canonical element $\mathrm{coev}\in M\otimes M$ determined by the form $\omega$. It is easy to write $\omega_r(z_r,z_r)=\tr_2((\coev)^r\widetilde{F}_2)$ where $\tr_2:\Cl(M\oplus M)\to k$ is the map \eqref{tr} for $M\oplus M$ and $\widetilde{F}_2\in \Cl(M\oplus M)$ is some lifting of $(F,F)\in \mathrm{O}(M\oplus M)$. The rest of the proof is again an explicit computation. This alternative proof highlights a similarity between the pairing $\omega_r(z_r,z_r)$ and the arithmetic volume of Shtukas, defined as a trace in \cite{feng2026arithmeticvolumesmodulistacks}. This perspective also extends beyond the strongly tempered cases considered here, and we plan to return to it in future work.
\end{remark}

\subsection{Kolyvagin systems from modules}\label{modulekoly}
In  \autoref{koly}, we introduced the notion of the Kolyvagin system associated to a Kolyvagin polarization $(M,\om,\tilF,\om_{\top}^{1/2})$ defined in \thref{kolypolarization}. In this section, for each $(M,\om,F)$ in \thref{Fsymp}, we introduce the Kolyvagin system associated to a finite-dimensional $F$-equivariant module $T\in \Mod(\Cl(M))^{\om,F}$.

For $F\in \Or(M)$, we denote its induced automorphism on $\Cl(M)$ by $F_{\Cl(M)}$.
\begin{defn}\thlabel{kolyrep}
    Fix a triple $(M,\om,F)$ in \thref{Fsymp}. We define the category of $F$-equivariant modules of $\Cl(M)$ to be \[\Mod(\Cl(M))^F:=\Mod(\Cl(M))^{F_{\Cl(M)}^*}\] in which the perfect complexes form the full subcategory \[\Mod(\Cl(M))^{\om,F}\sub\Mod(\Cl(M))^{F}.\] In concrete terms, given an object in $\Mod(\Cl(M))^F$ is equivalent to give an object $T\in\Mod(\Cl(M))$ together with an automorphism of perfect complexes $F_T:T\to T$ such that \begin{equation}\label{Fcompatible} F(x)\cdot F_T(t) = F_T(x\cdot t)\end{equation} for any $x\in \Cl(M)$ and $t\in T$.
\end{defn}

\begin{defn}\thlabel{kolydefrep}
    For each $T\in\Mod(\Cl(M))^{\om,F}$, we get a map of algebras $a:\Cl(M)\to\End_{\Vect}(T)$. We define the Kolyvagin system associated to $T$ to be sequence of elements \[\{z_{T,r}\in M^{*\otimes r}\}_{r\in\Z_{\geq 0}}\] where the element $z_r$ is defined as the coelement \[z_{T,r}:M^{\otimes r}\to k\] such that   
\begin{equation}\label{kolyconsrep}
    z_{T,r}(m_1\otimes m_2\otimes\cdots\otimes m_r)=\tr( a(m_1)\circ\cdots\circ a(m_r)\circ F_T).
\end{equation}
\end{defn}

Now we fix a Pin Kolyvagin polarization as defined in \thref{kolypolarization}. The parity $\om^{1/2}_{\top}$ gives a choice of parity on the Clifford module $S$. The Pin structure $\tilF$ defines an $F$-equivariant structure on $S$. From now on, we assume $\det(F)=1$. In this case, we are going to construct a Kolyvagin structure $\tr_{\Koly}(T)\in\Koly_F$ (see \thref{kolystructure} for definition) such that the system in \thref{kolydefrep} is the Kolyvagin system attached to the Kolyvagin polarization $(M,\om,\tr_{\Koly}(T),\om_{\top}^{1/2})$.

Since $\Cl(M)\isom \End(S)$, we have Morita equivalence \begin{equation}
\Mod(\Cl(M)) \isom \Vect \end{equation} where the equivalence is given by \[T\leftrightarrow \Hom_{\Mod(\Cl(M))}(S,T) \] \[V\otimes S\leftrightarrow V\]

In particular, for each $T\in\Mod(\Cl(M))^{\om,F}$, we get $V_T:=\Hom_{\Mod(\Cl(M))}(S,T)\in\Vect$ such that $T=S\otimes V_T$. 

 Note the following lemma:
\begin{lemma}\thlabel{Fequivmodclassification}
    When $\det(F)=1$, the $F$-equivariant $\Cl(M)$-module automorphism $F_T\in\End_{\Vect}(T)\isom\Cl(M)\otimes \End_{\Vect}(V_T)$ lies in the subspace \[\Koly_F\otimes \End_{\Vect}(V_T)\sub\Cl(M)\otimes \End_{\Vect}(V_T).\]
\end{lemma}
\begin{proof}[Proof of \thref{Fequivmodclassification}]
    By \eqref{Fcompatible}, an element $f\in\Cl(M)\otimes \End_{\Vect}(V_T)$ is $F$-equivariant if and only if \begin{equation}\label{Fcompatible1}
        F(x)f=fx
    \end{equation}
    for all $x\in\Cl(M)$. Here we are using the natural $\Cl(M)$-bimodule structure on $\Cl(M)\otimes \End_{\Vect}(V_T)$ (trivial on the second factor). By choosing a Pin structure $\tilF$ as in \thref{pinstructure}, we can rewrite \eqref{Fcompatible1} as \[\tilF x\tilF^{-1}f=fx,\] which further rewrites to \[x (\tilF^{-1}f)=(\tilF^{-1}f)x\] that holds for every $x\in\Cl(M)$. Then we are done since the center of $\oblv(\Cl(M))$ is $k\sub\Cl(M)$.
\end{proof}

\begin{remark}
    Note that \thref{Fequivmodclassification} gives a bijection between isomorphism classes of $F$-equivariant $\Cl(M)$-modules and matrices with elements in $\Koly_F$ when $\det(F)=1$.
\end{remark}

Then we have the following immediate corollary:
\begin{cor}\thlabel{kolyrep=koly}
    Assume $\det(F)=1$. For each $T\in\Mod(\Cl(M))^{\om,F}$, consider the map $\tr_{\Koly}:=\id\otimes\tr:\Koly_F\otimes \End_{\Vect}(V_T)\to\Koly_F$ which is taking trace on the second factor. Then we get \[\tr_{\Koly}(T):=\tr_{\Koly}(F_T)\in \Koly_F .\] The system \[\{z_{T,r}\in M^{*\otimes r}\}_{r\in\Z_{\geq 0}}\] in \thref{kolydefrep} is the Kolyvagin system attached to the Kolyvagin polarization $(M,\om,\tr_{\Koly}(F_T),\om_{\top}^{1/2})$ via \thref{kolydefn}.
\end{cor}

\begin{remark}
    When $T$ is irreducible (i.e., isomorphic to the Clifford module $S$ up to parity), the element $\tr_{\Koly}(T)$ is non-zero since $F_T$ is non-zero. This implies that the Kolyvagin system $\{z_{T,r}\in M^{*\otimes r}\}_{r\in\Z_{\geq 0}}$ is non-zero (i.e., is a GPin Kolyvagin system).
\end{remark}

Given \thref{kolyrep=koly}, one can apply \thref{kolylfun} to relate the system $\{z_{T,r}\in M^{*\otimes r}\}_{r\in\Z_{\geq 0}}$ with $L$-functions. For each $T\in\Mod(\Cl(M))^{\om,F}$, we can now define \begin{equation}\label{kolynormrep}
    \l(T):=\l(\tr_{\Koly}(T))
\end{equation} where $\l:\Koly_F\to k$ is defined in \eqref{kolynorm}. The number $\l(T)$ is the most important invariant of an $F$-equivariant module of $\Cl(M)$ for us.

\subsubsection{Computing $\l(T)$}
The result in this section will not be used until \thref{higherformulagln}.

We do not have a general way to compute the invariant $\l(T)$ unless we get very explicit knowledge of the $F$-equivariant module $T\in\Mod(\Cl(M))^{\om, F }$. Even if $T$ is isomorphic to the Clifford module as a $\Cl(M)$-module, we still need to know about its $F$-equivariant structure. In this section, we introduce a particular case where we can compute $\l(T)$ easily. For this, we assume we are in the following setting:
\begin{setting}\thlabel{Fstablepolar}
    Given $(M,\om, F)$ as in \thref{Fsymp}. Moreover, we fix an $F$-stable polarization $M=L\oplus L^*$.
\end{setting}
In this case, we have the following 

\begin{prop}\thlabel{kolynormstandard}
    Suppose we are in \thref{Fstablepolar}, and $S=\Sym^{\bullet}L^*$ is the Clifford module constructed in \thref{cliffordmodule}, equipped with its natural $F$-equivariant structure $F_{\Sym^{\bullet}L^*}$. Then we have \[\l(\Sym^{\bullet}L^*)=\e(L,F)^{-1}\] where $\e(L,F)$ is the root number defined in \itemref{lfuncrootnumber}{rootnumber}.
\end{prop} 
\begin{proof}
    By a Zariski dense argument, we are reduced to the case that $F$ is semisimple. Then we are done by \thref{actpin0}.
\end{proof}

Therefore, we can compute $\l(T)$ by relating $T$ to the module $S=\Sym^{\bullet}L^*$. For this purpose, we introduce the following lemma:

\begin{lemma}\thlabel{identifystandard}
    Suppose we are in \thref{Fstablepolar}. Given an $F$-equivariant module $T\in\Mod(\Cl(M))^{\om,F,\he}$ of dimension $(2^{n-1}|2^{n-1})$. Suppose there exists a $F_T$-eigen highest weight vector $t\in T$ in the sense that:
    \begin{enumerate}
        \item \label{vectorFstable} We have $F_T(t)=\a t$ for some $\a\in k^{\times}$;
        \item \label{vectorhighestweight}We have $L\cdot t=0$
    \end{enumerate}
    Then the natural action map \[a:S=\Sym^{\bullet}L^*\to\Sym^{\bullet}L^*\cdot t\sub T\] defines an isomorphism of $F$-equivariant modules after we replace the $F$-equivariant structure on $T$ by $F_T/\a$. In particular, we have \[\l(T)=\a^2\l(\Sym^{\bullet}L^*)=\a^2\e(L,F)^{-1}.\]
\end{lemma}
\begin{proof}[Proof of \thref{identifystandard}]
    It is easy to see that the map $a: S\to T$ is compatible with $\Cl(M)$-module structures by \eqref{vectorhighestweight} and is $F$-equivariant after replacing $F_T$ by $F_T/\a$ given \eqref{vectorFstable}, therefore must be an isomorphism of $F$-equivariant $\Cl(M)$-modules since both sides have the same dimension.  The final statement is then a consequence of \thref{kolynormstandard}.
\end{proof}

\section{Commutator relations on cohomological correspondences}\label{commutatorrelations}
In this section, we introduce the local-to-global construction of special cohomological correspondences. In particular, relations between local special cohomological correspondences (commutator relations) would give relations on global special cohomological correspondences. These relations are what we call the commutator relations on global special cohomological correspondences. The main result in this section is \thref{geokolycor} (proved under technical assumptions in \thref{geokolycorthm}).
\begin{itemize}
    \item In \autoref{sheaftheory}, we clarify the \'etale sheaf theory we are using and establish some properties of it which are well-known to the experts but we cannot find in the literature.
    \item In \autoref{cc}, we review the notion of cohomological correspondences used in \cite{FYZ3} and make some enhancements in our setting.
    \item In \autoref{geometricsetup}, we introduce the moduli spaces and categories involved in the local-to-global procedure.
    \item In \autoref{conscc}, we make the local-to-global construction of special cohomological correspondences.
    \item In \autoref{mainresult3}, we state the commutator relation on global special cohomological correspondences, which is the main result of this section. 
    \item Finally, in \autoref{proofgeokolycorthm}, we prove the main theorem \thref{geokolycorthm}.
\end{itemize}

\subsection{Sheaf theory}\label{sheaftheory}
We now come to set up the \'etale sheaf theory that we will use throughout the article. Experienced readers can safely skip this part and come back when necessary.

The local-to-global construction of cohomological correspondences, which we are going to introduce in \autoref{conscc}, is essentially part of the automorphic factorizable $\Theta$-series construction discussed in \cite[\S16]{BZSV}. It has been observed in \textit{loc.cit} that this construction encounters some difficulties in sheaf theory, and has not been made rigorous at this stage as far as we know. Establishing the full version of this construction is stated as an open problem in \cite[Problem\,16.3.2]{BZSV}. We do not try to address this problem in its full generality. Instead, we only establish part of it, which is enough for our purpose (i.e., to see the Clifford relation in the RTF-algebra). 

For this purpose, we need a $*$-sheaf theory\footnote{By $*$-sheaf theory, we mean a theory in which the upper-$*$ functor is always defined.} (to study the $!$-period sheaf $\cP_X$) on non-Artin and non-placid stacks like $LX/L^+G$ (and its moving point variants like $(LX/L^+G)_C$). This is the sheaf theory $\Shv$ we are going to produce in this section. Note that although the entire category $\Shv$ can be very bad on general stacks, the six-functor functoriality should be mostly well-behaved, which is enough for our purpose.

\subsubsection{Six-functor formalism of ind-constructible sheaves}
We briefly review the six-functor formalism of \'etale sheaf theory with $k=\Q_\ell$-coefficients, which will be used later.

We use $\Sch$ to denote the category of (classical) schemes over $\overline{\F}_q$, and $\Stack$ to be the category of (\'etale) stacks over $\overline{\F}_q$. 

Let $X$ be a scheme of finite type over $\overline{\F}_q$, the (infinity-categorical enhancement of) derived category of constructible complexes $\rrr{Shv}_c(X, k)$ has been constructed in \cite{liu2017enhancedoperationsbasechange} (for torsion coefficients) and \cite{liu2017enhancedadicformalismperverse} for adic coefficients via right Kan extension along embedding torsion rings into adic rings,\footnote{Actually in \textit{loc. cit}, it is even constructed for Artin stacks locally of finite type.} which extends to a lax symmetric monoidal functor
\begin{equation}
    \rrr{Shv}_c: \rrr{Corr}(\rrr{Sch}_{\textup{ft}})_{E=\textup{fp}}\rightarrow \rrr{Cat}^{\rrr{Idem}}_{k}.
\end{equation}
where $ft$ is short for finite type schemes and $fp$ is short for of finite presentation morphisms. We use $qcqs$ short for quasi-compact and quasi-separated.

Using left Kan extension along $\rrr{Sch}_{\textup{ft}}^{\op}\sub \rrr{Sch}_{\qcqs}^{\op}$, and applying ind-completion $\rrr{Shv}(-):=\rrr{Ind}(\rrr{Shv}_c(-,))$, we obtain a six functor formalism (in the sense of \cite[Definition 3.2.1]{HM6fun}\footnote{Our $\rrr{Corr(\cC)_E}$ is the same as $\rrr{Corr}(\cC, E)^{\otimes}$ in \cite{HM6fun}.})
\begin{equation}
     \rrr{Shv}: \rrr{Corr}(\rrr{Sch}_{\textup{qcqs}})_{E=\textup{fp}}\rightarrow \LinCat_{k},
 \end{equation}

Now we extend the six-functor formalism to \'etale stacks. Let $P$ be a property of morphisms (of quasi-compact quasi-separated schemes) that is stable under base change. In this case, we say a morphism $f: X\rightarrow Y$ has property $P$ if $f$ is representable and any base change of $f$ along a map from a qcqs scheme $S\rightarrow Y$ has property $P$.

Then we use the right Kan extension along the full embedding $\Sch^{\op}\sub\Stack^{\op}$. By \cite[Theorem\,3.4.11]{HM6fun}, we obtain a six-functor formalism
 \begin{equation}
     \rrr{Shv}: \rrr{Corr}(\Stack)_{E=\textup{fp}}\rightarrow \LinCat_{k},
 \end{equation}
whose restriction to qcqs schemes is the same as the one we constructed before.

More concretely, given a qcqs scheme $S=\varprojlim_{i}S_i$ where $S_i$ are schemes of finite-type, we have $\Shv(S)\cong\Ind\varinjlim_i\Shv_c(S_i)$, where the transition maps are $*$-pull. For an \'etale stack $X=\varinjlim_{i}X_i$ where $X_i\in\Sch_{\qcqs}$, we have $\rrr{Shv}(X)\cong \varprojlim\rrr{Shv}(X_i)$, where the transition maps are $*$-pull.

For $f: X\rightarrow Y$ a morphism of stacks, we always have $f^*$, which admits a (not necessarily continuous) right adjoint $f_*$. If $f$ is fp, we have $f_!$, which admits a (not necessarily continuous) right adjoint $f^!$. We also have a right adjoint $\iHom$ to the tensor product $\otimes$ on $\Shv(X)$. In this case, these functors satisfy various compatibilities such as projection formulas, proper base change, etc. We refer to \cite[Proposition\,3.2.2]{HM6fun} for a complete list of such compatibilities. 

Moreover, for a morphism between stacks $f:X\to Y$:
\begin{enumerate}
    \item If $f$ is proper, then $f$ is $D$-proper in the sense of \cite[Definition\,4.6.1(a)]{HM6fun}. In particular, we have a canonical isomorphism $f_!\cong f_*$;
    \item If $f$ is \'etale, then $f$ is $D$-\'etale in the sense of \cite[Definition\,4.6.1(b)]{HM6fun}. In particular, we have a canonical isomorphism $f^*\cong f^!$.
\end{enumerate} 
These claims follow directly from \cite[Lemma\,4.5.7]{HM6fun} and the right cancellative property of proper and \'etale maps.

\begin{remark}
This is a $*$-sheaf theory in the sense of \cite[Remark 1.3.10]{HM6fun}.
\end{remark}

Suppose $X$ is an Artin stack locally of finite type, the following proposition shows that $\Shv(X)$ is the same as the sheaf theory defined in \cite[\S F.1.1]{arinkin2022stacklocalsystemsrestricted}. 

\begin{prop}\thlabel{openpresen}
    Let $X=\varinjlim_{\alpha}X_\alpha$ be a filtered colimit of stacks with transition maps $j_{\a\b}: X_\a\rightarrow X_\b$ quasi-compact open immersions. Then 
    \begin{equation}
        \varinjlim_{\alpha}\rrr{Shv}(X_\alpha)\xrightarrow{\cong}\rrr{Shv}(X)
    \end{equation}
    is an equivalence, where the transition maps are given by $!$-push.
\end{prop}
\begin{proof}
    By \cite[Proposition\,6.2.1.9]{luriesag}, $\rrr{Shv}(X)\cong \varprojlim_{\alpha}\rrr{Shv}(X_\a)$, with the transition maps given by $*$-pull. Since each $j_\alpha: X_\alpha\rightarrow X_\b$ is a quasi-compact open immersion, $j_{\alpha\b}^*$ has continuous left adjoint $j_{\a\b,!}$. Therefore, by \cite[Corollary\,5.5.3.4, Theorem\,5.5.3.18]{luriehtt}, we have $$\varprojlim_{\alpha}\rrr{Shv}(X_\alpha)\cong \varinjlim_{\alpha}\rrr{Shv}(X_\alpha),$$ where the transition maps on the right-hand side are given by $j_{\a\b,!}$.
\end{proof}

\subsubsection{Pro-smooth base change}
\begin{defn}
    We say a morphism $f: X\rightarrow Y$ of quasi-compact quasi-separated schemes is \emph{elementary pro-smooth} (pro-smooth for short), if there exists a presentation $X\cong\varprojlim X_i$ as a cofiltered limit of qcqs schemes with smooth affine surjective transition maps, such that every map $X_i\rightarrow Y$ is smooth.
\end{defn}
The class of elementary pro-smooth is stable under base change, and hence this notion extends to stacks, i.e., a morphism of stacks is called elementary pro-smooth if it is representable and its base change to any scheme is elementary pro-smooth.
\begin{example}
If $X$ is an affine smooth scheme over $\overline{\F}_q$, then the jet scheme $L^+X$ (see \autoref{modulispaces}) is elementary pro-smooth. 
\end{example}
\begin{example}
Let $G$ be an affine smooth group scheme, and let $Z$ be a stack equipped with an $L^+G$-action. Then the natural map of stacks $Z\rightarrow Z/L^+G$ is elementary pro-smooth.
\end{example}
\begin{example}
    Let $D=\Spec \overline{\F}_q[[t]]$ be the formal disc and $\rrr{Aut}(D)$ the automorphism group of $D$ preserving the center of the disc. Then $\rrr{Aut}(D)$ is pro-smooth over $\overline{\F}_q$. In particular, the map of local coordinate $\rrr{Coor}: C\rightarrow B\rrr{Aut}(D)$ (see \autoref{modulispaces}) is pro-smooth.
\end{example}

\begin{lemma}\thlabel{conserv}
Let $f: X\rightarrow Y$ be a representable surjection of stacks. Then $f^*$ is conservative.
\end{lemma}
\begin{proof}
When $X$ and $Y$ are qcqs schemes, the argument in \cite{zhu2025tamecategoricallocallanglands}[Lemma\,10.27] works. 
In the general case, $\rrr{Shv}(Y)\cong\varprojlim_{S\rightarrow Y}\rrr{Shv}(S)$ and $\rrr{Shv}(X)\cong\varprojlim_{S\rightarrow Y}\rrr{Shv}(S\times_Y X)$, where $S$ is a qcqs scheme and all transition maps in the limits above are given by $*$-pullbacks. Then it follows from the scheme case by passing to limits.
\end{proof}

\begin{lemma}\thlabel{pro1}
Given a Cartesian square of qcqs schemes 
    \ppull{W}{X}{Y}{Z}{g'}{f}{g}{f'}
with $g$ elementary pro-smooth.
\begin{enumerate}
    \item \label{pro1Comp} If $f$ is of finite presentation, then $g'^*f^!\cong f'^!g^*$.
    \item \label{pro1BC} We always have that $g^*f_*\cong f'_*g'^*$.
\end{enumerate} 
\end{lemma}
\begin{proof}
This has been proved in \cite{zhu2025tamecategoricallocallanglands}[Proposition\,10.32, Corollary \,10.46] and we reproduce the proof here. Since all functors involved are continuous, we only need to verify $(g')^*f^!\cF\cong (f')^!g^*\cF\in\rrr{Shv}_c(Z)$. Since $g: Y\to Z$ is elementary pro-smooth, we can write $Z=\varprojlim_i Z_i$ and $Y=\varprojlim_i Y_i$ such that all $Y_i, Z_i$ are schemes of finite type and the map $g$ is induced from a compatible system of smooth morphisms $g_i: X_i\to Z_i$. Since $f$ is fp, we can assume that there exists $f_i:X_i\to Z_i$ such that $X=X_i\times_{Z_i} Z$ for each $i$. We obtain the following diagram

\begin{equation*}
\begin{tikzcd}
                             & W \arrow[ld, "f'"'] \arrow[rr, "{g'}"] \arrow[dd, dashed, "r_{W,i}"]    &              & X \arrow[ld, "f" '] \arrow[dd, "r_{X,i}"] \\
Y \arrow[dd, "r_{Y,i}"'] \arrow[rr, "g"'] &                                                          & Z \arrow[dd, "r_{Z,i}"] &                                     \\
                             & X_i\times_{Z_i}Y_i \arrow[ld, "f_i'"'] \arrow[rr, "g_i'"] &              & X_i \arrow[ld, "f_i"']               \\
Y_i \arrow[rr, "g_i"]        &                                                          & Z_i          &                                    
\end{tikzcd}
\end{equation*}
where all squares except the left and right faces are Cartesian. 

Assume $\cF$ comes from $\cF_0\in \Shv(Z_0)$, we use $\cF_i$ to denote the $*$-pull of $\cF_0$ to $Z_i$. Note we have \[f^!\cF\cong \varinjlim_i r_{X,i}^*f_i^!\cF_i\] and \[f'^!g^*\cF\cong \varinjlim_i r_{W,i}^*f_i'^!g_i^*\cF_i,\] which gives \[g'^*f^!\cF\cong \varinjlim_i g'^*r_{X,i}^*f_i^!\cF_i\cong \varinjlim_i r_{W,i}^*f_i'^!g_i^*\cF_i\cong f'^!g^*\cF.\]

Part \eqref{pro1BC} is proved \textit{mutatis mutandis}.
\end{proof}
\begin{lemma}\thlabel{pro2}
Given a Cartesian square of stacks
\ppull{X}{\sX}{Y}{\sY}{g'}{f}{g}{f'}
where $g$ and $g'$ are surjective pro-smooth, $X$ and $Y$ are qcqs schemes.
    \begin{enumerate}
    \item \label{pro2Comp} If $f$ is representable of finite presentation, then $g'^*f^!\cong f'^!g^*$;
    \item \label{pro2BC} We always have that $g^*f_*\cong f'_*g'^*$.
    \end{enumerate}
\end{lemma}
\begin{proof}
    Let $X_\bullet$ denote the Cech nerve of $g'$, where $X_n$ are qcqs schemes and all face maps are surjective pro-smooth. Let $X_\bullet^+$ be the underlying semi-simplicial diagram (i.e., discarding all non-injective maps in $\Delta^{\textup{op}}$). We always have $\varprojlim\rrr{Shv}(X_\bullet)\cong\rrr{Shv}(X)$, with transition maps given by $*$-pullbacks. By \cite[Lemma 6.5.3.7]{luriehtt}, we also have that $\varprojlim\rrr{Shv}(X_\bullet)\cong \varprojlim\rrr{Shv}(X_\bullet^+)$. Similarly for $Y_\bullet^+$. Now by \thref{pro1}, all the vertical maps in the following diagrams
    \[\begin{tikzcd}
\rrr{Shv}(Y) \arrow[r, shift right] \arrow[r] \arrow[d, "f'^!"] & \rrr{Shv}(Y_2) \arrow[r] \arrow[r, shift left] \arrow[r, shift right] \arrow[d] & \rrr{Shv}(Y_3) \arrow[d] \arrow[r] & \cdots \\
\rrr{Shv}(X) \arrow[r, shift left] \arrow[r]                               & \rrr{Shv}(X_2) \arrow[r, shift left] \arrow[r, shift right] \arrow[r]           & \rrr{Shv}(X_3) \arrow[r]           & \cdots
\end{tikzcd}\]
intertwine with face maps, and hence \eqref{pro2Comp} holds by \cite[Proposition\,4.7.4.19]{lurieha}. Similarly, \eqref{pro2BC} is proved.
\end{proof}
For applications, we further generalize pro-smooth base change to the ``locally pro-smooth'' case.
\begin{prop}[Locally pro-smooth base change]\thlabel{lprosmBC}
Suppose that we have a Cartesian square of stacks \ppull{\sW}{\sX}{\sY}{\sZ}{g'}{f}{g}{f'} Assume that $\sY$ admits a pro-smooth surjection $\gamma:Y\rightarrow \sY$, such that $Y$ is a quasi-compact quasi-separated scheme, and the composition $\delta:Y\xrightarrow{\gamma}\sY\xrightarrow{g}\sZ$ is also a pro-smooth surjection. 
\begin{enumerate}
    \item \label{pro3} If $f$ is representable of finite presentation, then we have $(g')^*f^!\cong (f')^!g^*$.
    \item \label{pro4} We always have $g^*f_*\cong f'_*(g')^*$.
\end{enumerate}
\end{prop}
\begin{proof}
Part \eqref{pro3}. Let $\gamma': W:=Y\times_\sY \sW\rightarrow\sW$ be the base change of $\gamma$ along $\sW\rightarrow\sY$. Then $\gamma'$ is a pro-smooth surjection. Since $\gamma'$ is surjective, $(\gamma')^*$ is conservative by \thref{conserv}, and we can check $(g')^*f^!\cong (f')^!g^*$ by pullback along $\gamma'$, i.e. we only need to show that 
$$(\gamma')^*(g')^*f^!\cong (\gamma')^*(f')^!g^*.$$ 
The three squares in 
\[\begin{tikzcd}\label{prodiagmproof}
	& W \\
	Y \\
	& \sW & \sX \\
	\sY & \sZ
	\arrow["{\widetilde{f}}"', from=1-2, to=2-1]
	\arrow["{\gamma'}"', from=1-2, to=3-2]
	\arrow["{\delta'}"{pos=0.3}, from=1-2, to=3-3]
	\arrow["\gamma"', from=2-1, to=4-1]
	\arrow["\delta"{pos=0.2}, from=2-1, to=4-2]
	\arrow["{g'}", from=3-2, to=3-3]
	\arrow["{f'}"', from=3-2, to=4-1]
	\arrow["f", from=3-3, to=4-2]
	\arrow["g"', from=4-1, to=4-2]
\end{tikzcd}\]
are all Cartesian.\par 
On the one hand, $(\gamma')^*(g')^*f^!\cong (\delta')^*f^!$. On the other hand, by \itemref{pro2}{pro2Comp}, $(\gamma')^*(f')^!\cong \widetilde{f}^!\gamma^*$, and hence $(\gamma')^*(f')^!g^*\cong \widetilde{f}^!\delta^*$. \par 
Again by \itemref{pro2}{pro2Comp}, $\widetilde{f}^!\delta^*\cong (\delta')^*f^!$, and hence we obtain the desired base change isomorphism.\par 
Part \eqref{pro4}.  Consider diagram \eqref{prodiagmproof}. Since $\gamma$ is surjective, by \thref{conserv}, we only need to show that $\gamma^*g^*f_*\cong \gamma^*f'_*(g')^*.$ \par 
On the one hand, $\gamma^*g^*f_*\cong \delta^*f_*$.
On the other hand, $\gamma^*f'_*\cong \widetilde{f}_*(\gamma')^*$ by \itemref{pro2}{pro2BC}. Then 
\begin{equation*}
\begin{split}
\gamma^*f'_*(g')^* &\cong \widetilde{f}_*(\gamma')^*(g')^*,\\
&\cong \widetilde{f}_*(\delta')^*.
\end{split}
\end{equation*}
Finally, $\widetilde{f}_*(\delta')^*\cong \delta^*f_*$ again by \itemref{pro2}{pro2BC}.
\end{proof}
\begin{remark}
The results above are also easily generalized to the case where we have a pro-smooth surjection $Z\rightarrow \sZ$ from a qcqs scheme $Z$, and $\sY\times_\sZ Z\rightarrow Z$ is pro-smooth (and $f$ is representable of finite presentation).
\end{remark}

\subsubsection{Excision}
We need the following easy excision property of $\Shv$:
\begin{prop}\thlabel{excisionSHV}
Let $j: U\rightarrow X$ be a quasi-compact open embedding of stacks with a closed immersion of complement $i: Z\rightarrow X$ (i.e., $U, Z$ form a stratification of $X$ after a base change to qcqs schemes). Assume that $i: Z\to X$ is of finite presentation, we have the following:
\begin{enumerate}
    \item \label{p111} The functor $i_!:\Shv(Z)\to\Shv(X)$ is fully faithful, with essential image consisting of $\cF\in\rrr{Shv}(X)$, such that $j^*\cF\cong 0$;
    \item \label{p222} The sequence
    \begin{equation}\label{fibseq}
        j_!j^*\cF\rightarrow \cF\rightarrow i_!i^*\cF
    \end{equation}
    is a fiber sequence for every $\cF\in\rrr{Shv}(X)$.
\end{enumerate}
\end{prop}
\begin{proof}
We first prove $i_!$ is fully faithful. This is equivalent to proving $i^*i_*\cong\id$. Using Yoneda lemma and the fact that $\Hom_{\sC}(x,y)=\varprojlim_i\Hom_{\sC_i}(x_i,y_i)$ whenever one has $\sC=\varprojlim_i\sC_i\in\LinCat_k$ and $x_i,y_i\in\sC_i$ are the images of $x,y\in \sC$, one can reduce to the case that $X$ is a qcqs scheme. By continuity of $i^*i_*$ in this case, one only needs to prove $i^*i_*\cF\cong\cF$ for $\cF\in\Shv_c(Z)$ coming from a finite type scheme, in which case the statement is well-known.

Then it is clear that \eqref{p222} implies \eqref{p111}. To prove \eqref{p222}, the same argument as above helps us reduce to the case that $X$ is a finite type scheme, in which the statement is well-known.

\end{proof}

\subsubsection{Cotruncated sheaves}
For a general stack $X$, $\rrr{Shv}(X)$ is too large. One main cause is that we use the right Kan extension to extend the sheaf theory from qcqs schemes to all stacks, while in applications, we only care about sheaves enjoying certain finiteness properties concerning a specific presentation of $X$. In this section, we introduce the cotruncated sheaves to single out the sheaves that we care about.\par
This section is not essentially used in this article. One can avoid introducing this notion by working with Schubert cells of Hecke stacks instead of the entire Hecke stack in \autoref{modulispaces}. We introduce this notion to compare with the sheaf theory in \cite{BZSV} (see \thref{COMPAREshvsafebzsv}) and simplify the notation in \autoref{modulispaces}.

\begin{defn}\thlabel{cotruncateddef}
Let $X$ be a stack. We define a \emph{cotruncated presentation} of $X$ to be the following data
\begin{enumerate}
    \item a $\ZZ$-indexed (for simplicity) ind-system $\{X_n\}_{n\in\ZZ}\in\Ind\!\Stack$,
    \item an equivalence $X\cong \varinjlim_{n} X_n$ in $\Stack$,
\end{enumerate}
such that each transition map $i_n: X_n\hookrightarrow X_{n+1}$ is a closed immersion of finite presentation.
\end{defn}
By \itemref{excisionSHV}{p111}, each $i_{n,!}: \rrr{Shv}(X_n)\rightarrow \rrr{Shv}(X)$ is fully faithful. 
\begin{remark}
    Note that $\rrr{Shv}(X)\cong\varprojlim_{n}\rrr{Shv}(X_n)$ with transition maps given by $i^*_n$, again by \cite[Proposition\,6.2.1.9]{luriesag}. Contrary to the case in \thref{openpresen}, the transition maps $i^*_n$ do not have a continuous left adjoint, so we can not rewrite $\Shv(X)$ as a colimit. Therefore, objects in $\rrr{Shv}(X)$ in general do not enjoy the finiteness property coming from $\rrr{Shv}(X_n)$, unlike its dual counterpart $!$-sheaf theory (considered for example in \cite[\S F]{arinkin2022stacklocalsystemsrestricted}).
\end{remark}
\begin{defn}\thlabel{shv*}
We define the category of \emph{cotruncated sheaves} on $X$ (associated to the cotruncated presentation $X=\varinjlim_{n} X_n$) to be \[\rrr{Shv}_*(X):=\varinjlim_{n}\rrr{Shv}(X_n)\] where the transition maps are given by $i_{n,*}$.
\end{defn}
\begin{remark}\hfill
\begin{enumerate}
    \item $\rrr{Shv}_*(X)$ depends on the choice of a cotruncated presentation of $X$. In applications, we choose specific cotruncated representations of the spaces we study.
    \item There is a fully faithful embedding $\rrr{Shv}_*(X)\hookrightarrow \rrr{Shv}(X)$, and we can view $\rrr{Shv}_*(X)$ as a subcategory of $\rrr{Shv}(X)$.
    \item If $X$ is an ind-scheme and $X\cong\varinjlim_{n}X_n$ with $X_n$ of finite type and transition maps that are closed embeddings (sometimes called strict ind-schemes), then $\rrr{Shv}_*(X)$ is isomorphic to $\rrr{Shv}^!(X):=\varprojlim_n\rrr{Shv}(X_n)$, where transition maps are given by $i_n^!$.
\end{enumerate}
\end{remark}
\begin{example}\thlabel{COMPAREshvsafebzsv}
     Let $G$ be a split reductive group acting on an affine smooth scheme $X$. Consider the loop stack (with cotruncated presentation) $LX/L^+G\rtimes\Aut(D)=\varinjlim_n L^{\geq -n}X/L^+G\rtimes\Aut(D)$ as defined in \autoref{modulispaces}. It is easy to see that all the transition maps $i_n:L^{\geq -n}X/L^+G\rtimes\Aut(D)\to L^{\geq-n-1}X/L^+G\rtimes\Aut(D)$ are of finite presentation. Then $\rrr{Shv}_*(LX/L^+G\rtimes\Aut(D))=\varinjlim_{n}\rrr{Shv}(L^{\geq -n}X/L^+G\rtimes\Aut(D))$ with transition maps $*$-push. We also have $\Shv_*(LX/L^+G)=\varinjlim_{n}\rrr{Shv}(L^{\geq -n}X/L^+G)$. Therefore, the full subcategory $\Shv_*(LX/L^+G)\sub\Shv(LX/L^+G)$ is the same as the large category of safe $*$-sheaves $\mathrm{SHV}_{s}^*(X_F/G_{\cO})$ defined in \cite[\S 7.2.3, \S B.7]{BZSV}.
\end{example}

\begin{defn}\thlabel{indp}
    Let $X$ be a stack with cotruncated presentation $X=\varinjlim_n X_n$ and $P$ be a property of morphisms (of qcqs schemes) stable under base change such that all closed immersions of finite presentation satisfy $P$. Let $Y$ be a stack. We say $f: X\rightarrow Y$ is \emph{ind-$P$}, if each induced morphism $f_n:X_n\to Y$ has property $P$.
\end{defn}
\begin{remark}
    Our definition of being ind-$P$ depends on a cotruncated presentation of the source stack.
\end{remark}
\begin{example}
    If $f: X=\varinjlim_n X_n \rightarrow Y$ is ind-fp, then the induced maps $f_n$ give $$f_{n,!}:\rrr{Shv}(X_n)\rightarrow \rrr{Shv}(Y)$$
    and we obtain a functor 
    $$f_!:\rrr{Shv}_*(X)\rightarrow \rrr{Shv}(Y)$$
    by passing to colimit. Moreover, if $f$ is ind-proper, we have a canonical isomorphism $f_!\cong f_*:\Shv_*(X)\to\Shv(Y)$.
\end{example}

\subsubsection{Derived fundamental class}
In this section, we temporarily work with derived algebraic geometry. Since the derived aspect will be used only in \autoref{dfc} in a mild way, we do not recall this theory and refer the readers to the literature (e.g. \cite{toen2014derived}, \cite{lurie2004derived}, etc.).

\begin{defn}[Derived fundamental class]\thlabel{dfcgeneral}
    Consider a square of derived stacks which is Cartesian on the classical truncation 
    \[ \begin{tikzcd}
    A \ar[r, "a"]\ar[d, "f"'] & B \ar[d, "g"] \\
    C \ar[r, "b"] & D
    \end{tikzcd}
    \]
    where $C,D$ are derived Artin stacks and $b$ is (representable) quasi-smooth of relative dimension $d$. Then there exists a natural map 
    \begin{equation}\label{dfcdefformula}
        [a]:\uk_{A} \to a^!\uk_{B} \langle -2d\rangle 
    \end{equation}
    defined as the composition
    \begin{equation}
        \uk_{A}\isom f^*\uk_{C}\xrightarrow{[b]} f^*b^!\uk_{D} \langle -2d\rangle \to a^!g^*\uk_{D} \langle -2d\rangle\isom   a^!\uk_{B} \langle -2d\rangle 
    \end{equation}
    Where $[b]:=[C/D]$ is the purity transformation for quasi-smooth maps between derived Artin stacks defined in \cite[Remark\,3.8]{khan2019virtualfundamentalclassesderived}. 
    
    In this case, we say that the map $a: A\to B$ is \emph{quasi-smooth}, and we call the map \eqref{dfcdefformula} the \emph{derived fundamental class} of the quasi-smooth map $a$. When $A$, $B$ themselves are derived Artin stacks, it follows from \cite[Theorem\,3.13]{khan2019virtualfundamentalclassesderived} that this notion of quasi-smoothness and derived fundamental class coincides with the usual one in \textit{loc.cit}.
\end{defn}

\begin{remark}
    The condition that $b$ is representable is not necessary.
\end{remark}

\begin{remark}
    We do not try to address the question concerning the dependence of the bottom map $b: C\to D$ in the construction of $[a]$ and the notion of quasi-smoothness. Whenever we use \thref{dfcgeneral} for non-Artin stacks, the choice of the bottom map $b$ will be specified or clear from the context.
\end{remark}

\begin{lemma}\thlabel{purity}
In the situation of \thref{dfcgeneral}, suppose we have a commutative diagram of Artin stacks
\[\begin{tikzcd}
    A\ar[dr, "p"'] \ar[rr, "a"] && B \ar[dl, "q"] \\
    & S &
\end{tikzcd}\]
where $p$, $q$ are smooth. Then the map \[[a]:\uk_{A} \to a^!\uk_{B} \langle -2d\rangle \] is an isomorphism. Here $d$ is the relative dimension of $a$. Moreover, for any $\cF\in\Shv(S)$, we have a canonical isomorphism
\[[a]:p^*\cF\isom a^!q^*\cF\langle -2d\rangle.\]
\end{lemma}

\begin{proof}
    This follows from 
    \[p^*\cF\isom p^!\cF\langle -2\dim(p)\rangle \isom a^!q^!\cF\langle -2\dim(p)\rangle \isom a^!q^*\cF\langle -2d\rangle\] and the construction of $[a]$.
\end{proof}

\subsubsection{Base change compatibilities}
    In this article, we are going to use many compatibilities between base change morphisms. These compatibilities are tautologically encoded in the six-functor formalism. We do not spell out all of them, but only give one typical example. The proof of other compatibilities will be the same as the following one:
    
\begin{lemma}\thlabel{!pushbcstar}
Consider a Cartesian square of stacks
    \begin{equation} \begin{tikzcd}
    A \ar[r, "a"]\ar[d, "f"'] & B \ar[d, "g"] \\
    C \ar[r, "b"] & D\rcart
    \end{tikzcd}.
    \end{equation}
    Suppose $f$, $g$ are representable of finite presentation, then we have the following commutative diagram
    \begin{equation}\label{bccompdiag}\begin{tikzcd}
        g_! \ar[d] \ar[r] & b_*b^*g_! \ar[d] \\
         g_!a_*a^*\ar[r] & b_*f_!a^* 
    \end{tikzcd}\end{equation}
\end{lemma}
\begin{proof}
    Recall the natural base change map $g_!a_*\to b_*f_!$ used in the bottom of \eqref{bccompdiag} is the Beck-Chevalley base change map induced from the proper base change map $b^*g_!\cong f_!a^*$, which is defined as the composition \begin{equation}\label{bcdef}g_!a_* \to b_*b^*g_!a_*\cong b_*f_!a^*a_*\to b_*f_!.\end{equation} Then the commutativity of \eqref{bccompdiag} following from the tautological commutativity of the following diagram
    \[\begin{tikzcd}
        g_! \ar[d] \ar[r] & b_*b^*g_! \ar[r] \ar[d] & b_*f_!a^*  \ar[d] \ar[dr] & \\
        g_!a_*a^* \ar[r] & b_*b^*g_!a_*a^* \ar[r] & b_*f_!a^*a_*a^* \ar[r] & b_*f_!a^*
    \end{tikzcd}.\]
\end{proof}

\begin{remark}
    When $b$ is D-proper in the sense of \cite[Definition\,4.6.1(b)]{HM6fun} (hence $a$ is D-proper as well), besides the Beck-Chevalley base change map $g_!a_*\to b_*f_!$ defined in \eqref{bcdef}, we also have another natural map defined as the composition $g_!a_*\cong g_!a_!\cong b_!f_!\cong b_*f_!$. One can easily check that these two natural maps are homotopic to each other: Using the inductive definition of D-properness in \cite[Definition\,4.6.1(b)]{HM6fun} and the definition of natural isomorphisms $b_!\cong b_*$ and $a_!\cong a_*$ in \cite[Definition\,4.6.4(ii)]{HM6fun}, one can reduce to the case where $a$, $b$ are isomorphisms, and then the statement is trivial.
\end{remark}

\subsection{Cohomological correspondence with kernel}\label{cc}
In this section, we briefly review the notion of cohomological correspondence used in \cite{FYZ3} and many other places. Compared to \cite{FYZ3}, we only consider Cartesian pull-back functoriality, but we need to put a kernel in the cohomological correspondence and work over non-Artin stacks. This is a mild generalization given \autoref{sheaftheory}, and readers familiar with the cohomological correspondences in \cite{FYZ3} can safely skip this part.
\subsubsection{Pushable squares}\label{pushabledef}
In this section, we briefly review the notion of pushable squares introduced in \cite[\S3.1]{FYZ3}.

A square of stacks \begin{equation}\begin{tikzcd}
    A \ar[r, "g'"] \ar[d, "f'"']& B \ar[d, "f"]\\
    C \ar[r, "g"] & D
\end{tikzcd}\end{equation} is called \emph{pushable} if the map $a:A\to \tilB$ in the following diagram is proper 
 \begin{equation}
    \begin{tikzcd}
        A\ar[drr, "g'"] \ar[dr, "a"] \ar[ddr, "f'" '] && \\
        &\tilB \ar[r, "\tilg" '] \ar[d, "\tilf"] & B \ar[d, "f"] \\
        & C \ar[r, "g"] & D\rcart
    \end{tikzcd}.
\end{equation} Here $\tilB:=C\times_D B$.

A pushable diagram in which $f$ and $f'$ are fp induces a natural base change natural transformation \begin{equation}\label{pushablebc}
    g^*f_!\to f'_!g'^*
\end{equation} via \[g^*f_!\isom \tilf_!\tilg^* \to \tilf_!a_*a^*\tilg^*\isom f'_!g'^*.\]

The following basic lemma is well-known:
\begin{lemma}\thlabel{pushablebccomp}
    Consider a diagram of stacks \[\begin{tikzcd}
        A\ar[r, "p'"]\ar[d, "f"'] & B \ar[r, "q'"] \ar[d, "g"]  & C \ar[d, "h"] \\
        D \ar[r, "p"] & E \ar[r, "q"] & F
    \end{tikzcd}\] where both squares are pushable. Then the outer square is also pushable, and we have commutative diagrams
    \[\begin{tikzcd}
        h^*q_!p_! \ar[d, "\sim"'] \ar[r] & q'_!g^*p_! \ar[r] & q'_!p'_!f^* \ar[d, "\sim"] \\
        h^*(q\circ p)_! \ar[rr] && (q'\circ p')_!f^*
    \end{tikzcd}\]
    and 
    \[\begin{tikzcd}
        p^*q^*h_! \ar[d, "\sim"'] \ar[r]& p^*g_!q'^* \ar[r]& f_!p'^*q'^* \ar[d, "\sim"] \\
        (q\circ p)^*h_! \ar[rr]& & f_!(q'\circ p')^*
    \end{tikzcd}.\] Here we assume that all the functors involved above are defined.
\end{lemma}
\begin{proof}
    When all the stacks involved are Artin stacks, this is \cite[Proposition\,3.2.3\&Proposition\,3.2.4]{FYZ3}. The proof in \textit{loc.cit} carries to general case given \thref{!pushbcstar}.
\end{proof}

\subsubsection{Definition}
Consider a correspondence between stacks
\[
\begin{tikzcd}
    &C\ar[dl, "g" ']\ar[dr, "f"]&\\
    Y&&X
\end{tikzcd}
\]
where $g$ is of finite presentation.
\begin{defn}
    For any $\cF\in \Shv(X)$, $\cG\in \Shv(Y)$ and $\cK\in\Shv(C)$, a \emph{cohomological correspondence (of degree $d$) between $\cF$ and $\cG$ (supported on $C$ with kernel $\cK$)} is an element \[c\in \Cor_{C,\cK}(\cF,\cG\langle d\rangle):=\Hom^0(g_!(f^*\cF\otimes\cK),\cG\langle d\rangle)\isom \Hom^0(f^*\cF\otimes \cK, g^!\cG\langle d\rangle) .\]
\end{defn}

\begin{remark}
    In our application, we also need the case where $C$ carries a cotruncated presentation (\thref{cotruncateddef}) such that $g$ is ind-fp and $\cK\in\Shv_*(C)$. All results in \autoref{cc} can be generalized to this case without difficulty.
\end{remark}

\begin{example}[Tautological correspondence]\thlabel{tautcor}
Consider the correspondence of stacks
\[\begin{tikzcd}
    &Y\ar[dl, "\id"']\ar[dr, "f"]& \\
    Y&&X
\end{tikzcd}.\]
Given a sheaf $\cF\in\Shv(X)$, we have a tautological element \[\taut_f\in\Cor_{Y,\uk_Y}(\cF,f^*\cF).\]
When $f$ has the form of a projection $f: Y=X\times Z\to X$, we also write $\taut_Z:=\taut_f$.
\end{example}

\begin{example}[Fundamental class]\thlabel{fundcor}
Consider correspondence
\[\begin{tikzcd}
    &X\ar[dl, "g"']\ar[dr, "\id"]& \\
    Y&&X
\end{tikzcd}\] where $g$ is smooth of dimension $d_g$.
Given a sheaf $\cG\in\Shv_*(Y)$, we have a element \[[g]\in \Cor_{X,\uk_X}(g^*\cG,\cG\langle -2d_g\rangle )\] given by the relative fundamental class $[g]:g^*\cF\to g^!\cF\langle -2d_g\rangle $ in \thref{purity}. When $g$ has the form of a projection $g:Y\times Z\to Y$, we also write $[Z]:=[g]$.
\end{example}

\subsubsection{Functoriality}\label{ccfunc}
We now introduce the pull-back and push-forward of cohomological correspondence. Consider a commutative diagram of correspondence between stacks (a map of correspondences) in which $g$ and $\tilg$ are of finite presentation:
\begin{equation}\label{functdiag}
\begin{tikzcd}
    \tilY\ar[d, "q_Y" '] &\tilC\ar[l,"\tilg" ']\ar[r,"\tilf"]\ar[d,"q"] & \tilX\ar[d, "q_X"] \\
    Y&C\ar[l, "g" ']\ar[r, "f"]&X
\end{tikzcd}.
\end{equation}
Moreover, assume we are given $\cK\in\Shv(C)$ and $\tilK\in\Shv(\tilC)$ equipped with a morphism \begin{equation}\label{cctrans}\a:q^*\cK\to\tilK.\end{equation}

We first introduce pull-back functoriality as defined in \cite[Section 4.4]{FYZ3}. While the definition in \cite{FYZ3} applies whenever the left square is \emph{pullable}, we only need the case that the special case that the left square is Cartesian.\footnote{While the pull-back functoriality along pullable maps of correspondences in \cite[Section 4.4]{FYZ3} is subtle, the Cartesian pull-back functoriality is trivial.}

\begin{defn}\thlabel{ccpull}
    When the left-square in \eqref{functdiag} is Cartesian and the map $\a$ in \eqref{cctrans} is an isomorphism, we say that $q$ is \emph{Cartesian pullable}. In this case, for each cohomological correspondence \[c\in\Cor_{C,\cK}(\cF,\cG\langle d\rangle )\] where $\cF\in \Shv(X), \cG\in \Shv(Y)$, we define its pull-back along $q$ to be the element \[q^*(c)\in\Cor_{\tilC,q^*\cK}(q_X^*\cF,q_Y^*\cG\langle d\rangle)\] as the image of $c$ along the map \[\begin{split}
\Cor_{C,\cK}(\cF,\cG\langle d\rangle )&= \Hom^0(g_!(f^*\cF\otimes\cK),\cG\langle d\rangle)\\
&\to \Hom^0(q_Y^*g_!(f^*\cF\otimes\cK),q_Y^*\cG\langle d\rangle)\\
&\isom\Hom^0(\tilg_!q^*(f^*\cF\otimes \cK),q_Y^*\cG\langle d\rangle)\\
&\isom \Hom^0(\tilg_!(\tilf^*q_X^*\cF\otimes q^*\cK),q_Y^*\cG\langle d\rangle)\\
&=\Cor_{\tilC,q^*\cK}(q_X^*\cF,q_Y^*\cG\langle d\rangle)
\end{split}.\]
\end{defn}

Now we introduce push-forward functoriality as defined in \cite[Section 4.3]{FYZ3}.

\begin{defn}\thlabel{ccpush}
    When the right square in \eqref{functdiag} is pushable in the sense of  \autoref{pushabledef} and the vertical maps are fp, we say that $q$ is pushable. In this case, we define the \emph{push-forward of cohomological correspondence} along $q$ as follows: For each cohomological correspondence \[c\in\Cor_{\tilC,\tilK}(\widetilde{\cF},\widetilde{\cG}\langle d \rangle)\] where $\widetilde{\cF}\in\Shv(\tilX)$ and $\widetilde{\cG}\in\Shv(\tilY)$, we define its push-forward along $q$ to be the element \[q_!(c)\in\Cor_{C,\cK}(q_{X,!}\widetilde{\cF},q_{Y,!}\widetilde{\cG}\langle d \rangle)\] defined as the image of $c$ along the map \[\begin{split}
    \Cor_{\tilC,\tilK}(\widetilde{\cF},\widetilde{\cG}\langle d\rangle)&=\Hom^0(\tilg_!(\tilf^*\widetilde{\cF}\otimes \tilK),\widetilde{\cG}\langle d\rangle) \\
    &\to \Hom^0(q_{Y,!}\tilg_!(\tilf^*\widetilde{\cF}\otimes q^*\cK),q_{Y,!}\widetilde{\cG}\langle d\rangle)\\
    &\isom \Hom^0(g_!q_!(\tilf^*\widetilde{\cF}\otimes q^*\cK),q_{Y,!}\widetilde{\cG}\langle d\rangle)\\
    &\isom \Hom^0(g_!(q_!\tilf^*\widetilde{\cF}\otimes \cK),q_{Y,!}\widetilde{\cG}\langle d\rangle)\\
    &\to \Hom^0(g_!(f^*q_{X,!}\widetilde{\cF}\otimes \cK),q_{Y,!}\widetilde{\cG}\langle d\rangle)\\
    &=\Cor_{C,\cK}(q_{X,!}\widetilde{\cF},q_{Y,!}\widetilde{\cG}\langle d \rangle) 
    \end{split}\]
\end{defn} where the fifth map uses the base change for pushable square \eqref{pushablebc}.

\subsubsection{Composition}\label{cccomp}
Consider the following diagram of correspondences of stacks with leftwards maps fp
\begin{equation}
    \begin{tikzcd}
    && E \ar[dr, "h'"] \ar[dl, "g'"'] && \\
         &D \ar[dr, "h"] \ar[dl, "i"'] && C \ar[dr, "f"] \ar[dl, "g"']& \\
        Z&& Y && X
    \end{tikzcd} 
\end{equation}
where the square is Cartesian. Given $\cK_C\in\Shv(C)$, $\cK_D\in\Shv(D)$, $\cF\in\Shv(X)$, $\cG\in\Shv(Y)$, $\cH\in\Shv(Z)$. Define $\cK_E:=g'^*\cK_D\otimes h'^*\cK_C$. For every $d_1,d_2\in\ZZ$, there is an obviously defined map 
\begin{equation}
    \Cor_{D,\cK_D}(\cG,\cH\langle d_2\rangle)\otimes\Cor_{C,\cK_C}(\cF,\cG\langle d_1\rangle) \to \Cor_{E,\cK_E}(\cF,\cH\langle d_1+d_2\rangle)
\end{equation} which would be denoted
\begin{equation}
    c_2\otimes c_1\mapsto c_2\circ c_1.
\end{equation} This is what we call the composition of cohomological correspondences.

\subsubsection{Compatibilities}\label{cccompatibilities}
In this section, we state all the compatibilities between the three operations on cohomological correspondences: push-forward \thref{ccpush}, pull-back \thref{ccpull}, and composition in \autoref{cccomp}. There is a compatibility for each pair of the three operations above, which gives us six different compatibilities. All of them except the base change compatibility between push and pull are easy to state and prove given \thref{pushablebccomp}. We only state the base change compatibility and leave the routine proof to the reader.

\begin{prop}\thlabel{ccbc}
Suppose that we have a commutative square of correspondences of stacks
\[\begin{tikzcd}
	& {Y_4} && {C_4} && {X_4} \\
	{Y_2} && {C_2} && {X_2} \\
	& {Y_3} && {C_3} && {X_3} \\
	{Y_1} && {C_1} && {X_1}
	\arrow["{p_{Y,42}}"', from=1-2, to=2-1]
	\arrow["{p_{Y,43}}"{pos=0.2}, dashed, from=1-2, to=3-2]
	\arrow["{g_4}"', from=1-4, to=1-2]
	\arrow["{f_4}", from=1-4, to=1-6]
	\arrow["{p_{C,42}}"'{pos=0.7}, from=1-4, to=2-3]
	\arrow["{p_{C,43}}"{pos=0.2}, dashed, from=1-4, to=3-4]
	\arrow["{p_{X,42}}"', from=1-6, to=2-5]
	\arrow["{p_{X,43}}", from=1-6, to=3-6]
	\arrow["{p_{Y,21}}"', from=2-1, to=4-1]
	\arrow["{g_2}"{pos=0.8}, from=2-3, to=2-1]
	\arrow["{f_2}"'{pos=0.2}, from=2-3, to=2-5]
	\arrow["{p_{C,21}}"{pos=0.2}, from=2-3, to=4-3]
	\arrow["{p_{X,21}}"{pos=0.2}, from=2-5, to=4-5]
	\arrow["{p_{Y,31}}"{pos=0}, dashed, from=3-2, to=4-1]
	\arrow["{g_3}"'{pos=0.8}, dashed, from=3-4, to=3-2]
	\arrow["{f_3}"{pos=0.2}, dashed, from=3-4, to=3-6]
	\arrow["{p_{C,31}}"{pos=0.1}, dashed, from=3-4, to=4-3]
	\arrow["{p_{X,31}}", from=3-6, to=4-5]
	\arrow["{g_1}", from=4-3, to=4-1]
	\arrow["{f_1}"', from=4-3, to=4-5]
\end{tikzcd}\]
such that all three left and right faces of the cubes are Cartesian squares and $g_i$ are of finite presentation for all $i$. Given any 
\begin{itemize}
    \item $\cF_3\in\Shv(X_3)$, $\cG_3\in\Shv(Y_3)$, $\cK_1\in\Shv(C_1)$ $\cK_3\in\Shv(C_3)$,
    \item $\a_{31}:p_{C,31}^* \cK_1\to \cK_3$,
    \item integer $d$,
\end{itemize} 
we set 
\begin{itemize}
    \item $\cF_1=p_{X,31,!}\cF_3$, $\cF_4=p_{X,43}^*\cF_3$, $\cF_2=p_{X,21}^*\cF_1\cong p_{X,42,!}\cF_4$ and similarly for $\cG_i$;
    \item $\cK_2=p_{C,21}^*\cK_1$ and $\cK_4=p_{C,43}^*\cK_3$;
    \item $\a_{42}=p_{C,43}^*\a_{31}:p_{C,42}^*\cK_{2}\to \cK_4$.
\end{itemize}
If $p_{C,31}$ and $p_{C,42}$ are pushable as in \thref{ccpush}, and $p_{C,21}$ and $p_{C,43}$ are Cartesian pullable as in \thref{ccpull}, then we have an identity 
\[p_{C,21}^*\circ p_{C,31,!}=p_{C,42,!}\circ p_{C,43}^*:\Cor_{C_3,\cK_3}(\cF_3,\cG_3\langle d\rangle)\to \Cor_{C_2,\cK_2}(\cF_2,\cG_2\langle d\rangle).\]
\end{prop}
\begin{proof}
    The only slightly non-obvious part is provided by \thref{pushablebccomp}.
\end{proof}

\subsection{Geometric setup}\label{geometricsetup}
We now introduce the geometric objects involved in this article. We remind the reader again that all the stacks in this section will be non-derived. Experienced readers can skip this part and come back for notations (see \autoref{modulinotations} for a summary of notations).

\begin{warning}\thlabel{warninghalfspin}
To simplify the notations, we do not include the half-spin twist as in \cite[\S10]{BZSV} in the definition of various moduli spaces. This makes all our constructions different from the normalization in \textit{loc.cit} unless the $\Ggr$-action on $X$ is trivial. One can easily adapt these normalizations.
\end{warning}

\subsubsection{Moduli spaces}\label{modulispaces}
Throughout this section, we fix an affine smooth $G$-variety $X=\Spec\!A$ together with a closed embedding $X\sub V$ where $V$ is a finite-dimensional representation of $G$ regarded as a $G$-variety. 

We start from the local theory. We use $L^+X\in\Sch_{\qcqs}$ to denote the jet space of $X$ which is defined as \[L^+X:\Aff^{\op}\to\Spc\]\[L^+X(R)=X(R[[t]]).\] We use $LX \in\Stack$ to denote the loop space defined as \[LX:\Aff^{\op}\to\Spc\] \[LX(R)=X(R(\!(t)\!)).\] The loop space $LX$ carries a natural cotruncated presentation $LX=\varinjlim_n L^{\geq-n}X$ where the closed subfunctor $L^{\geq-n}X\sub LX$ is defined by requiring the images of degree one generators of $\cO(V)$ under the map $\cO(V)\to A\to R(\!(t)\!)$ factor through $t^{-n}R[[t]]\sub R(\!(t)\!)$.

When we take $X=G$, the construction above gives us the jet group scheme $L^+G$ and loop group $LG$.

We use $\Aut(D)\in\Sch$ to denote the group scheme of automorphisms of the formal disk $D=\Spec\!\F_q[[t]]$ preserving the center of the disc, i.e., those automorphisms which are $\id$ mod $t$. We have $\Aut(D)=\Aut(D)^{\mathrm{u}}\rtimes \Gm$, where $\Aut(D)^{\mathrm{u}}$ is the pro-unipotent radical of $\Aut(D)$. There is a natural action of $\Aut(D)$ on $L^+X$ and $LX$.

Throughout this section, we fix a geometrically connected smooth projective curve $C\in\Sch$. There is a natural map of stacks $\Coor: C\to B\!\Aut(D)$ given by the canonical \'etale $\Aut(D)$-torsor of local coordinates of $C$.

Now we come to the semi-local version of the constructions above. Following \cite{KV} (See also \cite[\S 2]{cpsii}), we now define the jet and loop spaces over the curve $C$. For each finite set $I$, we denote the (multi)jet scheme of $X$ over $C^I$ by $(L^+X)_{C^I}\in\Sch$ whose $S$-points for $S\in\Aff$ are given by \[(L^+X)_{C^I}(S)=\{c_I\in C^I(S) ~\textup{and}~ D_{c_I}\to X\},\] where $D_{c_I}$ is the formal completion of $C\times S$ along the graph  $\Gamma_{c_I}$ of $c_I$. We use $l_{\sloc,I}:(L^+X)_{C^I}\to C^I$ to denote the map remembering the legs.

We denote the (multi)loop space of $X$ over $C^I$ by $(LX)_{C^I}\in\Stack$ whose $S$-points for $S\in\Aff$ are given by \[(LX)_{C^I}(S)=\{c_I\in C^I(S) ~\textup{and}~ \Dc_{c_I}\to X\},\] where $\Dc_{c_I}$ is the punctured disc along the graph $\Gamma_{c_I}$ of $c_I$ defined in \cite[\S 2.6]{cpsii}. We use $l_{\sloc,I}:(LX)_{C^I}\to C^I$ to denote the map remembering the legs.

For each subset $J\sub I$ we also consider the closed substack $(LX)_{C^{J\sub I}}\sub (LX)_{C^I}$ defined by \[(LX)_{C^{J\sub I}}(S)=\{c_I\in C^I(S) ~\textup{and}~ D_{c_I}\bs\Gamma_{c_J}\to X\}.\]

Finally, we define moduli spaces involved in global theory. We use $\Bun_G\in\Stack$ to denote the moduli of $G$-bundles on the curve $C$ and $\Bun_G^X\in\Stack$ to denote the moduli of $G$-bundles on $C$ together with a regular $X$-section. More specifically, the $S$-points of $\Bun_G^X$ for $S\in\Aff$ is given by \[\Bun_G^X(S)=\{p:\cE\to C\times S~\textup{an \'etale $G$-torsor on $C\times S$ and}~s:C\times S\to X\times^G\cE~\textup{a section of $p$}\} .\]

For each finite set $I$, we use $\Bun_{G, I}^X\in\Stack$ to denote the moduli of $G$-bundles with a rational $X$-section which is regular away from legs indexed by $I$. Precisely, for each $S\in\Aff$, we define the $S$-points of $\Bun_{G, I}^X$ to be \[\Bun_{G,I}^X(S)=\{(c_I,\cE,s)|c_I\in C^I(S), \cE\in\Bun_G(S), s ~\textup{is a section of $X\times^{G}\cE$ over $C\times S\bs \Gamma_{c_I}$} \}.\]

Finally, for each subset $J\sub I$, we also have $\Bun^X_{G,J\sub I}$ defined via the fiber product \[\begin{tikzcd}
    \Bun^X_{G,J\sub I} \ar[r] \ar[d] & \Bun_{G,I}^X \ar[d] \\
    (L^+X/L^+G)_{C^{I-J}} \ar[r] & (LX/L^+G)_{C^{I-J}}\rcart
\end{tikzcd}.\]

Now we introduce various versions of Hecke stacks. We define the \emph{local Hecke stack} to be the \'etale stack quotient $\hk^{\loc}_G= (L^+G\bs LG/L^+G)/
\Aut(D)\in\Stack$, which is also naturally a groupoid acting on $B(L^+G\rtimes\Aut(D))$ with the target and source morphisms denoted by $\rh_{\loc}$ and $\lh_{\loc}$. We use $\hk_G^{\Xp, \loc}\in\Stack$\footnote{We use $\Xp$ rather than $X$ here to emphasize that the $X$-section is \emph{not} required to be regular for either $G$-torsors over the disc.} to denote the fiber product \[\begin{tikzcd}
    \hk_G^{\Xp,\loc} \ar[r]\ar[d] & LX/(L^+G\rtimes\Aut(D))\ar[d] \\
    \hk_G^{\loc} \ar[r, "\rh_{\loc}"] & B(L^+G\rtimes\Aut(D)) \rcart
\end{tikzcd}.\]

More explicitly, the stack $\hk_G^{\Xp,\loc}$ is involved in the following correspondence:
\begin{equation}\label{localcordiag}
    \begin{tikzcd}
    &\hk_G^{\Xp,\loc}\ar[d, "\sim"]&\\
        &LX\times^{L^+G\rtimes\Aut(D)}LG\rtimes\Aut(D)/L^+G\rtimes\Aut(D)\ar[dl, "\lh^X_{\loc}"]\ar[dr,"\rh^X_{\loc}"']& \\
        LX/L^+G\rtimes\Aut(D)&&LX/L^+G\rtimes\Aut(D)
    \end{tikzcd}
\end{equation}
where $\rh^X_{\loc}(x,g)=x$ and $\lh^X_{\loc}(x,g)=xg$ for $x\in LX$ and $g\in LG\rtimes\Aut(D)$.

For the finite set $I=[r]$, we define the local iterated Hecke stack to be \[\thk^{\loc}_{G,I}:=\hk^{\loc}_G\times_{B(L^+G\rtimes\Aut(D))}\cdots\times_{B(L^+G\rtimes\Aut(D))} \hk^{\loc}_G.\] Similarly, we have the version with $X$-section given by the fiber product \[\begin{tikzcd}
    \thk_{G,I}^{\Xp,\loc} \ar[r]\ar[d] & LX/(L^+G\rtimes\Aut(D))\ar[d] \\
    \thk_{G,I}^{\loc} \ar[r, "\rh_{\loc}"] & B(L^+G\rtimes\Aut(D))\rcart
\end{tikzcd}.\]

For each finite set $I$, we define the \emph{semi-local Hecke stack} over $C^I$ to be $\hk^{\sloc}_{G,I}\in\Stack$ with $S$-points given by \[\hk^{\sloc}_{G,I}(S)=\{(c_I,\cE_0,\cE_1,\phi)|c_I\in C^I(S), \cE_0,\cE_1~\textup{are $G$-bundles on $D_{c_I}$},\phi:\cE_0|_{\Dc_{c_I}}\isom\cE_1|_{\Dc_{c_I}} \}.\] We denote the two maps remembering the $G$-bundles by $\rh_{I,\sloc}$ and $\lh_{I,\sloc}$.

For each subset $J\sub I$, we define the \emph{$J$-th partial (semi-local) Hecke stack} over $C^I$ to be $\hk^{\sloc}_{G,J\sub I}\in\Stack$ with $S$-points given by \[\hk^{\sloc}_{G,J\sub I}(S)=\{(c_I,\cE_0,\cE_1,\phi)\in \hk^{\sloc}_{G,I}|c_I\in C^I(S), \cE_0,\cE_1~\textup{are $G$-bundles on $D_{c_I}$},\phi:\cE_0|_{D_{c_{I}}\bs\Gamma_{c_J}}\isom\cE_1|_{D_{c_{I}}\bs\Gamma_{c_J}} \}.\]  We denote the two maps remembering the $G$-bundles by $\rh_{J\sub I,\sloc}$ and $\lh_{J\sub I,\sloc}$. For each element $x\in I$, we denote $\hk^{\sloc}_{G,x\in I}:=\hk^{\sloc}_{G,\{x\}\subset I}$.

Moreover, given a partition $I=J_1\sqcup\cdots \sqcup J_n$ (which will be denoted by $J_{\bullet}$ later), we have the \emph{iterated semi-local Hecke stack} $\thk^{\sloc}_{G,J_\bullet}\in\Stack$ over $C^I$ defined as\[\thk^{\sloc}_{G,J_{\bullet}}:=\hk^{\sloc}_{G,J_1\sub I}\times_{B(L^+G)_{C^I}}\cdots\times_{B(L^+G)_{C^I}}\hk^{\sloc}_{G,J_n\sub I}.\] We denote the two maps remembering the (right-most and left-most) $G$-bundles by $\rh_{J_{\bullet},\sloc}$ and $\lh_{J_{\bullet},\sloc}$.

Similar to the local setting, we can define $\hk^{\Xp,\sloc}_{G,I},\hk^{\Xp,\sloc}_{G,J\sub I},\thk^{\Xp,\sloc}_{G,J_{\bullet}}$ by adding a $X$-section over the punctured disc $D_{c_I}$ to the original moduli problem of the Hecke stack. We also consider the closed substack $\hk^{\Xp_{J\sub I},\sloc}_{G,J\sub I}\sub \hk^{\Xp,\sloc}_{G,J\sub I}$ by requiring that the $X$-section is regular on $D_{c_I}\bs\Gamma_{c_J}$ throughout the modification.

We also have three versions of global Hecke stacks $\hk^{\glob}_{G,I},\hk^{\glob}_{G,J\sub I},\thk^{\Xp,\glob}_{G,J_{\bullet}}$ and the version with rational $X$-section $\hk^{\Xp,\glob}_{G,I},\hk^{\Xp,\glob}_{G,J\sub I},\thk^{\Xp,\glob}_{G,J_{\bullet}}, \hk^{\Xp_{J\sub I},\glob}_{G,J\sub I}$.

\subsubsection{Sheaves and categories}\label{shvandcat}
Consider \[\Sat_{G}:=\Shv_*(L^+G\bs LG/L^+G)\] and \[\Sat_{G,\hbar}:=\Shv_*(\hk^{\loc}_G)\] where the cotruncated presentation of $L^+G\bs LG/L^+G$ is given as usual. These categories are the usual (monoidal) Satake categories without and with loop rotation. We also define $\Sat_G^{\nai}:=D\Sat_G^{\he}$ (with respect to perverse $t$-structure). Then the classical geometric Satake equivalence proved in \cite{MV} gives an equivalence of symmetric monoidal categories $\Sat_G^{\nai}\isom\Rep(\Gc)$, and the derived Satake equivalence proved in \cite{BF} gives us equivalence of monoidal categories $\Sat_{G}^{\ren}\isom\QCoh(\frgc^{*\shear}/\Gc)$\footnote{Here by $\Sat_{G}^{\ren}$ we mean the renormalization such that constructible sheaves become compact.} and $\Sat_{G,\hbar}^{\ren}\isom \Mod^{\Gc}(U_{\hbar}(\frgc^{\shear}))$, where the shearing $\frgc^{*\shear}=\frgc^*[2]$ is defined in \autoref{shear} in which the $\Ggr$-action on $\frgc$ has weight $-2$. Note that there is a natural equivalence of monoidal categories $D\Sat_{G,\hbar}^{\he}\isom D\Sat_G^{\he}=\Sat_G^{\nai}$. This allows us to view $\Rep(\Gc)\isom\Sat_G^{\nai}$ as full subcategories of both $\Sat_G$ and $\Sat_{G,\hbar}$. For each $V\in\Rep(\Gc)^{\he,\om}$, we write $\IC_V$ to be the corresponding element in $\Sat_G$ (or $\Sat_{G,\hbar}$). By our convention \autoref{icsheaf}, the sheaf $\IC_V$ has parity congruence to $\langle 2\r,\l_V\rangle$ where $\l_V\in X_*(T)^+$ is the highest weight of $V$.

Note that there are natural maps from various Hecke stacks defined in \autoref{modulispaces} to $\hk^{\loc}_G$ (e.g. $\pi:\hk^{\Xp,\loc}_G\to \hk^{\loc}_G$, $f_{x,\sltol}^{\hk}:\hk^{\sloc}_{G, I}\to\hk^{\loc}_G$, etc). For each $\cK\in\Sat_{G,\hbar}$, we still use $\cK$ or $\cK_i$ to denote the pull-back of $\cK$ to the corresponding Hecke stack. In particular, for each $V\in\Rep(\Gc)^{\he,\om}$, we get sheaves $\IC_V$ ($\IC_{V,x}$ for $x\in I$) on various Hecke stacks. Moreover, for $V^I\in\Rep(\Gc^I)^{\he,\om}$, we get sheaves $\IC_{V^I}$ on various Hecke stacks via pull-back and convolution.

Define $\Sat_{X}:=\Shv_*(LX/L^+G)$\footnote{The category $\Sat_X$ is the same as the large category of safe $*$-sheaves $\mathrm{SHV}_{s}^*(X_F/G_{\cO})$ defined in \cite[\S 7.2.3, \S B.7]{BZSV}, see \thref{COMPAREshvsafebzsv}.} and $\Sat_{X,\hbar}:=\Shv_*(LX/L^+G\rtimes \Aut(D))$. The category $\Sat_X$ is naturally a module over $\Sat_G$ and the category $\Sat_{X,\hbar}$ is natural a module over $\Sat_{G,\hbar}$. Explicitly, the module structure is given by the following: For each $\cF\in\Sat_{X,\hbar}$ and $\cK\in\Sat_{G,\hbar}$ we define \[\cK*\cF:=\lh^X_{\loc,!}(\rh_{\loc}^{X,*}\cF\otimes\cK)\in\Sat_{X,\hbar}\] where the maps $\lh^X_{\loc}$, $\rh^X_{\loc}$ are defined in \eqref{localcordiag}. Note that for each leg $x\in I$, similar constructions define actions of $\Sat_{G,\hbar}$ on \[\sC=\Shv(\Bun_G\times C^I),\Shv(\Bun_{G,I}^X),\Shv((LX/L^+G)_{C^I}),\] and for $\cF\in\sC$ and $\cK\in\Sat_{G,\hbar}$, we denote the action at leg $x$ by \begin{equation}\label{boxstar}\cK\boxstar_x \cF\in\sC.\end{equation}

For $V^I\in\Rep(\Gc)$, we define $T_{V^I}:=\lh_{I,\glob,!}(\rh_{I,\glob}^*(-)\otimes\IC_{V^I}):\Shv(\Bun_G\times C^I)\to\Shv(\Bun_G\times C^I)$. These are the Hecke operators considered in \cite{arinkin2022stacklocalsystemsrestricted}. By our convention, we have $\IC_{V}\boxstar-\cong T_{c^*V}$ where $c:\Gc\to \Gc$ is the Cartan involution.

\subsubsection{Unit object and the Plancherel algebra}\label{unitandplan}
The natural map $i_{\loc}:L^+X/L^+G\rtimes\Aut(D)\to LX/L^+G\rtimes\Aut(D)$ defines an object \[\d^X_{\loc}:=i_{\loc,*} \uk\in\Sat_{X,\hbar}.\] Consider similar natural maps $i_{I,\sloc}:(L^+X/L^+G)_{C^I}\to (LX/L^+G)_{C^I}$ and $i_{I,\glob}:\Bun_G^X\times C^I\to\Bun^X_{G,I}$, we can define \[\d^X_{I,\sloc}:=i_{I,\sloc,*}\uk\in\Shv((LX/L^+G)_{C^I})\]
\[\d^X_{J\sub I,\sloc}:=i_{J\sub I,\sloc,*}\uk\in\Shv((LX/L^+G)_{C^{J\sub I}})\]
and \[\d^X_{I,\glob}:=i_{I,\glob,*}\uk\in\Shv(\Bun^X_{G,I})\]
\[\d^X_{J\sub I,\glob}:=i_{J\sub I,\glob,*}\uk\in\Shv(\Bun^X_{G,J\sub I}).\]

Regarding $\Sat_{X,\hbar}\in\Mod(\Sat_{G,\hbar})$, from the object $\d^X_{\loc}$ we can construct its inner endomorphism algebra object \begin{equation}\label{noncommpl} \PL_{X,\hbar}:=\iHom_{\Sat_{G,\hbar}}(\d^X_{\loc},\d^X_{\loc})\in\Sat_{G,\hbar}\end{equation} which is characterized by the property that for each object $\cK\in \Sat_{G,\hbar}$ there exists a functorial isomorphism \[\Hom_{\Sat_{G,\hbar}}(\cK,\iHom_{\Sat_{X,\hbar}}(\d^X_{\loc},\d^X_{\loc}))=\Hom_{\Sat_{G,\hbar}}(\cK*\d^X_{\loc},\d^X_{\loc}).\] This is the (non-commutative) Plancherel algebra introduced in \cite[\S 8.5]{BZSV}. Similarly, we can regard $\d^X_{\loc}\in\Sat_X\in\Mod(\Sat_G)$ and define the (commutative) Plancherel algebra \[\PL_X:=\iHom_{\Sat_{G}}(\d^X_{\loc},\d^X_{\loc})\in\Sat_{G}.\]

By replacing the monoidal categories $\Sat_{G,\hbar}$ and $\Sat_G$ by their common monoidal subcategory $\Sat_G^{\nai}\isom \Rep(\Gc)$, we can regard both Plancherel algebras as objects $\PL_{X,\hbar},\PL_{X}\in\Rep(\Gc)$ by applying the right adjoint of the natural inclusion $\Sat_G^{\nai}\to\Sat_{G,\hbar}$ (and $\Sat_G^{\nai}\to\Sat_G$). Note that $\PL_{X,\hbar}$ is naturally equipped with a map from the algebra (with trivial $\Gc$-action) $\Gamma(B\Aut(D),\uk)\isom \Gamma(B\Gm,\uk)\isom k[\hbar]$ where the last isomorphism is normalized in  \autoref{cohbg}.

The commutative and non-commutative Plancherel algebras are related by $\PL_X\isom\PL_{X,\hbar}\otimes_{k[\hbar]}k.$ We have the following conjecture from \cite{BZSV}:

\begin{conj}\thlabel{bzsvloc}\cite[Conjecture
\,8.1.8\&8.5.2]{BZSV}
    Suppose the $G$-Hamiltonian space $T^*X$ is hyperspherical, then the following properties hold:
    \begin{enumerate}
        \item \label{bzsvlocflat} The non-commutative Plancherel algebra $\PL_{X,\hbar}$ is flat over $k[\hbar]$;
        \item \label{bzsvloccom} The Plancherel algebra without loop rotation $\PL_X$ is commutative;
        \item \label{bzsvlocdual} There is an isomorphism of $\Gc$-equivariant commutative 2-shifted Poisson algebras $\PL_X\isom \cO(\Mc^{\shear})$. Moreover, there is a canonical isomorphism $\PL_{X,\hbar}\isom \cO_{\hbar}(\Mc^{\shear})$ between quantization of both Poisson algebras. Here $\Mc$ is the hyperspherical dual of $T^*X$ introduced in \autoref{hdual}, and the shearing $\Mc^{\shear}$ is defined in \autoref{shear}.\footnote{When $X$ is not homogeneous, one needs to twist the $\Ggr$-action on $\Mc$ or the $\Sat_{G,\hbar}$-module structure on $\Sat_{X,\hbar}$. We refer to \cite[\S8]{BZSV} for details.}
    \end{enumerate}
\end{conj}

\begin{remark}
    We only state part of the conjectures made in \cite{BZSV}, which is relevant to our article.
\end{remark}

\subsubsection{Period sheaf}\label{periodsheaf}
Consider the map $\pi:\Bun_G^X\to\Bun_G$, we define the $X$-period sheaf to be \[\cP_X:=\pi_!\uk_{\Bun_G^X}.\]

\begin{remark}
    As we have mentioned in \thref{warninghalfspin}, when we consider hyperspherical $T^*X$ (see  \autoref{Gvar} for a setup), if the $\Ggr$-action on $X$ is trivial (equivalently speaking: when $X$ is $G$-homogeneous), the object $\cP_X$ is the unnormalized $!$-period sheaf introduced in \cite[\S10.3]{BZSV}. When the $\Ggr$-action on $X$ is non-trivial, this is different from the unnormalized period sheaf in \textit{loc.cit} by a half-spin twist. One can easily adapt these normalizations when needed.
\end{remark}

\subsubsection{Summary of notations}\label{modulinotations}
The notations in this section are complicated, and we will use them throughout the article. We summarize the rules that we use to design our notations here for the convenience of the reader.

We use the word \emph{local} to indicate moduli of something over the disc $D$ (or punctured disc $\Dc$). We use \emph{semi-local} to indicate moduli of something over the disc $D_c$ (or punctured disc $\Dc_c$) for a moving point $c\in C$ (or several points $c\in C^I$). We use the word \emph{global} to indicate moduli of something over the curve (or punctured curve, possibly several copies of the curve). Whenever we write a moduli (or map between moduli) that is local/semi-local/global, we would add a subscript (or sometimes superscript) $\loc,\sloc,\glob$ to indicate this whenever it is unclear from other notations. We sometimes use the subscript $\sltol,\gtosl$ to indicate a morphism is mapping from semi-local moduli to local moduli, or global moduli to semi-local moduli.

We always use $\hk$ to indicate Hecke-type stacks (i.e., do modification one time, but can happen at several legs) and $\thk$ to indicate iterated Hecke-type stacks (i.e, do modification several times). We use $\lh$ and $\rh$ to denote the map sending a Hecke modification to its left-most object and right-most object, respectively. We use the word \emph{leg} to describe the point at which a Hecke modification happens or an $X$-section is not required to be regular. We use a superscript $\hk$ (or $\thk$) on a morphism to indicate that the morphism is between Hecke-type stacks.

We use a superscript $X$ to indicate a moduli of objects equipped with an $X$-section (the $X$-section may or may not be regular, depending on the context). When we use this for Hecke-type stacks, we sometimes write $\Xp$ to emphasize that the $X$-section is not required to be regular (This is to avoid confusion with the relative Grassmanian $\Gr^X_G$ in which the $X$-section is regular on both sides of the modification which will be introduced in  \autoref{dfc}).

We use $I$ to denote the index set of legs. We use the letter $x$ to indicate an element $x\in I$. When we have subscript $J\sub I$ or $x\in I$, we mean the Hecke modification happens only at $J$ (or i.e.$\{x\}$) or the $X$-section is rational at $J$ (or $\{x\}$).

We always use $l$ to denote a map remembering all the legs. We use $f$ to denote maps between moduli of different levels (i.e., local/semi-local/global). We use $i$ to denote maps from moduli with regular $X$-section to moduli with meromorphic $X$-section. We use $r_x$ to denote a morphism remembering only information at the leg $x$. We use $\pi$ to denote maps forgetting the $X$-section. We use $w$ to denote the maps forgetting the legs. We use $\D$ to denote diagonal-type maps.

When we write $\IC_V$, $\IC_{V^I}$, we may mean sheaves on any Hecke-type stack (which would be clear from the context) to which the intersection complexes in $\Sat_G$ can be pulled
back as kernel sheaves for some cohomological correspondence. We sometimes write $\IC_{V,x}$ to indicate that the sheaf is pulled back from leg $x$.

We use $*$ to denote the convolution in the local categories, and $\boxstar$ to denote the action of objects in the local categories on global/semi-local categories (or the action of kernel sheaves on some Hecke-type moduli by pull-tensor-push). We sometimes use $\boxstar_x$ to indicate the action happens at leg $x$.

\subsection{Construction of special cohomological correspondence}\label{conscc}
In this section, we introduce the local-to-global construction of special cohomological correspondences. The main output of this section is \thref{globccdef}.

For notations on the moduli spaces and morphisms between them, we refer to \autoref{modulispaces} (see  \autoref{modulinotations} for a summary).
\begin{defn}\thlabel{loccordef}
    A \emph{local special cohomological correspondence of type $V\in\Rep(\Gc)^{\he,\om}$ of degree $d$} is an element \[c_V^{\loc}\in\Hom_{\Gc}(V\langle -d\rangle,\PL_{X,\hbar}).\] Equivalently, consider the diagram \eqref{localcordiag}, we can regard $c_V^{\loc}$ as a cohomological correspondence by the tautological identification \[\begin{split}\Hom_{\Gc}(V\langle-d\rangle,\PL_{X,\hbar})&=\Hom_{\Sat_{X,\hbar}}(\IC_V*\d^X_{\loc},\d^X_{\loc}\langle d\rangle)\\&=\Hom(\lh_{\loc,!}^X(\rh_{\loc}^{X,*}\d^X_{\loc}\otimes \IC_V),\delta_l^X\langle d\rangle)\\&= \Cor_{\hk^{\Xp,\loc}_G,\IC_V}(\d^X_{\loc},\d^X_{\loc}\langle d\rangle)\end{split}.\]
\end{defn}

Now we apply the pull-back functoriality in  \autoref{ccfunc} to produce special cohomological correspondences at the semi-local and global levels. Consider the following diagram: \begin{equation}
    \begin{tikzcd}
        \Bun^X_{G,\fone} \ar[d, "f_{\fone,\gtosl}" '] & \hk^{\Xp,\glob}_{G,\fone} \ar[d, "f^{\hk}_{\fone,\gtosl}"] \ar[l, "\lh^X_{\fone,\glob}" '] \ar[r, "\rh^X_{\fone,\glob}"] & \Bun^X_{G,\fone} \ar[d, "f_{\fone,\gtosl}"] \\
        (LX/L^+G)_C \lcart\ar[d, "f_{\fone,\sltol}" '] & \hk^{\Xp,\sloc}_{G,\fone} \ar[d, "f^{\hk}_{\fone,\sltol}"] \ar[l, "\lh^X_{\fone,\sloc}" '] \ar[r, "\rh^X_{\fone,\sloc}"] & (LX/L^+G)_C \ar[d, "f_{\fone,\sltol}"] \rcart\\
        LX/L^+G\rtimes\Aut(D) \lcart& \hk^{\Xp,\loc}_{G} \ar[l, "\lh^X_{\loc}" '] \ar[r, "\rh^X_{\loc}"] & LX/L^+G\rtimes\Aut(D)\rcart.
    \end{tikzcd}
\end{equation}

Note that both left squares are Cartesian (in fact, all squares are Cartesian), from the isomorphisms $f_{\fone,\sltol}^*\d^X_{\loc}\isom\d^X_{\sloc}$ and $f_{\fone,\gtosl}^*\d^X_{\sloc}\isom\d^X_{\glob}$, we get cohomological correspondences \begin{equation}
    c_V^{\sloc}:=f_{\fone,\sltol}^*c_V^{\loc}\in\Cor_{\hk^{\Xp,\sloc}_{G,\fone},\IC_V}(\d^X_{\sloc},\d^X_{\sloc}\langle d\rangle)
\end{equation}
and
\begin{equation}
    c_V^{\glob}:=f_{\fone,\gtosl}^*c_V^{\sloc}\in\Cor_{\hk^{\Xp,\glob}_{G,\fone},\IC_V}(\d^X_{\glob},\d^X_{\glob}\langle d\rangle).  
\end{equation}

Now we construct cohomological correspondences with legs indexed by a finite set $I$. To begin with, for any $x\in I$, consider the following diagram \begin{equation}
    \begin{tikzcd}
        (LX/L^+G)_C \arrow[dr, phantom, "\ulcorner", shift left=2, very near end] & \hk^{\Xp,\sloc}_{G,\fone}  \ar[l, "\lh^X_{\fone,\sloc}" '] \ar[r, "\rh^X_{\fone,\sloc}"] & (LX/L^+G)_C  \arrow[dl, phantom, "\urcorner", shift right =2, very near end]\\
        (LX/L^+G)_{C^{x\in I}} \ar[u, "r_{x,\sloc}"] \ar[d, "i_{x\in I,\sloc}" ']  & \hk^{\Xp_{x\in I},\sloc}_{G,x\in I} \ar[l, "\lh^{X_{x\in I}}_{x\in I,\sloc}" '] \ar[r, "\rh^{X_{x\in I}}_{x\in I, \sloc}"] \ar[u, "r_{x,\sloc}^{\hk}" '] \ar[d, "i_{x\in I,\sloc}^{\hk}"] & (LX/L^+G)_{C^{x\in I}} \ar[u, "r_{x,\sloc}" '] \ar[d, "i_{x\in I,\sloc}"] \\
        (LX/L^+G)_{C^I}  \lcart& \hk^{\Xp,\sloc}_{G,x\in I} \ar[l, "\lh^X_{x\in I,\sloc}" '] \ar[r, "\rh^X_{x\in I, \sloc}"] & (LX/L^+G)_{C^I} \rcart
    \end{tikzcd}
\end{equation}

Note that all the squares above are Cartesian, and we have natural isomorphism $r_{x,\sloc}^*\d^X_{\sloc}\isom\d^X_{x\in I,\sloc}$ and $\d^X_{I,\sloc}\isom i_{x\in I,\sloc,*}\d^X_{x\in I,\sloc}$. By applying the pull-back functoriality for the upper diagram and push-forward functoriality for the lower diagram, we get cohomological correspondences
\begin{equation}
    c_{V,x\in I}^{\sloc}:= r_{x,\sloc}^*c_V^{\sloc}\in \Cor_{\hk^{\Xp_{x\in I},\sloc}_{G,x\in I},\IC_{V,x}}(\d^X_{x\in I,\sloc},\d^X_{x\in I,\sloc}\langle d\rangle )
\end{equation}
and
\begin{equation}\label{partialsloc}
    c_{V,x}^{\sloc}:=i_{x\in I,\sloc,!}c_{V,x\in I}^{\sloc}\in\Cor_{\hk^{\Xp,\sloc}_{G,x\in I},\IC_{V,x}}(\d^X_{I,\sloc},\d^X_{I,\sloc}\langle d\rangle).
\end{equation}

We also have a global counterpart of the picture above: Consider the diagram 
\begin{equation}
    \begin{tikzcd}
        \Bun^X_{G,\fone}  \rucart& \hk^{\Xp,\glob}_{G,\fone}  \ar[l, "\lh^X_{\fone,\glob}" '] \ar[r, "\rh^X_{\fone,\glob}"] & \Bun^X_{G,\fone}  \lucart\\
        \Bun^X_{G,x\in I} \ar[u, "r_{x,\glob}"] \ar[d, "i_{x\in I,\glob}" ']  & \hk^{\Xp_{x\in I},\glob}_{G,x\in I} \ar[l, "\lh^{X_{x\in I}}_{x\in I,\glob}" '] \ar[r, "\rh^{X_{x\in I}}_{x\in I, \glob}"] \ar[u, "r_{x,\glob}^{\hk}" '] \ar[d, "i_{x\in I,\glob}^{\hk}"] & \Bun^X_{G,x\in I} \ar[u, "r_{x,\glob}" '] \ar[d, "i_{x\in I,\glob}"] \\
        \Bun^X_{G,I} \lcart & \hk^{\Xp,\glob}_{G,x\in I} \ar[l, "\lh^X_{x\in I,\glob}" '] \ar[r, "\rh^X_{x\in I, \glob}"] & \Bun^X_{G,I}\rcart
    \end{tikzcd}.
\end{equation}
From the natural isomorphism $r_{x,\glob}^*\d^X_{\glob}\isom\d^X_{x\in I,\glob}$ and $\d^X_{I,\glob}\isom i_{x\in I,\glob,*}\d^X_{x\in I,\glob}$, we get cohomological correspondences
\begin{equation}
    c_{V,x\in I}^{\glob}:= r_{x,\glob}^*c_V^{\glob}\in \Cor_{\hk^{\Xp_{x\in I},\glob}_{G,x\in I},\IC_{V,x}}(\d^X_{x\in I,\glob},\d^X_{x\in I,\glob}\langle d\rangle )
\end{equation}
and
\begin{equation}\label{partialglob}
    c_{V,x}^{\glob}:=i_{x\in I,\glob,!}c_{V,x\in I}^{\glob}\in\Cor_{\hk^{\Xp,\glob}_{G,x\in I},\IC_{V,x}}(\d^X_{I,\glob},\d^X_{I,\glob}\langle d\rangle).
\end{equation}

From now on, we drop the subscript $J_{\bullet}$ in the notation for iterated Hecke stacks and always use the finest partition. We are going to construct global special cohomological correspondences and variants out of local special cohomological correspondences.

\begin{defn}\thlabel{loctovarious}
We define a \emph{local special cohomological correspondences datum} to be a tuple \[(I,V^I,\{c_{V^x}^{\loc}\}_{x\in I})\] where
\begin{itemize}
    \item $I$ is a finite set;
    \item $V^{\boxtimes I}\in\Rep(\Gc^I)^{\he,\om}$;
    \item For each $x\in I$, $c_{V^x}^{\loc}\in\Cor_{\hk^{\Xp,\loc}_G,\IC_{V^x}}(\d^X_{\loc},\d^X_{\loc}\langle d_x\rangle)$ is a local special cohomological correspondence of degree $d_x$ as in \thref{loccordef}.
\end{itemize}

Set $d_I=\sum_{x\in I}d_x$ to be the degree of $(I,V^I,\{c_{V^x}^{\loc}\}_{x\in I})$. Suppose $I=[r]$.

\begin{enumerate}
\item  We define the \emph{local composition} of $(I,V^I,\{c_{V^x}^{\loc}\}_{x\in I})$ to be the element
\begin{equation}\label{loccomp}
    c^{\loc}_{V^I}:=c^{\loc}_{V^1}\circ\cdots\circ c^{\loc}_{V^r}\in\Cor_{\hk^{\Xp,\loc}_G,\IC_{V^{\otimes I}}}(\d^X_{\loc},\d^X_{\loc}\langle d_I\rangle),
\end{equation}
\item  We define the \emph{semi-local special cohomological correspondence} of $(I,V^I,\{c_{V^x}^{\loc}\}_{x\in I})$ to be the element
\begin{equation}\label{sloccomp}
    c^{\sloc}_{V^I,I}:=c^{\sloc}_{V^1,1}\circ\cdots\circ c^{\sloc}_{V^r,r}\in\Cor_{\thk^{\Xp,\sloc}_{G,I},\IC_{V^{I}}}(\d^X_{I,\sloc},\d^X_{I,\sloc}\langle d_I\rangle)\isom \Cor_{\hk^{\Xp,\sloc}_{G,I},\IC_{V^{I}}}(\d^X_{I,\sloc},\d^X_{I,\sloc}\langle d_I\rangle);
\end{equation}
where the elements $c^{\sloc}_{V^x,x}$ are defined in \eqref{partialsloc};
\item  We define the \emph{global special cohomological correspondence} of $(I,V^I,\{c_{V^x}^{\loc}\}_{x\in I})$ to be the element
\begin{equation}\label{globcomp}
    c^{\glob}_{V^I,I}:=c^{\glob}_{V^1,1}\circ\cdots\circ c^{\glob}_{V^r,r}\in\Cor_{\thk^{\Xp,\glob}_{G,I},\IC_{V^{I}}}(\d^X_{I,\glob},\d^X_{I,\glob}\langle d_I\rangle)\isom \Cor_{\hk^{\Xp,\glob}_{G,I},\IC_{V^{I}}}(\d^X_{I,\glob},\d^X_{I,\glob}\langle d_I\rangle),
\end{equation}
where the elements $c^{\glob}_{V^x,x}$ are defined in \eqref{partialglob}.
\end{enumerate}
\end{defn}

We are finally able to define special cohomological correspondences for the period sheaf $\cP_X$. 

\begin{defn}\thlabel{globccdef}
Continuing the setting of \thref{loctovarious}, consider the diagram
\begin{equation}\label{globtoperioddiag}
    \begin{tikzcd}
        \Bun^X_{G,I} \ar[d, "\pi_I" ']  & \hk^{\Xp,\glob}_{G,I} \ar[l, "\lh^X_{I,\glob}" '] \ar[r, "\rh^X_{I, \glob}"] \ar[d, "\pi_I^{\hk}"] & \Bun^X_{G,I} \ar[d, "\pi_I"] \\
        \Bun_G\times C^I \lcart& \hk^{\glob}_{G,I}  \ar[l, "\lh_{I,\glob}" '] \ar[r, "\rh_{I, \glob}"] & \Bun_G\times C^I\rcart
    \end{tikzcd}
\end{equation}
where both squares are Cartesian. Note that we have natural isomorphism $\pi_{I,!}\d^{X}_{I,\glob}\isom \cP_X\boxtimes \uk_{C^I}$. 
\begin{enumerate}
\item We define the \emph{extended special cohomological correspondence} associated to $(I,V^I,\{c_{V^x}^{\loc}\}_{x\in I})$ to be the element
\begin{equation}\label{periodpush}
    c_{V^I,I}:=\pi_{I,!}c^{\glob}_{V^I,I}\in\Cor_{\hk^{\glob}_{G,I},\IC_{V^I}}(\cP_X\boxtimes \uk_{C^I},\cP_X\boxtimes \uk_{C^I}\langle d_I\rangle)
\end{equation}
where the element $c^{\glob}_{V^I,I}$ is defined in \eqref{globcomp};
\item Consider the map $w_I:\Bun_G^X\times C^I\to\Bun_G^X$. We define the \emph{special cohomological correspondence} associated to $(I,V^I,\{c_{V^x}^{\loc}\}_{x\in I})$ to be the element
\begin{equation}\label{removeleg}
    c_{V^I}:=[C^I]\circ c_{V^I,I}\circ \taut_{C^I}\in\Cor_{\hk^{\glob}_{G,I},\IC_{V^I}}(\cP_X,\cP_X\langle d_I-2|I|\rangle)
\end{equation}
where the cohomological correspondences $\taut_{C^I},[C^I]$ are introduced in \thref{tautcor},\thref{fundcor}.
\end{enumerate}
\end{defn}

\subsection{Main result}\label{mainresult3}
In this section, we state the so-called automorphic commutator relation \thref{geokolycor} (partially proved in \thref{geokolycorthm}). This is the main result in \autoref{commutatorrelations}.
\subsubsection{Local datum}
We now introduce the local special cohomological correspondence datum needed to formulate the commutator relation.

\begin{setting}\thlabel{commsetup}
    We fix a local special cohomological correspondence datum \[(I,V^I,\{c_{V^x}^{\loc}\}_{x\in I})\] as in \thref{loctovarious}, in which we take $I=[r]$.

    We fix $x\in [r-1]\sub I$. Consider $I':=I-\{x,x+1\}$ and $[I]:=[r-1]$. Denote $V=V_x$ and $W=V_{x+1}$. Set $d_V=d_x$ and $d_W=d_{x+1}$.
    
    Consider $\sw_{x,x+1}\in\Aut(I)$ to be the element swapping $x$ and $x+1$, and $\D_{x,x+1}\in\Hom(I,[I])$ to be the order-preserving map collapsing $x$ and $x+1$. These maps induce natural maps $\sw_{G,x,x+1}\in\Aut(BG^I)$, $\D_{G,x,x+1}\in\Hom(BG^{[I]},BG^I)$. We denote $V^{I,\sw}:=\sw_{G,x,x+1}^*V^I\in \Rep(\Gc^I)^{\he,\om}$ and $V^{[I]}:=\D_{G,x,x+1}^*V^I\in \Rep(\Gc^{[I]})^{\he,\om}$. 

    Assume \thref{genassum} (see below), we consider two more local special cohomological correspondence data \[(I, V^{I,\sw},\{c_{V^x}^{\loc}\}_{x\in I})\] and \[([I], V^{[I]},\{c_{V^x}^{\loc}\}_{x\in [I]})\] where the local special cohomological correspondences remain to be the old ones except we choose \begin{equation}\label{diagcor} 
c^{\loc}_{V^x\otimes V^{x+1}}\in\Cor_{\hk_G^{\Xp,\loc},\IC_{V^x\otimes V^{x+1}}}(\d^X_{\loc},\d^X_{\loc}\langle d_x+d_{x+1}-2\rangle).
\end{equation} to be the element in \itemref{genassum}{comassum}.

From these local cohomological correspondences, we can run the construction in  \autoref{conscc} and get \begin{equation}\label{cV}c_{V^I}\in \Cor_{\hk^{\glob}_{G, I},\IC_{V^I}}(\cP_X,\cP_X\langle d_I-2|I|\rangle ),\end{equation} \begin{equation}\label{cVsw} c_{V^{I,\sw}}\in \Cor_{\hk^{\glob}_{G , I},\IC_{V^{I, \sw}}}(\cP_X,\cP_X\langle d_I-2|I|\rangle ) ,\end{equation} and \begin{equation}\label{cVcom} c_{V^{[I]}}\in \Cor_{\hk^{\glob}_{G,[I]},\IC_{V^{[I]}}}(\cP_X,\cP_X\langle d_I-2|I|\rangle )\end{equation} defined in \eqref{removeleg}. Here we are using $d_{[I]}-2|[I]|=(d_I-2)-2(|I|-1)=d_I-2|I|$ in the last formula.

\end{setting}

\subsubsection{Fusion procedure}\label{cliffus}
Consider the diagram of correspondence
\begin{equation}
    \begin{tikzcd}
        \Bun_G \ar[d, "\id" '] & \hk^{\glob}_{G,[I]} \ar[d, "\D^{\hk}_{x,x+1}"] \ar[l, "\lh_{[I],\glob}" '] \ar[r, "\rh_{[I],\glob}"] & \Bun_G \ar[d, "\id"] \\
        \Bun_G & \hk^{\glob}_{G,I} \ar[l, "\lh_{I,\glob}" '] \ar[r, "\rh_{I,\glob}"] & \Bun_G
    \end{tikzcd}.
\end{equation}
Note that we have an isomorphism $\D^{\hk^*}_{x,x+1}\IC_{V^I}\isom \IC_{V^{[I]}}$. The push-forward functoriality in \thref{ccpush} gives us \begin{equation}\label{nolegdiagpush}
    \D^{\hk}_{x,x+1,!}c_{V^{[I]}}\in \Cor_{\hk^{\glob}_{G,I}, \IC_{V^I}}(\cP_X,\cP_X\langle d_I-2|I|\rangle ).
\end{equation}

\subsubsection{Commutator in the Plancherel algebra}
The natural multiplication map $m:\PL_{X,\hbar}\otimes \PL_{X,\hbar}\to\PL_{X,\hbar}$ on the (non-commutative) Plancherel algebra can be characterized as following: Given $V,W\in\Rep(\Gc)^{\he,\om}$ and \[c_V^{\loc}\in\Cor_{\hk^{\loc}_G,\IC_V}(\d^X_{\loc},\d^X_{\loc}\langle d_V\rangle),\] \[c_W^{\loc}\in\Cor_{\hk^{\loc}_G,\IC_W}(\d^X_{\loc},\d^X_{\loc}\langle d_W\rangle),\] the product $V\otimes W\langle -d_V-d_W\rangle \to \PL_{X,\hbar}\otimes\PL_{X,\hbar}\to  \PL_{X,\hbar}$ is defined by \[c_{V\boxtimes W}^{\loc}=c_V^{\loc}\circ c_W^{\loc} \in \Cor_{\hk^{\loc}_G,\IC_{V\otimes W}}(\d^X_{\loc},\d^X_{\loc}\langle d_V+d_W\rangle).\] Similarly, we have \[c^{\loc}_{W\boxtimes V}\in \Cor_{\hk^{\loc}_G,\IC_{W\otimes V}}(\d^X_{\loc},\d^X_{\loc}\langle d_V+d_W\rangle).\] Using the natural isomorphism \[\Cor_{\hk^{\loc}_G,\IC_{W\otimes V}}(\d^X_{\loc},\d^X_{\loc}\langle d_V+d_W\rangle)\isom \Cor_{\hk^{\loc}_G,\IC_{V\otimes W}}(\d^X_{\loc},\d^X_{\loc}\langle d_V+d_W\rangle), \] we can view both $c^{\loc}_{V\boxtimes W}$ and $c^{\loc}_{W\boxtimes V}$ as elements in the same vector space $\Cor_{\hk^{\loc}_G,\IC_{V\otimes W}}(\d^X_{\loc},\d^X_{\loc}\langle d_V+d_W\rangle).$ We denote \[[c_V^{\loc},c_W^{\loc}]:=c^{\loc}_{V\boxtimes W}-c^{\loc}_{W\boxtimes V}.\] We make the following assumption:
\begin{assumption}\thlabel{genassum}
The following two things hold:
\begin{enumerate}
    \item \label{flatassum} The multiplication-by-$\hbar$ map \[\cdot\hbar:\Hom^{0}(V\otimes W\langle-d_V-d_W+2\rangle ,\PL_{X,\hbar})\to\Hom^{0}(V\otimes W\langle -d_V-d_W\rangle ,\PL_{X,\hbar})\] is injective;
    \item \label{comassum} There exists an element $c_{V\otimes W}^{\loc}\in \Hom^{0}(V\otimes W\langle-d_V-d_W+2\rangle ,\PL_{X,\hbar})$ such that \[c_{V\otimes W}^{\loc}\cdot \hbar = [c_V^{\loc},c_W^{\loc}].\]
\end{enumerate}
\end{assumption}
\begin{remark}
    Note that in the case $T^*X$ is a hyperspherical $G$-Hamiltonian space, \thref{genassum} above is a direct consequence of \thref{bzsvloc}.
\end{remark}

\subsubsection{Automorphic commutator relations}
Assume we are in \thref{commsetup}, we get three special cohomological correspondences \[c_{V^I},c_{V^{I,\sw}},\D^{\hk}_{x,x+1,!}c_{V^{[I]}}\in \Cor_{\hk^{\glob}_{G,I},\IC_{V^I}}(\cP_X,\cP_X\langle d_{I}-2|I|\rangle)\] where the last one is defined in \eqref{nolegdiagpush}. Our main conjecture in this section is the following:

\begin{conj}[Automorphic commutator relation]\thlabel{geokolycor}
    Under \thref{commsetup}, we have the following identity:
    \[c_{V^{I}}-c_{V^{I,\sw}}=\D^{\hk}_{x,x+1,!}c_{V^{[I]}}\in\Cor_{\hk^{\glob}_{G,I},\IC_{V^I}}(\cP_X,\cP_X\langle d_{I}-2|I|\rangle).\]
\end{conj}

We are currently not able to prove \thref{geokolycor} in general but are only able to prove it under some additional technical assumptions. The additional condition we have to add is the following:

\begin{assumption}\thlabel{techassum}
The following hold:
    \begin{enumerate}
        \item \label{bzsvlocvan} 
        For $d=1,2,3$, we have $\Hom^0(V\otimes W\langle -d_V-d_W+2\rangle[d],\PL_{X,\hbar})=0$
        \item \label{gnot1} Let $g=g(C)$ to be the genus of the curve $C$. We have $g\neq 1$.
    \end{enumerate}
\end{assumption}

Our main theorem in this section is the following, which will be proved in  \autoref{proofgeokolycorthm}:
\begin{thm}\thlabel{geokolycorthm}
    Assume \thref{techassum}, then \thref{geokolycor} is true.
\end{thm}

\begin{remark}
    We are mostly interested in the middle-dimensional case (i.e. $d_V+d_W=2$, see \thref{middimassum}) and $T^*X$ is tempered hyperspherical as defined in  \autoref{hdual}. In this case, \itemref{techassum}{bzsvlocvan}
    just follows from \itemref{bzsvloc}{bzsvlocdual} which implies that $\PL_{X,\hbar}$ is concentrated in nonnegative cohomological degree.
    
    Therefore, \itemref{techassum}{bzsvlocvan} conjecturally holds in the middle-dimensional tempered case. \itemref{techassum}{gnot1} is more restrictive, and we would like to remove it in the future. We believe that all the additional assumptions in \thref{techassum} can be removed at the same time once a more essential argument is discovered.
\end{remark}

\begin{remark}
    One can also formulate and prove \thref{geokolycor} without assuming \thref{techassum} when taking $C=[\A^1/\Gm]$. The key point making this case provable is that \thref{slocvanishing} is automatically true in this case. This is the reason that we believe \thref{geokolycor} is true even without \thref{techassum}.  We would hope one can prove \thref{geokolycor} by reducing to the case $C=[\A^1/\Gm]$, but we do not know how to achieve this. 
\end{remark}

\begin{remark}
    Instead of using the Plancherel algebra with loop rotation, it may be more conceptual to formulate \thref{geokolycor} using the $\mathbb{E}_2$-Plancherel algebra and factorization structure developed in \cite[\S16]{BZSV}. However, as far as we know, the factorizable local conjecture is far from being known or even well-formulated, while the loop-rotated version seems to be reachable or known in many cases. Therefore, we choose to use the loop-rotated version rather than the factorizable version.
\end{remark}

\subsection{Proof of global commutator relations}\label{proofgeokolycorthm}
In this section, we give the proof of \thref{geokolycorthm}. The idea is easy: One first reduces to the case $|I|=2$, which is done in  \autoref{redtotwo}. After that, one reduces to a semi-local statement (\thref{geokolycorsloc}), which is the main subject of  \autoref{refgeokolycor} and  \autoref{globslocvar}. Finally, one relates the semi-local statement to the local commutator relation \itemref{genassum}{comassum} and concludes the proof in \autoref{proofgeokolycorsloc}.

\subsubsection{Reduction to the case $|I|=2$}\label{redtotwo}
To reduce the proof of \thref{geokolycor} to the case $|I|=2$, by compatibility of push-forward functoriality in \thref{ccpush} and composition in  \autoref{cccomp}, we only need to prove the following:
\begin{prop}\thlabel{compfund}
    Under \thref{commsetup}, we have \[c_{V^I}=c_{V^1}\circ c_{V^2}\circ\cdots\circ c_{V^r}\in \Cor_{\thk^{\glob}_{G,I},\IC_{V^I}}(\cP_X,\cP_X\langle d_I-2|I|\rangle)\isom\Cor_{\hk^{\glob}_{G,I},\IC_{V^I}}(\cP_X,\cP_X\langle d_I-2|I|\rangle)\] where $c_{V^I}$ and $c_{V^x}$ for $x\in I$ are both defined via \eqref{removeleg}.
\end{prop}

\begin{proof}[Proof of \thref{compfund}]
We use $C_1,\cdots, C_r$ to distinguish the different factors of $C^I$.
By \eqref{globcomp}, \eqref{periodpush}, \eqref{removeleg}, and compatibility of push-forward and composition (and compatibility between pull-back and composition), we know \[c_{V^I}=[C^I]\circ c_{V^1,1}\circ c_{V^2,2}\circ\cdots\circ c_{V^r,r}\circ \taut_{C^I}.\] On the other hand, we have \[c_{V^1}\circ c_{V^2}\circ\cdots\circ c_{V^r}= [C_1]\circ c_{V^1,1} \circ \taut_{C_1}\circ\cdots\circ[C_r] \circ c_{V^r,r}\circ \taut_{C_r}.\] Then it is easy to check that one can exchange two correspondences appearing in the formula above whenever they have different subindices $x\in I$.
\end{proof}

\subsubsection{A reformulation of \thref{geokolycor}}\label{refgeokolycor}
From now on, we assume $I=\{1,2\}$ and $[I]=\{1\}$.

\thref{geokolycor} is formulated in terms of special cohomological correspondences \eqref{removeleg}. In this section, we reformulate \thref{geokolycor} using \eqref{periodpush}.

Consider the diagram \begin{equation}\label{coevperioddiag}
    \begin{tikzcd}
        \Bun_G\times C \ar[d, "\id\times\D_C" ']  & \hk^{\glob}_{G,\fone} \ar[d, "\D^{\hk}_{\glob}"] \ar[l, "\lh_{\fone,\glob}" '] \ar[r, "\rh_{\fone,\glob}"] & \Bun_G\times C \ar[d, "\id\times\D_C"] \\
        \Bun_G\times C^2 & \hk^{\glob}_{G,\ftwo} \ar[l, "\lh_{\ftwo,\glob}" '] \ar[r, "\rh_{\ftwo,\glob}"] & \Bun_G\times C^2
    \end{tikzcd}.
\end{equation}
We have a map \begin{equation}\label{coevperiod}
\begin{split}
    \coev:&\Cor_{\hk^{\glob}_{G,\fone},\IC_{V\otimes W}}(\cP_X\boxtimes\uk_C,\cP_X\boxtimes\uk_C\langle d_V+d_W-2\rangle) \\
    \to& \Cor_{\hk^{\glob}_{G,\ftwo},\IC_{V\boxtimes W}}(\cP_X\boxtimes\uk_{C^2},\cP_X\boxtimes\uk_{C^2}\langle d_V+d_W\rangle)
    \end{split}
\end{equation}
defined by 
\begin{equation}\label{coevperioddef}
\begin{split}
    \coev:&\Cor_{\hk^{\glob}_{G,\fone},\IC_{V\otimes W}}(\cP_X\boxtimes\uk_C,\cP_X\boxtimes\uk_C\langle d_V+d_W-2\rangle) \\
    =&\Hom(\IC_{V\otimes W}\boxstar (\cP_X\boxtimes \uk_C), \cP_X\boxtimes \uk_C\langle d_V+d_W-2\rangle) \\
    \isom& \Hom((\id\times\D_C)^*\IC_{V\boxtimes W}\boxstar (\cP_X\boxtimes \uk_{C^2}), (\id\times\D_C)^*\cP_X\boxtimes \uk_{C^2}\langle d_V+d_W-2\rangle) \\
    \isom& \Hom((\id\times\D_C)^*\IC_{V\boxtimes W}\boxstar (\cP_X\boxtimes \uk_{C^2}), (\id\times\D_C)^!\cP_X\boxtimes \uk_{C^2}\langle d_V+d_W\rangle) \\
    \isom& \Hom(\IC_{V\boxtimes W}\boxstar (\cP_X\boxtimes \uk_{C^2}), (\id\times\D_C)_!(\id\times\D_C)^!\cP_X\boxtimes \uk_{C^2}\langle d_V+d_W\rangle) \\
    \to& \Hom(\IC_{V\boxtimes W}\boxstar (\cP_X\boxtimes \uk_{C^2}), \cP_X\boxtimes \uk_{C^2}\langle d_V+d_W\rangle) \\
    =&\Cor_{\hk^{\glob}_{G,\ftwo},\IC_{V\boxtimes W}}(\cP_X\boxtimes\uk_{C^2},\cP_X\boxtimes\uk_{C^2}\langle d_V+d_W\rangle)
    \end{split}
\end{equation}
Where the third step uses the purity isomorphism \thref{purity}.

We can reformulate \thref{geokolycor} as follows:

\begin{conj}\thlabel{geokolycorperiod}
    Under \thref{commsetup}, we have the following identity:
    \[c_{V\boxtimes W,\ftwo}-c_{W\boxtimes V,\ftwo}=\coev(c_{V\otimes W,\fone})\in\Cor_{\hk^{\glob}_{G,\ftwo},\IC_{V\boxtimes W}}(\cP_X\boxtimes\uk_{C^2},\cP_X\boxtimes\uk_{C^2}\langle d_V+d_W\rangle)\] where the elements $c_{V\boxtimes W,\ftwo},c_{W\boxtimes V,\ftwo},c_{V\otimes W,\fone}$ are both defined via \eqref{periodpush}.
\end{conj}
\begin{proof}[Proof of \thref{geokolycor} assuming \thref{geokolycorperiod}]
    Recall the map forgetting legs $w_I:\Bun_G\times C^I\to \Bun_G$. Consider the map \[\begin{split}
        \phi_{\ftwo}:=[C^2]\circ (\bullet)\circ \taut_{C^2}:&\Cor_{\hk^{\glob}_{G,\ftwo},\IC_{V\boxtimes W}}(\cP_X\boxtimes\uk_{C^2},\cP_X\boxtimes\uk_{C^2}\langle d_V+d_W\rangle)\\
        \to& \Cor_{\hk^{\glob}_{G,\ftwo},\IC_{V\boxtimes W}}(\cP_X,\cP_X\langle d_V+d_W-4\rangle)
    \end{split} \]
    It follows from the definition in \eqref{removeleg} that \[\phi_{\ftwo}(c_{V\boxtimes W,\ftwo}-c_{W\boxtimes V,\ftwo})=c_{V\boxtimes W}-c_{W\boxtimes V}.\] Therefore, it suffices to prove the following two maps coincide: \begin{equation}\label{interpcor}\begin{split}
        \phi_{\ftwo}\circ\coev=\D^{\hk}_{\glob,!}\circ \phi_{\fone}: & \Cor_{\hk^{\glob}_{G,\fone},\IC_{V\otimes W}}(\cP_X\boxtimes\uk_C,\cP_X\boxtimes\uk_C\langle d_V+d_W-2\rangle) \\
        \to & \Cor_{\hk^{\glob}_{G,\ftwo},\IC_{V\boxtimes W}}(\cP_X,\cP_X\langle d_V+d_W-4\rangle)
    \end{split}\end{equation} where $\phi_{\fone}=[C]\circ (\bullet) \circ \taut_C$ is the one-leg analogue of $\phi_{\ftwo}$.
    Consider the diagram \[\begin{tikzcd}
        \Bun_G \times C \ar[d, "\id\times\D_C"'] &\hk^{\glob}_{G,\fone} \ar[d, "\D^{\hk}_{\glob}" '] \ar[l, "\lh_{\fone,\glob}" '] &  \\
        \Bun_G\times C^2  \lcart&\hk^{\glob}_{G,\ftwo} \ar[r, "\rh_{\ftwo,\glob}"] \ar[l, "\lh_{\ftwo,\glob}" '] & \Bun_G \times C^2,
    \end{tikzcd}\] where the square is Cartesian.
    Denote \[\cF:=\rh_{\ftwo,\glob}^*(\cP_X\boxtimes\uk_{C^2})\otimes\IC_{V\boxtimes W}\langle -d_V-d_W+2\rangle \in \Shv(\hk^{\glob}_{G,\ftwo}).\] We can identify both maps in \eqref{interpcor} as maps \begin{equation}
            \phi_{\ftwo}\circ\coev=\D^{\hk}_!\circ \phi_{\fone}:\Hom(\lh_{\fone,\glob,!}\D^{\hk,*}_{\glob}\cF,w_{\fone}^!\cP_X\langle -2\rangle )\to \Hom(w_{\ftwo,!}\lh_{\ftwo,\glob,!}\cF,\cP_X\langle -2\rangle)
    \end{equation}
    For $\a\in \Hom(\lh_{\fone,\glob,!}\D^{\hk,*}_{\glob}\cF,w_{\fone}^*\cP_X)$, unwinding the definition of all the maps given in \eqref{nolegdiagpush} and \eqref{coevperioddef}, we have the following diagrams: \begin{equation}\label{droplegone}
        \begin{tikzcd}
            w_{\ftwo,!}\lh_{\ftwo,\glob,!}\cF \ar[d, "\id"] \ar[r]  & w_{\ftwo,!}\lh_{\ftwo,\glob,!}\D^{\hk}_{\glob,!}\D^{\hk,*}_{\glob}\cF \ar[r] & w_{\ftwo,!}(\id\times\D_C)_!\lh_{\fone,\glob,!}\D^{\hk,*}_{\glob}\cF \ar[d,"\id"] \\
            w_{\ftwo,!}\lh_{\ftwo,\glob,!}\cF \ar[r] & w_{\ftwo,!}(\id\times\D_C)_!(\id\times\D_C)^*\lh_{\ftwo,\glob,!}\cF \ar[r]  & w_{\ftwo,!}(\id\times\D_C)_!\lh_{\fone,\glob,!}\D^{\hk,*}_{\glob}\cF
        \end{tikzcd}
    \end{equation}
    \begin{equation}\label{droplegtwo}
        \begin{tikzcd}
             w_{\ftwo,!}(\id\times\D_C)_!w_{\fone}^!\cP_X\langle-2\rangle \ar[d, "\id"] \ar[r]  & w_{\fone,!}w_{\fone}^!\cP_X\langle -2\rangle \ar[r] & \cP_X\langle -2\rangle \ar[d, "\id"] \\
             w_{\ftwo,!}(\id\times\D_C)_!w_{\fone}^!\cP_X\langle-2\rangle \ar[r] & w_{\ftwo,!}(\id\times\D_C)_!(\id\times \D_C)^!w_{\ftwo}^!\cP_X\langle -2\rangle \ar[r] & \cP_X\langle -2\rangle
        \end{tikzcd}
    \end{equation}

    Note that $\D^{\hk}_!\circ \phi_{\fone}(\a)$ is given by the composition of the upper row of \eqref{droplegone} and $w_{\ftwo,!}(\id\times\D_C)_!\a$ and the upper row of \eqref{droplegtwo}, and $\phi_{\ftwo}\circ\coev(\a)$ is given by the composition of the lower row of \eqref{droplegone} and $w_{\ftwo,!}(\id\times\D_C)_!\a$ and the lower row of \eqref{droplegtwo}.
    The commutativity of the diagram \eqref{droplegone} follows from \thref{!pushbcstar}. The commutativity of the diagram \eqref{droplegtwo} follows from a similar argument. This concludes the proof.
\end{proof}

\subsubsection{Global and semi-local variants of \thref{geokolycorperiod}}\label{globslocvar}
In this section, we formulate variants of \thref{geokolycorperiod} in terms of \eqref{globcomp} and \eqref{sloccomp} and reduce the proof of the original statement to these variants.

Consider the diagram parallel to \eqref{coevperioddiag}:
\begin{equation}\label{coevglobdiag}
    \begin{tikzcd}
        \Bun_{G,\fone}^X \ar[d, "\D^{X}_{\glob}" ']  & \hk^{\Xp,\glob}_{G,\fone} \ar[d, "\D^{X,\hk}_{\glob}"] \ar[l, "\lh^X_{\fone,\glob}" '] \ar[r, "\rh^X_{\fone,\glob}"] & \Bun_{G,\fone}^X \ar[d, "\D^{X}_{\glob}"] \\
        \Bun_{G,\ftwo}^X & \hk^{\Xp,\glob}_{G,\ftwo} \ar[l, "\lh^X_{\ftwo,\glob}" '] \ar[r, "\rh^X_{\ftwo,\glob}"] & \Bun_{G,\ftwo}^X
    \end{tikzcd}
\end{equation}
and also the semi-local version
\begin{equation}\label{coevslocdiag}
    \begin{tikzcd}
        (LX/L^+G)_C \ar[d, "\D^{X}_{\sloc}" ']  & \hk^{\Xp,\sloc}_{G,\fone} \ar[d, "\D^{X,\hk}_{\sloc}"] \ar[l, "\lh^X_{\fone,\sloc}" '] \ar[r, "\rh^X_{\fone,\sloc}"] & (LX/L^+G)_C \ar[d, "\D^{X}_{\sloc}"] \\
        (LX/L^+G)_{C^2} & \hk^{\Xp,\sloc}_{G,\ftwo} \ar[l, "\lh^X_{\ftwo,\sloc}" '] \ar[r, "\rh^X_{\ftwo,\sloc}"] & (LX/L^+G)_{C^2}
    \end{tikzcd}
\end{equation}

We have maps of vector spaces similar to \eqref{coevperiod}:
\begin{equation}\label{coevglob}
\begin{split}
    \coev_{\glob}:&\Cor_{\hk^{\Xp,\glob}_{G,\fone},\IC_{V\otimes W}}(\d^X_{\fone,\glob},\d^X_{\fone,\glob}\langle d_V+d_W-2\rangle) \\
    \to& \Cor_{\hk^{\Xp,\glob}_{G,\ftwo},\IC_{V\boxtimes W}}(\d^X_{\ftwo,\glob},\d^X_{\ftwo,\glob}\langle d_V+d_W\rangle)
    \end{split}
\end{equation}
and 
\begin{equation}\label{coevsloc}
\begin{split}
    \coev_{\sloc}:&\Cor_{\hk^{\Xp,\sloc}_{G,\fone},\IC_{V\otimes W}}(\d^X_{\fone,\sloc},\d^X_{\fone,\sloc}\langle d_V+d_W-2\rangle) \\
    \to& \Cor_{\hk^{\Xp,\sloc}_{G,\ftwo},\IC_{V\boxtimes W}}(\d^X_{\ftwo,\sloc},\d^X_{\ftwo,\sloc}\langle d_V+d_W\rangle)
    \end{split}
\end{equation}
whose definition is completely parallel to \eqref{coevperioddef} given the following lemma:
\begin{lemma}\thlabel{deltapurity}
We have $\d^X_{\fone,\glob}\isom \D^{X,!}_{\glob}\d^X_{\ftwo,\glob}\langle 2\rangle $ and $\d^X_{\fone,\sloc}\isom \D^{X,!}_{\sloc}\d^X_{\ftwo,\sloc}\langle 2\rangle $ .
\end{lemma}
\begin{proof}[Proof of lemma]
Consider \begin{equation}
\begin{tikzcd}
    \Bun^X_{G}\times C \ar[r, "i^X_{\fone,\glob}" ] \ar[d, "\id\times\D_C" '] & \Bun^X_{G,\fone} \ar[d, "\D^X_{\glob}"] \\
    \Bun^X_{G}\times C^2 \ar[r, "i^X_{\ftwo,\glob}" ] & \Bun^X_{G,\ftwo}\rcart
    \end{tikzcd}
\end{equation}
and 
\begin{equation}
\begin{tikzcd}
    (L^+X/L^+G)_C \ar[r, "i^X_{\fone,\sloc}" ] \ar[d, "\D^{X,+}_{\sloc}" '] & (LX/L^+G)_C \ar[d, "\D^X_{\sloc}"] \\
    (L^+X/L^+G)_{C^2} \ar[r, "i^X_{\ftwo,\sloc}" ] & (LX/L^+G)_{C^2}\rcart
    \end{tikzcd}.
\end{equation}
By proper base change, we only need to show the purity transformations
\begin{equation}
    [\id\times \D_C]:\uk_{\Bun_G^X\times C}\to (\id\times\D_C)^!\uk_{\Bun_G^X\times C^2}
\end{equation}
and 
\begin{equation}
    [\D^{X,+}_{\sloc}]:\uk_{(L^+X/L^+G)_C}\to \D^{X,+,!}_{\sloc}\uk_{(L^+X/L^+G)_{C^2}}
\end{equation}
are isomorphisms. Here, both fundamental classes are defined via \thref{dfcgeneral}, where the second one uses the Cartesian square
\begin{equation*}
\begin{tikzcd}
(L^+X/L^+G)_{C} \arrow[r, "{\D^{X,+}_{\sloc}}"] \arrow[d, "l_{\{1\},\sloc}"'] & (L^+X/L^+G)_{C^2} \arrow[d, "l_{\{1,2\},\sloc}"]  \\
 C    \arrow[r, "\Delta_C"]                           & C^2    \rcart              
\end{tikzcd}
\end{equation*}
which makes the map ${\D^{X,+}_{\sloc}}$ quasi-smooth as in \thref{dfcgeneral}.

The first isomorphism follows from \thref{purity} and the second follows from \thref{purity} and \itemref{lprosmBC}{pro3} (note that $(L^+X)_{C^2}\rightarrow (L^+X/L^+G)_{C^2}$, $(L^+X)_{C^2}\rightarrow C^2$ are pro-smooth surjection, and $(L^+X)_{C^2}$ is qcqs).

\end{proof}

We can now formulate global and semi-local variants (or primitive versions) of \thref{geokolycorperiod}:
\begin{conj}[Global version]\thlabel{geokolycorglob}
    Under \thref{genassum}, we have the following identity:
    \[c_{V\boxtimes W,\ftwo}^{\glob}-c_{W\boxtimes V,\ftwo}^{\glob}=\coev_{\glob}(c_{V\otimes W,\fone}^{\glob})\in\Cor_{\hk^{\Xp,\glob}_{G,\ftwo},\IC_{V\boxtimes W}}(\d^X_{\ftwo,\glob},\d^X_{\ftwo,\glob}\langle d_V+d_W\rangle)\] where the elements $c_{V\boxtimes W,\ftwo}^{\glob},c_{W\boxtimes V,\ftwo}^{\glob},c_{V\otimes W,\fone}^{\glob}$ are both defined via \eqref{globcomp}.
\end{conj}
\begin{conj}[Semi-local version]\thlabel{geokolycorsloc}
    Under \thref{genassum}, we have the following identity:
    \[c_{V\boxtimes W,\ftwo}^{\sloc}-c_{W\boxtimes V,\ftwo}^{\sloc}=\coev_{\sloc}(c_{V\otimes W,\fone}^{\sloc})\in\Cor_{\hk^{\Xp,\sloc}_{G,\ftwo},\IC_{V\boxtimes W}}(\d^X_{\ftwo,\sloc},\d^X_{\ftwo,\sloc}\langle d_V+d_W\rangle)\] where the elements $c_{V\boxtimes W,\ftwo}^{\sloc},c_{W\boxtimes V,\ftwo}^{\sloc},c_{V\otimes W,\fone}^{\sloc}$ are both defined via \eqref{sloccomp}.
\end{conj}

We now deduce \thref{geokolycorperiod} from \thref{geokolycorglob} and then deduce \thref{geokolycorglob} from \thref{geokolycorsloc}.

\begin{proof}[Proof of \thref{geokolycorperiod} assuming \thref{geokolycorglob}]\label{periodpushproof}
    Take $I=\fone,\ftwo$ in the diagram \eqref{globtoperioddiag}. The push-forward functoriality in  \autoref{ccfunc} gives us \begin{equation}
        \pi_{I,!}:\Cor_{\hk^{\Xp,\glob}_{G,I},\IC_{V^I}}(\d^X_{I,\glob},\d^X_{I,\glob}\langle d_I\rangle) \to \Cor_{\hk^{\glob}_{G,I},\IC_{V^I}}(\cP_X\boxtimes\uk_{C^I},\cP_X\boxtimes\uk_{C^I}\langle d_I\rangle)
    \end{equation}
for $I=\fone,\ftwo$. We only need to show \begin{equation}\label{periodpushone}
\begin{split}
    \pi_{\ftwo,!}\circ \coev_{\glob} = \coev\circ \pi_{\fone,!} :& \Cor_{\hk^{\Xp,\glob}_{G,\fone},\IC_{V\otimes W}}(\d^X_{\fone,\glob},\d^X_{\fone,\glob}\langle d_V+d_W-2\rangle) \\
    \to& \Cor_{\hk^{\Xp,\sloc}_{G,\ftwo},\IC_{V\boxtimes W}}(\d^X_{\ftwo,\sloc},\d^X_{\ftwo,\sloc}\langle d_V+d_W\rangle)
    \end{split}.
\end{equation}
The proof is similar to \eqref{interpcor}, but we write down the argument again for safety.
Denote \[\cF:=\rh^{X,*}_{\ftwo,\glob}\d^X_{\glob}\otimes \IC_{V\boxtimes W}\langle -d_V-d_W\rangle \in\Shv(\hk^{\Xp,\glob}_{G,\ftwo}).\] We can identify both maps in \eqref{periodpushone} as maps
\begin{equation}
\begin{split}
    \pi_{\ftwo,!}\circ \coev_{\glob} = \coev\circ \pi_{\fone,!} :& \Hom(\lh^X_{\fone,\glob,!}\D^{\hk,*}_{\glob}\cF, \D^{X,!}_{\glob}\d^X_{\ftwo,\glob}) \\
    \to& \Hom(\pi_{\ftwo,!}\lh^X_{\ftwo,\glob,!}\cF, \pi_{\ftwo,!}\d^X_{\ftwo,\glob})
    \end{split}   
\end{equation}
Consider commutative diagrams
\begin{equation}\label{periodpushtwo}
    \begin{tikzcd}
        \pi_{\ftwo,!}\lh^X_{\ftwo,\glob,!}\cF \ar[d, "\id" '] \ar[r] & \pi_{\ftwo,!}\D^X_{\glob,!}\D^{X,*}_{\glob}\lh^X_{\ftwo,\glob,!}\cF \ar[r] & (\id\times \D_C)_!\pi_{\fone,!}\lh^X_{\fone,\glob,!}\D^{\hk,*}_{\glob} \cF \ar[d, "\id" ] \\
        \pi_{\ftwo,!}\lh^X_{\ftwo,\glob,!}\cF \ar[r] & (\id\times \D_C)_!(\id\times\D_C)^*\pi_{\ftwo,!}\lh^X_{\ftwo,\glob,!}\cF \ar[r] & (\id\times \D_C)_!\pi_{\fone,!}\lh^X_{\fone,\glob,!}\D^{\hk,*}_{\glob} \cF
    \end{tikzcd}
\end{equation}
and 
\begin{equation}\label{periodpushthree}
    \begin{tikzcd}
        (\id\times \D_C)_!\pi_{\fone,!}\D^{X,!}_{\glob}\d^X_{\ftwo,\glob} \ar[d, "\id" '] \ar[r] & \pi_{\ftwo,!}\D^X_{\glob,!}\D^{X,!}_{\glob}\d^X_{\ftwo,\glob} \ar[r] & \pi_{\ftwo,!}\d^X_{\ftwo,\glob} \ar[d, "\id"] \\
        (\id\times \D_C)_!\pi_{\fone,!}\D^{X,!}_{\glob}\d^X_{\ftwo,\glob} \ar[r] & (\id\times \D_C)_!(\id\times \D_C)^!\pi_{\ftwo,!}\d^X_{\ftwo,\glob} \ar[r] & \pi_{\ftwo,!}\d^X_{\ftwo,\glob}
    \end{tikzcd}
\end{equation}
Here the first arrow in the bottom of \eqref{periodpushthree} uses the Beck-Chevalley base change map $\pi_{\fone,!}\D^{X,!}_{\glob}\d^X_{\ftwo,\glob}\to (\id\times \D_C)^!\pi_{\ftwo,!}\d^X_{\ftwo,\glob}$ which can be identified with the natural purity isomorphism $\cP_X\boxtimes\uk_C\langle -2\rangle \isom \cP_X\boxtimes \uk_{C^2}$ using \thref{deltapurity}.

Fix an element $\a\in\Hom(\lh^X_{\fone,\glob,!}\D^{\hk,*}_{\glob}\cF, \D^{X,!}_{\glob}\d^X_{\ftwo,\glob})$. Unwinding definitions, one can see that the map $\pi_{\ftwo,!}\circ \coev_{\glob}(\a)$ can be identified with the composition of the top row of \eqref{periodpushtwo} and $(\id\times \D_C)_!\pi_{\fone,!}(\a)$ and the top row of \eqref{periodpushthree}. The map $\coev\circ \pi_{\fone,!}(\a)$ can be identified with the composition of the bottom row of \eqref{periodpushtwo} and $(\id\times \D_C)_!\pi_{\fone,!}(\a)$ and the bottom row of \eqref{periodpushthree}. The commutativity of \eqref{periodpushtwo} follows from \thref{!pushbcstar}. The commutativity of \eqref{periodpushthree} follows from a similar argument.
This concludes the proof.
\end{proof}

\begin{proof}[Proof of \thref{geokolycorglob} assuming \thref{geokolycorsloc}]
    Keep our notation the same as the previous proof. The pull-back functoriality in  \autoref{ccfunc} gives us
    \begin{equation}
        f_{I,\gtosl}^*: \Cor_{\hk^{\Xp,\sloc}_{G,I},\IC_{V^I}}(\d^X_{I,\sloc},\d^X_{I,\sloc}\langle d_I\rangle) \to \Cor_{\hk^{\Xp,\glob}_{G,I},\IC_{V^I}}(\d^X_{I,\glob},\d^X_{I,\glob}\langle d_I\rangle)
    \end{equation}
\end{proof}
for $I=\fone,\ftwo$. 
Given the compatibility of the pull-back functoriality in  \autoref{ccfunc} and composition  \autoref{cccomp}, we only need to show \begin{equation}\label{sloctoglobone}
\begin{split}
    f_{\ftwo,\gtosl}^*\circ \coev_{\sloc} = \coev_{\glob}\circ f_{\fone,\gtosl}^* :& \Cor_{\hk^{\Xp,\sloc}_{G,\fone},\IC_{V\otimes W}}(\d^X_{\fone,\sloc},\d^X_{\fone,\sloc}\langle d_V+d_W-2\rangle) \\
    \to& \Cor_{\hk^{\Xp,\glob}_{G,\ftwo},\IC_{V\boxtimes W}}(\d^X_{\ftwo,\glob},\d^X_{\ftwo,\glob}\langle d_V+d_W\rangle)
    \end{split}.
\end{equation}
Denote \[\cF:=\rh^{X,*}_{\ftwo,\sloc}\d^X_{\sloc}\otimes \IC_{V\boxtimes W}\langle -d_V-d_W\rangle \in\Shv(\hk^{\Xp,\sloc}_{G,\ftwo}).\] Then the maps in \eqref{sloctoglobone} are identifies with maps
\begin{equation}
\begin{split}
    f_{\ftwo,\gtosl}^*\circ \coev_{\sloc} = \coev_{\glob}\circ f_{\fone,\gtosl}^* :& \Hom(\lh^X_{\fone,\sloc,!}\D^{\hk,*}_{\sloc}\cF, \D^{X,!}_{\sloc}\d^X_{\ftwo,\sloc}) \\
    \to& \Hom(f^*_{\ftwo,\gtosl}\lh^X_{\ftwo,\sloc,!}\cF, f^*_{\ftwo,\gtosl}\d^X_{\ftwo,\sloc})
    \end{split}.
\end{equation}
Consider commutative diagrams
\begin{equation}\label{sloctoglobtwo}
    \begin{tikzcd}
        f^*_{\ftwo,\gtosl}\lh^X_{\ftwo,\sloc,!}\cF \ar[d, "\id" '] \ar[r] & f^*_{\ftwo,\gtosl}\D^X_{\sloc,!}\D^{X,*}_{\sloc}\lh^X_{\ftwo,\sloc,!}\cF \ar[r] & \D^{X}_{\glob,!}f^*_{\fone,\gtosl}\lh^X_{\fone,\sloc,!}\D^{\hk,*}_{\sloc}\cF \ar[d, "\id"] \\
        f^*_{\ftwo,\gtosl}\lh^X_{\ftwo,\sloc,!}\cF \ar[r] & \D^X_{\glob,!}\D^{X,*}_{\glob}f^*_{\ftwo,\gtosl}\lh^X_{\ftwo,\sloc,!}\cF \ar[r] & \D^{X}_{\glob,!}f^*_{\fone,\gtosl}\lh^X_{\fone,\sloc,!}\D^{\hk,*}_{\sloc}\cF
    \end{tikzcd}
\end{equation}
and 
\begin{equation}\label{sloctoglobthree}
    \begin{tikzcd}
        \D^{X}_{\glob,!}f^*_{\fone,\gtosl}\D^{X,!}_{\sloc}\d^X_{\ftwo,\sloc} \ar[d, "\id"] \ar[r] & \D^{X}_{\glob,!}\D^{X,!}_{\glob}f^*_{\ftwo,\gtosl}\d^X_{\ftwo,\sloc} \ar[r] & f^*_{\ftwo,\gtosl}\d^X_{\ftwo,\sloc} \ar[d, "\id"] \\
        \D^{X}_{\glob,!}f^*_{\fone,\gtosl}\D^{X,!}_{\sloc}\d^X_{\ftwo,\sloc} \ar[r] & f^*_{\ftwo,\gtosl}\D^{X}_{\sloc,!}\D^{X,!}_{\sloc}\d^X_{\ftwo,\sloc} \ar[r] & f^*_{\ftwo,\gtosl}\d^X_{\ftwo,\sloc}
    \end{tikzcd}.
\end{equation}
Here the first arrow in the top of \eqref{sloctoglobthree} uses the Beck-Chevalley base change map $f^*_{\fone,\gtosl}\D^{X,!}_{\sloc}\d^X_{\ftwo,\sloc}\to \D^{X,!}_{\glob}f^*_{\ftwo,\gtosl}\d^X_{\ftwo,\sloc}$, which is identified with the natural purity isomorphism $\d^X_{\fone,\glob}\langle -2\rangle\isom \D^{X,!}_{\glob}\d^X_{\ftwo,\glob}$ defined in \thref{deltapurity}.
The commutativity of both diagrams follows from statements similar to \thref{!pushbcstar}. The rest of the argument is similar to the previous proof, and we omit it.

\subsubsection{Proof of \thref{geokolycorsloc} assuming \thref{techassum}}\label{proofgeokolycorsloc}
Now we come to the technical heart of this section.
\begin{proof}[Proof of \thref{geokolycorsloc} assuming \thref{techassum}]
Consider $U:=C^2\bs \D_C(C)$.
Consider the diagram
\begin{equation}
\begin{tikzcd}
    (LX/L^+G)_C \ar[d,"l_C" '] \ar[r, "\D^X_{\sloc}" ] & (LX/L^+G)_{C^2} \ar[d,"l_{C^2}"] & (LX/L^+G)_{U} \ar[d,"l_{U}"] \ar[l, "j^X" ']\\
    C \ar[r, "\D_C"] & C^2 \rcart\lcart& U \ar[l, "j" ']
    \end{tikzcd}
\end{equation}
where both squares are Cartesian. 
Consider the excision fiber sequence in \thref{excisionSHV}:
\begin{equation}
    j^X_!j^{X,*}\IC_{V\boxtimes W}\boxstar \d^X_{\ftwo,\sloc}\to \IC_{V\boxtimes W}\boxstar \d^X_{\ftwo,\sloc}\to \D^X_{\sloc,!}\D^{X,*}_{\sloc} \IC_{V\boxtimes W}\boxstar \d^X_{\ftwo,\sloc}.
\end{equation}
It gives rise to the left column of the following diagram
\begin{equation}\label{funddiagone}
\begin{tikzcd}
    \Hom^0(\D^X_{\sloc,!}\D^{X,*}_{\sloc} \IC_{V\boxtimes W}\boxstar \d^X_{\ftwo,\sloc}, \d^X_{\ftwo,\sloc}\langle d_V+d_W\rangle) \ar[r, "\adj", "\sim" '] \ar[d, "\Gys" '] & \Hom^0(\IC_{V\otimes W}\boxstar\d^{X}_{\fone,\sloc},\d^X_{\fone,\sloc}\langle d_V+d_W-2\rangle ) \ar[d, "\cdot c_1(T_C)"] \\
    \Hom^0( \IC_{V\boxtimes W}\boxstar \d^X_{\ftwo,\sloc}, \d^X_{\ftwo,\sloc}\langle d_V+d_W\rangle) \ar[r, "\D^{X,*}_{\sloc}"] \ar[d, "j^{X,*}" '] & \Hom^0(\IC_{V\otimes W}\boxstar \d^X_{\fone,\sloc},\d^X_{\fone,\sloc}\ \langle d_V+d_W\rangle)\\
    \Hom^0(j^X_!j^{X,*}\IC_{V\boxtimes W}\boxstar \d^X_{\ftwo,\sloc},\d^X_{\ftwo,\sloc}\langle d_V+d_W\rangle)&
\end{tikzcd}
\end{equation}
Here, the map $\adj$ is induced by $(\D^X_{\sloc,!},\D^{X,!}_{\sloc})$-adjunction and the isomorphism $\D^{X,!}_{\sloc}\d^X_{\fone,\sloc}\isom\d^X_{\ftwo,\sloc}\langle -2\rangle$ in \thref{deltapurity}. The map $\cdot c_1(T_C)$ is the cup product with the first Chern class of the tangent bundle $T_C$ of the curve $C$. The map $\D^{X,*}_{\sloc}$ is induced by the pull-back functoriality of cohomological correspondence in \thref{ccpull}.

Note that by definition of the map $\coev_{\sloc}$ given in \eqref{coevsloc} we have \[\coev_{\sloc}=\Gys\circ\adj^{-1}.\] Therefore, to prove \thref{geokolycorsloc} is to prove
\begin{equation}\label{geokolycorslocref}
    c_{V\boxtimes W}^{\sloc} - c_{W\boxtimes V}^{\sloc} = \Gys\circ\adj^{-1}(c_{V\otimes W}^{\sloc})
\end{equation}

\begin{lemma}\thlabel{funddiagonecom}
    The square in \eqref{funddiagone} is commutative.
\end{lemma}
\begin{proof}[Proof of \thref{funddiagonecom}]
    Write $\cF:=\IC_{V\boxtimes W}\boxstar \d^X_{\ftwo,\sloc}\langle -d_V-d_W\rangle$. Recall $\d^X_{\ftwo,\sloc}=i_{\ftwo,\sloc,*}\uk_{(L^+X/L^+G)_{C^2}}$. Write \[i_{|I|}=i_{I,\sloc}:(L^+X/L^+G)_{C^I}\to (LX/L^+G)_{C^I}\] for $I=\ftwo,\fone$, write \[\D=\D^{X}_{\sloc}:(LX/L^+G)_C\to (LX/L^+G)_{C^2}\] and $\uk=\uk_{(L^+X/L^+G)_{C^2}}$. Then we can rewrite the square in \eqref{funddiagone} as
    \begin{equation}\label{funddiagonesimp}
        \begin{tikzcd}
            \Hom^0(\D_!\D^*\cF,i_{2,*}\uk) \ar[r, "\adj", "\sim" '] \ar[d, "\Gys" '] & \Hom^0(\D^*\cF,\D^!i_{2,*}\uk) \ar[d, "\cdot c_1(T_C)"]   \\
            \Hom^0(\cF,i_{2,*}\uk) \ar[r, "\D^*"] & \Hom^0(\D^*\cF,\D^*i_{2,*}\uk)
        \end{tikzcd}.
    \end{equation}

    Write \[\D^+:(L^+X/L^+G)_C\to (L^+X/L^+G)_{C^2}\] and $\cG:=i_2^*\cF$. Applying $(i^*,i_*)$-adjunction and base change isomorphisms for $i_*$, we can further rewrite \eqref{funddiagonesimp} as
    \begin{equation}\label{funddiagonebc}
        \begin{tikzcd}
            \Hom^0(\D^+_!\D^{+,*}\cG, \uk) \ar[r, "\adj", "\sim" '] \ar[d, "\Gys" '] & \Hom^0(\D^{+,*}\cG,\D^{+,!}\uk) \ar[d, "\cdot c_1(T_C)"] \\
            \Hom^0(\cG,\uk) \ar[r, "\D^{+,*}"] & \Hom(\D^{+,*}\cG,\D^{+,*}\uk)
        \end{tikzcd}.
    \end{equation}

    For any $\a\in\Hom^0(\D^{+,*}\cG,\D^{+,!}\uk)$, consider the following diagram
    \begin{equation}\label{funddiagoneuf}
        \begin{tikzcd}
            \D^{+,*}\cG \ar[r]\ar[d ,"\a" ']& \D^{+,*}\D^+_*\D^{+,*}\cG\ar[d, "\D^{+,*}\D^+_*\a"] & \\
            \D^{+,!}\uk \ar[dash, r, "\sim"] & \D^{+,*}\D^+_*\D^{+,!}\uk \ar[r] & \D^*\uk
        \end{tikzcd}
    \end{equation}
    where arrows without names are induced by natural adjunctions. It is easy to see $\D^{+,*}\circ\Gys\circ\adj^{-1}(\a)$ gives the upper route from the left-top of the diagram \eqref{funddiagoneuf} to the right-bottom of the diagram, while $\a\cdot c_1(T_C)$ gives the lower route. The commutativity of the square in \eqref{funddiagoneuf} is tautological. This concludes the proof of \thref{funddiagonecom}.

    \end{proof}

    Consider the following diagram
    \begin{equation}\label{funddiagtwo}
        \begin{tikzcd}
            \Hom^0(\IC_{V\otimes W}*\d^X_{\loc},\d^X_{\loc}\langle d_V+d_W-2\rangle)\ar[r, "f^{\hk,*}_{\sltol}"] \ar[d, "\cdot\hbar" '] & \Hom^0(\IC_{V\otimes W}\boxstar\d^{X}_{\fone,\sloc},\d^X_{\fone,\sloc}\langle d_V+d_W-2\rangle ) \ar[d, "\cdot c_1(T_C)"] \\
            \Hom^0(\IC_{V\otimes W}*\d^X_{\loc},\d^X_{\loc}\langle d_V+d_W\rangle) \ar[r, "f^{\hk,*}_{\sltol}"] & \Hom^0(\IC_{V\otimes W}\boxstar\d^{X}_{\fone,\sloc},\d^X_{\fone,\sloc}\langle d_V+d_W\rangle )
        \end{tikzcd}
    \end{equation}
    where the map $f^{\hk,*}_{\sltol}$ is defined via the pull-back functoriality of cohomological correspondences in  \autoref{ccfunc}. 
    \begin{lemma}\thlabel{funddiagtwocom}
        The diagram \eqref{funddiagtwo} is commutative.
    \end{lemma}
    \begin{proof}[Proof of \thref{funddiagtwocom}]
    This follows from the fact that the composition $C\to B\Aut(D)\to B\Gm$ is the map defining $T_C\in\Pic(C)$, hence gives us $\Coor^*\hbar=c_1(T_C)$. See  \autoref{cohbg} for our choice of $\hbar$.
    \end{proof}

    We make the following observation:
    \begin{lemma}\thlabel{fundequal}
        We have \[\D^{X,*}c_{V\boxtimes W}^{\sloc} = f^{\hk,*}_{\sltol}c_{V\boxtimes W}^{\loc}\in \Hom^0(\IC_{V\otimes W}\boxstar \d^X_{\fone,\sloc},\d^X_{\fone,\sloc}\ \langle d_V+d_W\rangle)\] and \[\D^{X,*}c_{W\boxtimes V}^{\sloc} = f^{\hk,*}_{\sltol}c_{W\boxtimes V}^{\loc}\in\Hom^0(\IC_{V\otimes W}\boxstar \d^X_{\fone,\sloc},\d^X_{\fone,\sloc}\ \langle d_V+d_W\rangle),\] where $\D^{X,*}:=\D^{X,*}_{\sloc}$ is the middle horizontal arrow in \eqref{funddiagone}, $f^{\hk,*}_{\fone,\sltol}$ is the bottom arrow in \eqref{funddiagtwo}, and $c_{V\boxtimes W}^{\loc}$, $c_{W\boxtimes V}^{\loc}$ are defined in \eqref{loccomp}.
    \end{lemma}
    \begin{proof}
        This follows directly from the compatibility of pull-back and composition of cohomological correspondences in  \autoref{ccfunc} and  \autoref{cccomp}.
    \end{proof}

    \begin{lemma}\thlabel{gencomm}
    We have \begin{equation}\label{gencommidentity}
        j^{X,*}c_{V\boxtimes W}^{\sloc}=j^{X,*}c_{W\boxtimes V}^{\sloc}\in \Hom^0(j^X_!j^{X,*}\IC_{V\boxtimes W}\boxstar \d^X_{\ftwo,\sloc},\d^X_{\ftwo,\sloc}\langle d_V+d_W\rangle).
    \end{equation}
    \end{lemma}
    \begin{proof}
        Define $\hk^{\Xp,\sloc}_{G,U}=l_{\ftwo,\sloc}^{\hk,-1}(U)$ where $l_{\ftwo,\sloc}^{\hk}:\hk^{\Xp,\sloc}_{G,\ftwo}\to C^2$ is the map remembering only legs. Consider the following diagram
        \[\begin{tikzcd}
            (LX/L^+G)_U \ar[d, "f_{\sltol}^{\times 2}" '] & \hk^{\Xp,\sloc}_{G,U} \ar[l, "\lh^X_{\ftwo,\sloc}" '] \ar[r, "\rh^X_{\ftwo,\sloc}"] \ar[d, "(f^{\hk}_{\sltol})^{\times 2}"] & (LX/L^+G)_U \ar[d, "f_{\sltol}^{\times 2}" ] \\
            (LX/L^+G\rtimes\Aut(D))^2 \lcart& (\hk^{\Xp,\loc}_G)^2 \ar[l,"(\lh^X_{\loc})^{\times 2}"' ] \ar[r, "(\rh^X_{\loc})^{\times 2}"] & (LX/L^+G\rtimes\Aut(D))^2\rcart
        \end{tikzcd}\]
        where both squares are Cartesian. Then we have \[j^{X,*}c_{V\boxtimes W}^{\sloc}=(f^{\hk}_{\sltol})^{\times 2,*}(c_V^{\loc}\boxtimes c_W^{\loc})=
        j^{X,*}c_{W\boxtimes V}^{\sloc}\] and we are done.
    \end{proof}

    Denote $\e=c_{V\boxtimes W}^{\sloc} - c_{W\boxtimes V}^{\sloc} - \Gys\circ\adj^{-1}(c_{V\otimes W}^{\sloc})$. We have \begin{equation}
    \begin{split}
        \D^{X,*}\e&=\D^{X,*}(c_{V\boxtimes W}^{\sloc} - c_{W\boxtimes V}^{\sloc}) - \D^{X,*}\circ\Gys\circ\adj^{-1}(c_{V\otimes W}^{\sloc}) \\
        &=f^{\hk,*}_{\sltol}(c_{V\boxtimes W}^{\loc}-c_{W\boxtimes V}^{\loc})-c_{V\otimes W}^{\sloc}\cdot c_1(T_C) \\
        &=f^{\hk,*}_{\sltol}(c_{V\boxtimes W}^{\loc}-c_{W\boxtimes V}^{\loc}-c_{V\otimes W}^{\loc}\cdot\hbar) \\
        &=0
        \end{split},
    \end{equation}
    where the second equality uses \thref{fundequal} and \thref{funddiagonecom}, the third equality uses \thref{funddiagtwocom}, and the last equality follows from our choice of $c_{V\otimes W}^{\loc}$ in \itemref{genassum}{comassum}.

    On the other hand, we have
    \begin{equation}
            j^{X,*}\e =j^{X,*}(c_{V\boxtimes W}^{\sloc} - c_{W\boxtimes V}^{\sloc}) - j^{X,*}\circ\Gys\circ\adj^{-1}(c_{V\otimes W}^{\sloc}) =0 -0=0
    \end{equation}
    where the second identity uses \thref{gencomm} and $j^{X,*}\circ\Gys=0$.

    We can conclude our proof given the following lemma:
    \begin{lemma}\thlabel{slocvanishing}
        Under \thref{techassum}, the map $\cdot c_1(T_C)$ in \eqref{funddiagone} is injective.
    \end{lemma}

    Assuming the lemma, we can conclude the proof of \thref{geokolycorsloc} as follows. Since $j^{X,*}\e=0$ and the left column of \eqref{funddiagone} is exact, there exists $\e_0\in\Hom^0(\D^X_{\sloc,!}\D^{X,*}_{\sloc} \IC_{V\boxtimes W}\boxstar \d^X_{\ftwo,\sloc}, \d^X_{\ftwo,\sloc}\langle d_V+d_W\rangle)$ at the upper-left corner of \eqref{funddiagone} such that $\e=\Gys(\e_0)$. By the commutativity of the diagram \eqref{funddiagone} we know $\adj(\e_0)\cdot c_1(T_C)=\D^{X,*}\circ\Gys(\e_0)=\D^{X,*}\e=0$. We can conclude $\e_0=0$ by \thref{slocvanishing} and get $\e=0$, which concludes the proof of \thref{geokolycorsloc}.
    \end{proof}

    \begin{proof}[Proof of \thref{slocvanishing}]
    We drop the subscript $\fone$ (which stands for one leg) from everywhere below. Consider the following diagram in which both squares are Cartesian
    \begin{equation}
        \begin{tikzcd}
            (LX/L^+G)_C \ar[d, "f_{\sltol}"'] & (L^+X/L^+G)_C \ar[r, "l_{\sloc}"] \ar[l, "i_{\sloc}"'] \ar[d, "f^+_{\sltol}"] & C \ar[d,"\Coor"] \\
            LX/L^+G\rtimes\Aut(D) \lcart& L^+X/L^+G\rtimes\Aut(D) \ar[r, "l_{\loc}"] \ar[l, "i_{\loc}"'] & B\Aut(D)\rcart
        \end{tikzcd}.
    \end{equation}
    We denote $\cF:=\IC_{V\otimes W}*\d^X_{\loc}\langle -d_V-d_W+2\rangle \in \Shv(LX/L^+G\rtimes\Aut(D)) $, then the map we care about becomes \begin{equation}\label{injectiontobeproved}
        \cdot c_1(T_C):\Hom^0(f^{*}_{\sltol}\cF,\d^X_{\sloc})\to\Hom^0(f^{*}_{\sltol}\cF,\d^X_{\sloc}\langle 2\rangle).
    \end{equation}
    We further denote $\cG:=i_{\loc}^*\cF$ and $\uk=\uk_{B\Aut(D)}$. 
    Note that \begin{equation}\label{localpush}
    \begin{split}
        \Hom(f^{*}_{\sltol}\cF,\d^X_{\sloc})&= \Hom(f^{*}_{\sltol}\cF,i_{\sloc,*}\uk_{(L^+X/L^+G)_C}), \\
        & \isom \Hom(i_{\sloc}^*f^{*}_{\sltol}\cF,\uk_{(L^+X/L^+G)_C}), \\
        &\isom \Hom(f^{+,*}_{\sltol}\cG,\uk_{(L^+X/L^+G)_C}), \\
        &\isom \Hom(\cG,f^{+}_{\sltol,*}l_{\sloc}^*\Coor^*\uk), \\
        &\isom \Hom(\cG,l_{\loc}^*\Coor_*\Coor^*\uk),
        \end{split}
    \end{equation}
    where the last step uses \itemref{lprosmBC}{pro4}, since $L^+X\rightarrow L^+X/L^+G\rtimes\rrr{Aut}(D)$ and $L^+X\rightarrow B\rrr{Aut}(D)$ are pro-smooth surjections.

    Under the identification \eqref{localpush}, the map \eqref{injectiontobeproved} becomes 
    \begin{equation}
        \cdot\hbar: \Hom^0(\cG,l_{\loc}^*\Coor_*\Coor^*\uk)\to\Hom^0(\cG,l_{\loc}^*\Coor_*\Coor^*\uk\langle 2\rangle).
    \end{equation}
    induced by the map $\hbar:\uk\to\uk\langle 2\rangle$.

    Note the following fact:
    \begin{lemma}\thlabel{curvecoh}
        $\Coor_*\Coor^*\uk=\cofib(\uk\langle-4\rangle\xrightarrow{\hbar^2}\uk)\oplus \cofib(\uk\langle-2\rangle\xrightarrow{\hbar} \uk)^{\oplus 2g}[-1]$.
    \end{lemma}
    \begin{proof}[Proof of \thref{curvecoh}]
        Note that we have a natural equivalence of categories
        \begin{equation}\label{shvba}
            \Hom(\uk,\bullet):\Shv(B\Aut(D))^{>-\infty}\isom \Mod(k[\hbar])^{>-\infty},
        \end{equation}
        since $\Hom(\uk,\uk)=k[\hbar]\in \Mod(k[\hbar])^{>-\infty}$. Here for a stable $\infty$-category $\sC$ with a t-structure, we write $\sC^{>-\infty}$ to denote the full-subcategory of $\sC$ consisting of objects bounded on the left. Moreover, given \itemref{techassum}{gnot1}, by the formality of the cohomology of $C$, we have \[\Hom(\uk,\Coor_*\Coor^*\uk)\isom k[\hbar]/(\hbar^2)\oplus k^{2g}[-1]\in \Mod(k[\hbar])^{>-\infty}.\]  Our desired result follows as
        \[k[\hbar]/(\hbar^2)\oplus k^{2g}[-1]\isom \cofib(k[\hbar]\langle-4\rangle\xrightarrow{\cdot\hbar^2}k[\hbar])\oplus \cofib(k[\hbar]\langle-2\rangle\xrightarrow{\cdot\hbar}k[\hbar])^{\oplus 2g}[-1].\]
    \end{proof}

    Denote $\cA=l_{\loc}^*\cofib(\uk\langle-4\rangle\xrightarrow{\hbar^2}\uk)$ and $\cB=l_{\loc}^*\cofib(\uk\langle-2\rangle\xrightarrow{\hbar} \uk)[-1]$. We need to show that both maps
    \begin{equation}\label{vanfora}
        \cdot\hbar:\Hom^0(\cG,\cA)\to\Hom^0(\cG,\cA\langle 2\rangle)
    \end{equation}
    \begin{equation}\label{vanforb}
        \cdot\hbar:\Hom^0(\cG,\cB)\to\Hom^0(\cG,\cB\langle 2\rangle)
    \end{equation}
    are injective.
    Denote $\uk_X=l_{\loc}^*\uk$. Note that we have \[\Hom(\cG,\uk_X)\isom\Hom(\IC_{V\otimes W}*\d^X_{\loc},\d^X_{\loc}\langle d_V+d_W-2\rangle )=\Hom_{\Gc}(V\otimes W\langle -d_V-d_W+2\rangle, \PL_{X,\hbar}).\]
    \itemref{techassum}{bzsvlocvan} translates to \begin{equation}\label{vanref}
        \Hom^0(\cG,\uk_X[d])=0
    \end{equation} for $d=-1,-2,-3$.
    \itemref{genassum}{flatassum} translates to the map \begin{equation}\label{injref}
     \cdot\hbar:\Hom^0(\cG,\uk_X)\to\Hom^0(\cG,\uk_X[2])
    \end{equation}
    is injective.
    For \eqref{vanfora}, we have the following diagram
    \begin{equation}\label{diagfora}
    \begin{tikzcd}
        & \Hom^0(\cG,\uk_X) \ar[d, "\gamma"] \ar[r, "\a"] & \Hom^0(\cG,\cA) \ar[d, "\delta"] \ar[r] & \Hom^0(\cG,\uk_X[-3]) \\
        \Hom^0(\cG,\uk_X[-2]) \ar[r] & \Hom^0(\cG,\uk_X[ 2]) \ar[r, "\b"] & \Hom^0(\cG,\cA[ 2]) & 
        \end{tikzcd}
    \end{equation}
    where both rows are exact sequences. Note that \eqref{vanref} implies that $\a$ is surjective and $\b$ is injective. Since $\gamma$ is injective by \eqref{injref}, we know $\d$ is injective and hence prove injectivity of \eqref{vanfora}.

    For \eqref{vanforb}, we have a short exact sequence \[\Hom^0(\cG,\uk_X[-1])\to\Hom^0(\cG,\cB)\to\Hom^0(\cG,\uk_X[-2]).\] Note that left and right terms are zero by \eqref{vanref}, we know $\Hom^0(\cG,\cB)=0$, which implies injectivity of \eqref{vanforb}. This concludes the proof of \thref{slocvanishing}.

    \end{proof}

\begin{proof}[Proof of \thref{geokolycorthm}]
    One combines all the arguments in  \autoref{proofgeokolycorthm}.
\end{proof}

\section{Automorphic Clifford relations}\label{automorphiccliffordrelations}
In this section, we restrict to a special case (minuscule and Poisson-pure case) such that the commutator relations (\thref{geokolycor}) on global special cohomological correspondences simplify to what we call the automorphic Clifford relations (\thref{geokolycormidmin}).
\begin{itemize}
    \item In \autoref{dfc}, we introduce a construction of local special cohomological correspondences via derived fundamental classes.
    \item In \autoref{middimcor}, we induce the minuscule and Poisson-pure condition and state the main result \thref{geokolycormidmin}. 
    \item In \autoref{verclifrel}, we introduce a way to compute a scalar to be determined in \thref{geokolycormidmin}. 
\end{itemize}

\subsection{Cohomological correspondence via derived fundamental class}\label{dfc}
In this section, we study a particular construction of the local special cohomological correspondences $c^{\loc}_V$ in \thref{loccordef} via derived fundamental classes. For notations on moduli spaces and morphisms between them, we refer to  \autoref{geometricsetup} (see  \autoref{modulinotations} for a summary).

We by default consider only moduli spaces with one leg with index set $I=\fone$ and drop the subscript $\fone$ everywhere.

\subsubsection{Local special cohomological correspondence and relative affine Grassmanian}

We first recall the geometric interpretation of the space of local special cohomological correspondences \[\Cor_{\hk^{\Xp,\loc}_{G},\IC_V}(\d^X_{\loc},\d^X_{\loc}\langle d\rangle)\] introduced in \thref{loccordef} using relative Grassmanian given in \cite[\S 8.2]{BZSV}.

We use $\Gr^X_G$ to denote the relative affine Grassmanian associated to the smooth affine $G$-variety $X$, which can be defined by defining the relative Hecke stack $\hk^{X,\loc}_G:=\Gr^X_G/L^+G\rtimes\Aut(D)$ via the Cartesian diagram:
\[\begin{tikzcd}
    \hk^{X,\loc}_G \ar[r, "i^{\hk}_{\loc}"] \ar[d, "\lh^{+,X}_{\loc}\times \rh^{+,X}_{\loc} " '] & \hk^{\Xp,\loc}_{G} \ar[d, "\lh^X_{\loc}\times \rh^X_{\loc}"] \\
    (L^+X/L^+G\rtimes\Aut(D))^2 \ar[r, "i_{\loc}^{\times 2}"] & (LX/L^+G\rtimes\Aut(D))^2\rcart
\end{tikzcd}.\]

Then we see that \begin{equation}\label{loccorrelgr}
    \begin{split}
        &\Cor_{\hk^{\Xp,\loc}_{G},\IC_V}(\d^X_{\loc},\d^X_{\loc}\langle d\rangle) \\
        =&\Hom^0(\IC_{V}*\d^X_{\loc},\d^X_{\loc}\langle d\rangle) \\
        \isom& \Hom^0(\IC_{V}*i_{\loc,*}\uk_{L^+X/L^+G\rtimes\Aut(D)},i_{\loc,*}\uk_{L^+X/L^+G\rtimes\Aut(D)}\langle d\rangle) \\
        \isom& \Hom^0(\IC_V,\lh^{+,X,!}_{\loc}\uk_{L^+X/L^+G\rtimes\Aut(D)}\langle d\rangle ) \\
        =& \Cor_{\hk^{X,\loc}_G,\IC_V}(\uk_{L^+X/L^+G\rtimes\Aut(D)},\uk_{L^+X/L^+G\rtimes\Aut(D)}\langle d\rangle)
    \end{split}.
\end{equation}

\subsubsection{Construction}

Now we assume $V\in\Rep(\Gc)^{\he,\om}$ is \emph{irreducible minuscule} with higher weight $\l_V$. More precisely, we assume that $\IC_V$ is \emph{the} (shifted) constant sheaf supported on the minuscule Schubert cell indexed by $\l_V$. For various versions of Hecke stacks, we add a subscript $V$ (or the extreme weight $\l_V$) to indicate the Schubert cell supporting $\IC_V$. We can further rewrite the vector space in \eqref{loccorrelgr} as \begin{equation}\label{loccorrelgrmin}
    \begin{split}
        &\Cor_{\hk^{\Xp,\loc}_{G},\IC_V}(\d^X_{\loc},\d^X_{\loc}\langle d\rangle) \\
        =& \Cor_{\hk^{X,\loc}_G,\IC_V}(\uk_{L^+X/L^+G\rtimes\Aut(D)},\uk_{L^+X/L^+G\rtimes\Aut(D)}\langle d\rangle) \\
        =& \Cor_{\hk^{X,\loc}_{G,V},\uk}(\uk_{L^+X/L^+G\rtimes\Aut(D)},\uk_{L^+X/L^+G\rtimes\Aut(D)}\langle d - d_{G,\l_V}\rangle)
    \end{split}
\end{equation}
where the number $d_{G,\l_V}=\langle 2\r_G,\l_V\rangle$.

To define the derived fundamental class, we make the following assumption:
\begin{assumption}\thlabel{Xlinear}
    There exists a split reductive subgroup $H\sub G$ and $Y\in\Rep(H)^{\he,\om}$ such that $X=Y\times^H G$.
\end{assumption}


Now we choose a minuscule coweight $\l\in X_*(T_H)$ such that $V$ has highest weight $\l_V=\l^+$. We refer to  \autoref{lie} for our notations concerning Lie theory.

\begin{defn}\thlabel{defect}
    For $\l\in X_*(T_H)$ and $Y\in\Rep(H)^{\he,\om}$, we define the \emph{defect module} of $\l$ with respect to $Y$ to be \[D_{\l,Y}:=Y[[t]]\cdot t^{-\lambda}/(Y[[t]] \cap Y[[t]]\cdot t^{-\l})\in\Rep(\cP_{H,\l}\rtimes\Aut(D))\] and the \emph{defect} of $\l$ with respect to $Y$ to be
    \[\d_{\l,Y}:=\dim (D_{\l,Y}) .\] Here \begin{equation}\label{parahoricg} \cP_{H,\l}:=t^{\l}L^+Ht^{-\l}\cap L^+H \end{equation} is a parahoric subgroup of $L^+H$.
\end{defn}

We define $V_{H,\l}\in\Rep(\Hc)^{\he,\om}$ to be the unique irreducible representation of $\Hc$ containing extremal weight $\l$.

Now, we are going to construct an element \[c_{V,\l}^{\loc}\in \Cor_{\hk^{\loc}_{G},\IC_V}(\d^X_{\loc},\d^X_{\loc}\langle d_{V,\l}\rangle)\] where \[d_{V,\l}:=d_{G,\l} - 2d_{H,\l} +2\d_{\l,Y}.\]

Consider diagram \begin{equation}\label{htogdiag}
    \begin{tikzcd}
        L^+Y/L^+H\rtimes\Aut(D) \ar[d, "a_{\loc}" ', "\sim"] & \hk^{Y,\loc}_{H,\l} \ar[l, "\lh^{+,Y}_{\loc}"'] \ar[r, "\rh^{+,Y}_{\loc}"] \ar[d, "a_{\loc}^{\hk}"] & L^+Y/L^+H\rtimes\Aut(D) \ar[d, "a_{\loc}", "\sim" '] \\
        L^+X/L^+G\rtimes\Aut(D) & \hk^{X,\loc}_{G,V} \ar[l, "\lh^{+,X}_{\loc}"'] \ar[r, "\rh^{+,X}_{\loc}"] & L^+X/L^+G\rtimes\Aut(D)
    \end{tikzcd}
\end{equation}
where left and right vertical arrows are isomorphisms. It is clear that both squares are pushable in the sense of  \autoref{pushabledef} and the push-forward functoriality of cohomological correspondences \thref{ccpush} gives us a map \begin{equation}\label{dfcpushloc}\begin{split}
       a_{\loc,!}^{\hk}:&\Cor_{\hk^{Y,\loc}_{H,\l},\uk}(\uk_{L^+Y/L^+H\rtimes\Aut(D)},\uk_{L^+Y/L^+H\rtimes\Aut(D)}\langle d - d_{G,\l}\rangle ) \\
       \to & \Cor_{\hk^{X,\loc}_{G,V},\uk}(\uk_{L^+X/L^+G\rtimes\Aut(D)},\uk_{L^+X/L^+G\rtimes\Aut(D)}\langle d - d_{G,\l}\rangle).   
\end{split}
\end{equation}
Define \begin{equation}\label{parahoric} \cP_{G,\l}=t^{\l}L^+G t^{-\l}\cap L^+G .\end{equation} Consider the correspondence diagram 
\begin{equation}
\adjustbox{max width=\linewidth}{%
    \begin{tikzcd}[column sep=small]
    &LY/\cP_{H,\l}\rtimes\Aut(D)\ar[d, "s", "\sim"']\ar[dl, "\lh^Y_{H,\loc}"  '] \ar[dr, "\rh^Y_{H,\loc}"]& \\
        LY/L^+H\rtimes\Aut(D) & LY\times^{L^+H\rtimes\Aut(D)}L^+Ht^{\lambda}L^+H\rtimes\Aut(D)/L^+H\rtimes\Aut(D) \ar[l, "\lh^Y_{H,\loc}"  ] \ar[r, "\rh^Y_{H,\loc}"'] & LY/L^+H\rtimes\Aut(D)\\
        &\hk^{\mathring{Y},\loc}_{H,\l}\ar[u, "\sim"]&
    \end{tikzcd}
    }
\end{equation}
where the lower half diagram is exactly \eqref{localcordiag} for $H$ restricted to the $\l$-Schubert cell. The isomorphism $s$ can be described as follows: One first defines the map 
\begin{equation}
    \tils:LY\to LY\times (L^+Ht^{\l}L^+H\rtimes\Aut(D))
\end{equation} as
\begin{equation}
    \tils(y)=(y,(t^{\l},1)).
\end{equation}
and defines
\begin{equation}
    s_{\cP}:\cP_{H,\l}\rtimes\Aut(D)\to (L^+H\rtimes\Aut(D))\times (L^+H\rtimes\Aut(D))
\end{equation} 
\begin{equation}
    s_{\cP}(p,a)=((p,a),(t^{-\l}pa^{-1}(t)^{\l},a)).
\end{equation}
Then one checks that the map $\tils$ is $s_{\cP}$-equivariant and descends to a map $s$, where we remind the reader that the right action of $\cP_{H,\l}\rtimes\Aut(D)$ on $LY$ is obtained by restricting the right action of $L^+H\rtimes\Aut(D)$ on $LY$ given by 
\begin{equation}
    y(t)\cdot(h(t),a)=y(a(t))h(a(t)) \in LY
\end{equation} for $y(t)\in LY$ and $(h(t),a)\in L^+H\rtimes\Aut(D)$, and the right action of $(L^+H\rtimes\Aut(D))^2$ on $LY\times (LH\rtimes\Aut(D))$ is given by 
\begin{equation}
    (y,(h,a))\cdot ((h_1,a_1),(h_2,a_2)) = (y\cdot (h_1,a_1),(h_1,a_1)^{-1}\cdot(h,a)\cdot (h_2,a_2))\in LY\times LH\rtimes\Aut(D)
\end{equation} for $(y,(h,a))\in LY\times (LH\rtimes\Aut(D))$ and $((h_1,a_1),(h_2,a_2))\in (L^+H\rtimes\Aut(D))^2$.

Under the isomorphism $s$, we can rewrite the upper row of \eqref{htogdiag} as the top of the diagram
\begin{equation}\label{dfccordiag}
\begin{tikzcd}
&L^+Y\cap L^+Y\cdot t^{-\l}/\cP_{H,\l}\rtimes\Aut(D) \ar[d, "s^+", "\sim"'] \ar[dl, "\lh^{+,Y}_{\loc}"'] \ar[dr, "\rh^{+,Y}_{\loc}"]&\\
    L^+Y/L^+H\rtimes\Aut(D) & \hk^{Y,\loc}_{H,\l} \ar[l, "\lh^{+,Y}_{\loc}"'] \ar[r, "\rh^{+,Y}_{\loc}"]  & L^+Y/L^+H\rtimes\Aut(D) 
    \end{tikzcd}
\end{equation}
such that for \[y\in L^+Y\cap L^+Y\cdot t^{-\l}\] a representative of an element in $L^+Y\cap L^+Y\cdot t^{-\l}/\cP_{H,\l}\rtimes\Aut(D)$, we have
\[\rh^{+,Y}_{\loc}(y)=y \in L^+Y /L^+H\rtimes\Aut(D) \]\[\lh^{+,Y}_{\loc}(y)=y\cdot t^{\l} \in L^+Y/L^+H\rtimes\Aut(D).\]

Now, it is easy to see that the map $\lh^{+,Y}_{\loc}$ is quasi-smooth of relative virtual dimension $d_{H,\l} - \d_{\l,Y}$. In fact, one can choose suitable $n,m,k\in\Z_{\geq 1}$ depending only on $(H,Y,\l)$ such that 
one can define \[L^{(m)}H:=L^+H/L^{\geq m}H\] \[\cP_{H,\l}^{(m)}=\cP_{H,\l}/L^{\geq m}H\] \[\Aut^{(k)}(D)=\Aut(D)/\Aut^{\geq k}(D)\] \[L^{(n)}Y:=L^+Y/t^nL^+Y.\]
and get a Cartesian diagram
\begin{equation}\label{derivedrelativehecke}
    \begin{tikzcd}
    L^+Y/L^+H\rtimes\Aut(D) \ar[d, "q_{\loc}"] &\hk^{Y,\loc}_{H,\l}\ar[l, "\lh^{+,Y}_{\loc}"'] \ar[d, "q_{\loc}^{\hk}"]  \\
    L^{(n)}Y/L^{(m)}H\rtimes\Aut^{(k)}(D)\lcart & \hk^{Y,\loc}_{H,\l,\fin}:=L^{(n)}Y\cap L^{(n)}Y\cdot t^{\l} /\cP_{H,\l}^{(m)}\rtimes\Aut^{(k)}(D) \ar[l, "\lh^{+,Y}_{\fin,\loc}" ']   
    \end{tikzcd}
\end{equation}
where the bottom map is quasi-smooth, hence makes $\lh^{+,Y}_{\loc}$ a quasi-smooth map in the sense of \thref{dfcgeneral}. Then the construction in \textit{loc.cit} gives a derived fundamental class \begin{equation}
\begin{split}
    [\lh^{+,Y}_{\loc}]& =[\hk^{Y,\loc}_{H,\l}/(L^+Y/L^+H\rtimes\Aut(D))] \\
    &\in \Hom^0(\uk_{\hk^{Y,\loc}_{H,\l}},\lh^{+,Y,!}_{\fin,\loc}\uk_{L^+Y/L^+H\rtimes\Aut(D)}\langle -2d_{H,\l} + 2\d_{\l,Y} \rangle ) \\
    &= \Cor_{\hk^{Y,\loc}_{H,\l},\uk}(\uk_{L^+Y/L^+H\rtimes\Aut(D)},\uk_{L^+Y/L^+H\rtimes\Aut(D)}\langle d_{V,\l} - d_{G,\l}\rangle )
    \end{split}.
\end{equation}

We can finally define \begin{equation}\label{dfcloc}
    \begin{split}
        c_{V,\l}^{\loc}:= a_{\loc,!}^{\hk}[\lh^{+,Y}_{\loc}]&\in  \Cor_{\hk^{X,\loc}_{G,V},\uk}(\uk_{L^+X/L^+G\rtimes\Aut(D)},\uk_{L^+X/L^+G\rtimes\Aut(D)}\langle d_{V,\l} - d_{G,\l}\rangle)\\
        &\isom \Cor_{\hk^{\Xp,\loc}_{G},\IC_V}(\d^X_{\loc},\d^X_{\loc}\langle d_{V,\l}\rangle)
    \end{split}
\end{equation}
where the map $a_{\loc,!}^{\hk}$ is the map \eqref{dfcpushloc}. This finishes the definition of local special cohomological correspondences via the derived fundamental class.

\subsubsection{Pull-back to global moduli}
In this section, we restrict to the case that $X$ is $G$-homogeneous. In this case, we have $X=H\backslash G$ for a split reductive subgroup $H\sub G$, and we have $\Bun_G^X=\Bun_H$. In this case, we have defect $\d_{\l,Y}=0$ and $d_{V,\l}=d_{G,\l} - 2d_{H,\l}$.
We now describe the global cohomological correspondences $c_{V,\l}$ constructed in \eqref{removeleg} given the input $c^{\loc}_{V,\l}$ constructed in \eqref{dfcloc} (i.e. consider the local datum $(\fone,V,c_{V,\l}^{\loc})$ in the language of \thref{loctovarious} and \thref{globccdef}).

\begin{prop}\thlabel{fundtoglob}
    Consider the diagram
    \[\begin{tikzcd}
        \Bun_H \ar[d, "\pi" '] & \hk_{H,\l}^{\glob} \ar[d, "\pi^{\hk}_{\l}"] \ar[l, "\lh_{H,\glob}" '] \ar[r, "\rh_{H,\glob}"] &\Bun_H \ar[d, "\pi"] \\
        \Bun_G & \hk_{G,V}^{\glob} \ar[l, "\lh_{\glob}" '] \ar[r, "\rh_{\glob}"] & \Bun_G
    \end{tikzcd}.\]
    The push-forward functoriality in \thref{ccpush} gives us a map \[\begin{split}\pi_{\l,!}^{\hk}:\Cor_{\hk_{H,\l}^{\glob},\uk}(\uk_{\Bun_H},\uk_{\Bun_H}\langle -2-2d_{H,\l} \rangle)&\to \Cor_{\hk_{G,V}^{\glob},\uk}(\pi_!\uk_{\Bun_H},\pi_!\uk_{\Bun_H}\langle -2-2d_{H,\l}\rangle )\\ &\isom\Cor_{\hk_{G}^{\glob},\IC_V}(\cP_X,\cP_X\langle d_{V,\l}-2\rangle)\end{split}.\]
    Then the element $c_{V,\l}$ constructed in \eqref{removeleg} from the local fundamental class $c_{V,\l}^{\loc}$ in \eqref{dfcloc} can be identified as \[c_{V,\l}=\pi_{\l,!}^{\hk}[\hk_{H,\l}^{\glob}/\Bun_H]\] where $[\hk_{H,\l}^{\glob}/\Bun_H]=[\lh_{H,\glob}]$ is the fundamental class introduced in \thref{dfcgeneral}.
\end{prop}
\begin{proof}
    Applying the base change identity \thref{ccbc} in which the middle square is given by
    \[\begin{tikzcd}
    \hk^{X,\glob}_{G,V} \ar[d, "f_{G,\gtol}^{X,\hk}" ']& \hk^{\glob}_{H,\l} \ar[d, "f_{H,\gtol}^{\hk}"] \ar[l, "a_{\glob}^{\hk}" '] \\
    \hk^{X,\loc}_{G,V} & \hk^{\loc}_{H,\l} \ar[l, "a_{\loc}^{\hk}"']
\end{tikzcd},\] we get a commutative diagram
\[\begin{tikzcd}
    \Cor_{\hk^{\loc}_{H,\l},\uk}(\uk_{BL^+H\rtimes\Aut(D)},\uk_{BL^+H\rtimes\Aut(D)}\langle -2-2d_{H,\l} \rangle)\ar[r, "a_{\loc,!}^{\hk}"] \ar[d, "f_{H,\gtol}^{\hk,*}" '] & \Cor_{\hk_{G,V}^{X,\loc},\uk}(\d_X^{\loc},\d_X^{\loc}\langle -2-2d_{H,\l} \rangle) \ar[d, "f_{G,\gtol}^{X,\hk,*}"] \\
    \Cor_{\hk_{H,\l}^{\glob},\uk}(\uk_{\Bun_H\times C},\uk_{\Bun_H\times C}\langle -2-2d_{H,\l} \rangle) \ar[r, "a_{\glob,!}^{\hk}"] & \Cor_{\hk_{G,V}^{X,\glob},\uk}(\d_X^{\glob},\d_X^{\glob}\langle -2-2d_{H,\l} \rangle)
\end{tikzcd}.\]
Consider $\pi_{!}^{\hk}:\Cor_{\hk_{G,V}^{X,\glob},\uk}(\d_X^{\glob},\d_X^{\glob}\langle -2-2d_{H,\l} \rangle)\to \Cor_{\hk_{G,V},\uk}(\cP_X\boxtimes\uk_C,\cP_X\boxtimes\uk_C\langle -2-2d_{H,\l} \rangle)$. Unwinding the definition of $c_{V,\l}$, we get \[\begin{split}c_{V,\l}&=[C]\circ(\pi_!^{\hk}f_{G,\gtol}^{X,\hk,*}a_{\loc,!}^{\hk}[\hk_{H,\l}^{\loc}/BL^+H\rtimes\Aut(D)])\circ\taut_C\\
&=[C]\circ\pi_!^{\hk}a_{\glob,!}^{\hk}f_{H,\gtol}^{\hk,*}[\hk_{H,\l}^{\loc}/BL^+H\rtimes\Aut(D)]\circ\taut_C\\
&=[C]\circ\pi_!^{\hk}a_{\glob,!}^{\hk}[\hk_{H,\l}^{\glob}/\Bun_H\times C]\circ \taut_C\\
&=\pi_{\l,!}^{\hk}[\hk_{H,\l}^{\glob}/\Bun_H]
\end{split}\] which concludes the proof.
    
\end{proof}

\subsection{Middle-dimensional correspondence}\label{middimcor}
Assume we are under \thref{commsetup}. In this section, we consider a specific case of \thref{geokolycor} that we care about the most: the middle-dimensional case, that is, the case $d_V+d_W=2$. We say that this case is middle-dimensional because when $|I|=2$, the associated special cycle classes are middle-dimensional Borel-Moore homology classes in this case. 

In case $d_V+d_W=2$, we see that $c^{\loc}_{V\otimes W}\in \Hom^0(V\otimes W,\PL_{X,\hbar})$. Note that there is a canonical element $c^{\loc}_{\triv}\in\Hom^0(\triv,\PL_{X,\hbar})$ where $\triv\in\Rep(\Gc)$ is the trivial $\Gc$-representation. To single out the most interesting example, we make the following assumptions:

\begin{assumption}[Poisson-purity]\thlabel{middimassum}
We have $d_V+d_W=2$, and there exists a bilinear form $b:V\otimes W\to \triv$
such that the map \[c_{V\otimes W}^{\loc}\in\Hom^0(V\otimes W,\PL_{X,\hbar})\] is given by \begin{equation}\label{pos=symp}
    V\otimes W\xrightarrow{b}\triv\xrightarrow{c^{\loc}_{\triv}}\PL_{X,\hbar}.
\end{equation}
\end{assumption}

Now we further make the following assumption:
\begin{assumption}[Minuscule]\thlabel{minassum}
    We have $V,W\in\Rep(\Gc)^{\he,\om}$ are irreducible minuscule representations of $\Gc$. More specifically, we take $V, W\in\Rep(\Gc)^{\he}\isom\Sat_G^{\he}$ to be \emph{the} (shift of) constant sheaf supported on the corresponding minuscule Schubert cell.
\end{assumption}

\begin{remark}
    The minuscule assumption is only necessary when applying the construction in \autoref{dfc}.
\end{remark}

In this case, whenever a non-zero bilinear form $b: V\otimes W\to \triv $ exists, there would exist a canonical bilinear form \begin{equation}\label{geoform} b_{\geo}: V\otimes W\to \triv\end{equation} (or more precisely, a map $\IC_V*\IC_W\to\IC_{\triv}$) coming from geometry which we are going to explain now.

Suppose $V$ is of higher (minuscule) weight $\l_V\in X_*(T)^+$ (hence $W$ has lowest weight $-\l_V$). We use \begin{equation}\label{Vparabolic} P_{G,\l_V}\end{equation} to denote the parabolic subgroup of $G$ generated by roots $\{\a\in \Phi_G|\langle \a, \l_V\rangle \leq 0\}$ (hence contains the opposite Borel). Then it is easy to see that $\Hom_{\Sat_G^{\he}}(\IC_V*\IC_W,\IC_{\triv}) = H^0(G/P_{G,\l_V},\uk)$. Therefore, we can take $b_{\geo}:V\otimes W\to \triv$ to be given by the canonical element $1\in H^0(G/P_{G,\l_V},\uk)$ (the fundamental class of $G/P_{\mu}$ under Poincare duality).

When our affine smooth $G$-variety $X$ has the form $X=Y\times^H G$, for each $\l\in X_*(T_H)$ such that $\l^{+}=\l_V$ we have the derived fundamental class $c_{V,\l}^{\loc}$ constructed in \eqref{dfcloc} and similarly for $\mu\in X_*(T_H)$ where $\mu^+=\l_W$ a class $c^{\loc}_{W,\mu}$. Assume \thref{middimassum} is satisfied, we denote the bilinear form in \textit{loc.cit} by \begin{equation}\label{Xform}b_{\l,\mu}:V\otimes W\to\triv.\end{equation}

Now, it is a natural question to ask about the relation between $b_{\l,\mu}$ and $b_{\geo}$. Since the former $b_{\l,\mu}$ involves the choice of the $G$-variety $X$ as well as some choices of $\l,\mu\in X_*(T_H)^+$, while the second $b_{\geo}$ does not involve any extra choice, these two forms a priori have no relation and would differ by a scalar in general.  We denote \begin{equation}\label{geometricscalar} b_{\l,\mu}=\kappa_{\l,\mu}\cdot b_{\geo}.\end{equation} 

\subsubsection{Automorphic Clifford relations}

We have the following reformulation of \thref{geokolycor} under the strong assumptions in the previous section, which is what we call the automorphic Clifford relation:
\begin{cor}[Automorphic Clifford relation]\thlabel{geokolycormidmin}
        Under \thref{commsetup}. Suppose \thref{middimassum} and \thref{minassum} hold. Consider the diagram \begin{equation}\label{autoclifdiag1}\begin{tikzcd}
        \Bun_G \ar[d, "\id" '] & \hk^{\glob}_{G,\fone,V} \ar[l, "b_{\glob}" '] \ar[r, "b_{\glob}"] \ar[d, "\D^{\hk}_{V,\glob}"] & \Bun_G \ar[d, "\id"] \\
        \Bun_G & \thk_{G,\ftwo,V\boxtimes W}^{\glob} \ar[l, "\lh_{\ftwo,\glob}" '] \ar[r, "\rh_{\ftwo,\glob}"] & \Bun_G
    \end{tikzcd},\end{equation} where $b_{\glob}=\lh_{\fone,\glob}$ and the map $\D^{\hk}_{V}$ is given by sending a modification $\cE\dashrightarrow \cE'$ to the iterated modification $\cE\dashrightarrow \cE'\dashrightarrow \cE$.
    We get a map \[\begin{split}\D^{\hk}_{V,\glob,!}:\Cor_{\hk^{\glob}_{G,\fone,V},\uk}(\cP_X,\cP_X\langle -2-2d_{G,\l_V}\rangle)&\to\Cor_{\thk_{G,\ftwo,V\boxtimes W}^{\glob},\uk}(\cP_X,\cP_X\langle -2-2d_{G,\l_V}\rangle)\\ &\isom \Cor_{\hk_{G,\ftwo,V\boxtimes W}^{\glob},\IC_{V\boxtimes W}}(\cP_X,\cP_X\langle -2\rangle)\end{split}\] via the push-forward functoriality in \thref{ccpush}. Assuming \thref{geokolycor}, we have \[c_{V\boxtimes W} - c_{W\boxtimes V} = \kappa_{\l,\mu}\cdot\D_{V,\glob,!}^{\hk}[\hk^{\glob}_{G,\fone,V}/\Bun_G],\] where $c_{V\boxtimes W}$ and $c_{W\boxtimes V}$ are defined in \eqref{removeleg} and $[\hk^{\glob}_{G,\fone,V}/\Bun_G]$ is the (relative) fundamental class defined in \thref{dfcgeneral} viewed as a cohomological correspondence.
\end{cor}
\begin{proof}
This follows easily from the base change compatibility between pull-back and push-forward of cohomological correspondences in \thref{ccbc}.
\end{proof}

\subsubsection{Examples}\label{sec:examples}

There are many examples that \thref{middimassum} and \thref{minassum} hold.
\begin{example}[Rankin--Selberg convolution: homogeneous case]\thlabel{rsconv}
    Case $G=\GL_n\times \GL_{n-1}$ and $X=\GL_n$. See the discussion in  \autoref{higherrs}.

\end{example}

\begin{example}[Rankin--Selberg convolution: inhomogeneous case]\thlabel{rsconvtwo}
Case $G=\GL_n\times \GL_n$ and $X=\GL_n\times\A^n$, on which the action in given by $(g,x)\cdot(g_1,g_2)=(g_1^{-1}gg_2,x\cdot g_2)$ for $(g_1,g_2)\in\GL_n\times \GL_n$ and $(g,x)\in \GL_n\times\A^n$. One has $\Mc=T^*(\std_n\boxtimes\std_n)$ and \thref{bzsvloc} is proved in \cite{BFGT}. One can take $V=\std_n\boxtimes\std_n$ and $W=V^*$, which uniquely determines the coweights $\l,\mu\in X_*(T_H)$ up to Weyl group translation. In this case, \thref{middimassum} and \thref{minassum} are satisfied. 
\end{example}

\begin{example}[Orthogonal Gan--Gross--Prasad]\thlabel{ggp}
Case $G=\SO_n\times\SO_{n-1}$ and $X=\SO_n$, on which the action in given by $x\cdot(g_1,g_2)=g_1^{-1}xg_2$ for $(g_1,g_2)\in\SO_n\times\SO_{n-1}$ and $x\in \SO_n$. One has $\Mc=\std\boxtimes\std$ as a symplectic representation of $\Gc$. \thref{bzsvloc} in this case is proved in \cite{Braverman_2022}. One can take $W=V=\std\boxtimes\std$, which uniquely determines the coweights $\l,\mu\in X_*(T_H)$ up to Weyl group translation. In this case, \thref{middimassum} and \thref{minassum} are satisfied.
\end{example}

\begin{example}[Strongly tempered case]
Case $T^*X$ is strongly tempered, the hyperspherical dual $\Mc$ is a symplectic representation of $\Gc$. This includes \thref{rsconv}, \thref{rsconvtwo}, and \thref{ggp}. In this case, \thref{bzsvloc} predicts that one can take $V$ and $W=V^*$ to be constituents of $\Mc$, and then  \thref{middimassum} will be satisfied. When $X=H\bs G$ is homogeneous, by checking all such examples in \cite{wan2023periodsautomorphicformsassociated}, we find that for each irreducible $\Gc$-representation $V$, the multiplicity of $V$ in $\Mc$ is exactly the number of $\l\in X_*(T_H)^+$ such that $\l^+=\l_V$ and $d_{V,\l}=1$. The representation $\Mc$ is multiplicity-free and has all its constituents minuscule in all such examples in \textit{loc.cit}. Therefore, the construction of local special cohomological correspondences via \autoref{dfc} should provide a \emph{canonical} presentation of $\Mc$ as a direct sum of minuscule irreducible $\Gc$-representations.
\end{example}

\begin{example}[Linear period]\thlabel{fj}
When $G=\GL_{2n}$, $H=\GL_n\times\GL_n$, and $X=H\bs G$. The hyperspherical dual $\Mc=(T^*(\std_{2n})\oplus (\mathfrak{gl}_{2n}/\mathfrak{sp}_{2n})^*)\times^{\Sp_{2n}}\GL_{2n}$. One can take $V=\Std_{2n}\in\Rep(\GL_{2n})$ and $W=V^*$. In this case, there will be multiple choices of $\l\in X_*(T_H)$ such that $\l^+=\l_V$: One can take $\l=\l_1:=((1,0\dots,0),(0,\dots,0))$ or $\l=\l_2:=((0,\dots,0),(1,0,\dots,0))$, and these are all possible choices of $\l$ up to $W_H$ permutation. In this case, \thref{bzsvloc} predicts that \thref{middimassum} is satisfied for $(\l,\mu)=(\l_1,-\l_1)$ and $(\l,\mu)=(\l_2,-\l_2)$.
\end{example}

\subsection{Determine the constant}\label{verclifrel}
This section will only be used in \autoref{localinput} and can be safely skipped on a first reading.

In this section, we develop an effective way to determine the number $\kappa_{\l,\mu}$ in \eqref{geometricscalar} under \thref{middimassum} and \thref{minassum}. This method will be applied later in the Rankin--Selberg case  \autoref{geometricscalarglnsection}. The method also works for \thref{rsconvtwo}, \thref{ggp}, \thref{fj}, and many other examples. We leave the details to the reader.

\begin{remark}
    We are not developing a method to verify \thref{middimassum}. We do not know an effective way to verify the existence of $b$ in \textit{loc. cit} when it indeed exists. Instead, we are assuming the existence of $b$, and try to determine the relation between $b$ and $b_{\geo}$ defined via \eqref{pos=symp} and \eqref{geoform}. However, when the Soergel functor \eqref{Xsoergel} is fully faithful on the objects we care about,\footnote{This is highly possible when $X$ is homogeneous and strongly tempered. The full faithfulness has been proved for \thref{rsconv}, \thref{ggp} in \cite{BFGT},\cite{Braverman_2022}.} our method can be used to prove the existence of $b$.
\end{remark}

\subsubsection{Setup}

We refer to  \autoref{lie} for notations on Lie theory. 

We first work under the general setup \thref{commsetup}.

Define \begin{equation} R_G:=\Gamma(B(L^+G\rtimes\Aut(D)),\uk)\isom\Gamma(BG\times B\Gm,\uk) \end{equation} and \begin{equation} \overline{R}_G:=\Gamma(BL^+G,\uk)\isom\Gamma(BG,\uk) .\end{equation} We have $R_G=\overline{R}_G[\hbar]$. Moreover, it is well-known that $R_G=R_T^{W_G}$ and $R_T\isom \cO(\frt)[\hbar]$.

Consider the functor taking global sections \begin{equation}\label{Xsoergel}
    \Gamma:\Sat_{X,\hbar}=\Shv_*(LX/L^+G\rtimes\Aut(D))\to \Mod(R_G)
\end{equation} and \begin{equation}
    \Gamma:\Sat_{G,\hbar}=\Shv_*(\hk^{\loc}_G)\to \Mod(R_G\otimes_{k[\hbar]} R_G).
\end{equation} We have the formula \begin{equation} \Gamma(\cK*\cF)=\Gamma(\cK)\otimes_{R_{G}}\Gamma(\cF) \end{equation} for nice\footnote{For $\cK$ belonging to the full-subcategory of $\Sat_{X,\hbar}$ generated by $\Sat_{G,\hbar}$ and $\d^X_{\loc}$, the argument in \cite[Corollary\,B.4.2]{bezrukavnikov2013koszul} works. This is enough for our purpose.} objects $\cK\in\Sat_{G,\hbar}$ and $\cF\in\Sat_{X,\hbar}$.

For each $V\in\Rep(\Gc)^{\he,\om}$, define \begin{equation}\label{MV} M_{V}:=\Gamma(\IC_V)\in\Mod(R_G\otimes_{k[\hbar]} R_G)\end{equation} and \begin{equation}\overline{M}_{V}:=M\otimes_{k[\hbar]}k.\end{equation} 

Now assume \thref{Xlinear} and use the notations there. We know \begin{equation}
    R_H=\Gamma(\d^X_{\loc})\in\Mod(R_G)
\end{equation} where the $R_G$-module structure on $R_H$ is given by the ring map $R_G\to R_H$ induced by the natural map $BH\times \Gm\to BG\times\Gm$ which is inclusion on the first actor and identity on the second factor.

Given each local special cohomological correspondence \begin{equation} c_V^{\loc}\in\Hom^0(\IC_V*\d^X_{\loc},\d^X_{\loc}\langle d_V\rangle),\end{equation} we get a map \begin{equation}
c_V^{\coh}:=\Gamma(c_V^{\loc})\in\Hom^0(M_{V}\otimes_{R_G}R_H,R_H\langle d_V\rangle)
\end{equation} 

Therefore, by the construction of $c_{V\otimes W}^{\loc}$ in \itemref{genassum}{comassum}, we get 
\begin{equation}\label{cohcom}
    c_V^{\coh}\circ c_W^{\coh}-c_W^{\coh}\circ c_V^{\coh}\circ \sw_{V,W} = \hbar\cdot c_{V\otimes W}^{\coh}\in\Hom^0(M_{V}\otimes_{R_G}M_{W}\otimes_{R_G}R_H,R_H\langle d_V+d_W\rangle).
\end{equation}
where $\sw_{V,W}:M_V\otimes_{R_G}M_W\isom M_W\otimes_{R_G}M_V$ is induced from the unique element in \[\Hom^0_{\Sat_{G,\hbar}}(\IC_V*\IC_W,\IC_W*\IC_V)\]  lifting the natural commutativity constraint in $\Sat_G$ obtained from the fusion procedure. \footnote{Here, when writing a map whose source is a tensor product and the map is a priori only defined on certain tensor factors, we always mean extending the map trivially to other tensor factors. We will use this convention throughout this section without mention to simplify notations.}

From now on, we work under \thref{middimassum} and \thref{minassum}, and we take $c_V^{\loc}=c_{V,\l}^{\loc}, c_W^{\loc}=c_{W,\mu}^{\loc}$ constructed in \eqref{dfcloc}. In this case, we have a natural map $b^{\coh}_{\geo}:M_V\otimes_{R_G}M_W\to M_{\triv}=R_G$ by applying $\Gamma$ to $b_{\geo}$ in \eqref{geoform}. Then the map $c_{V\otimes W}^{\coh}$ can be identified as
\begin{equation}
    c_{V\otimes W}^{\coh}=\kappa_{\l,\mu}\cdot b_{\geo}^{\coh}:M_{V}\otimes_{R_G}M_{W}\otimes_{R_G}R_H\to R_H
\end{equation}

The relation \eqref{cohcom} becomes
\begin{equation}\label{cohclif}
    c_{V,\l}^{\coh}\circ c_{W,\mu}^{\coh} - c_{W,\mu}^{\coh}\circ  c_{V,\l}^{\coh}\circ \sw_{V,W} = \kappa_{\l,\mu}\hbar \cdot b_{\geo}^{\coh}
\end{equation}

Therefore, to determine the scalar $\kappa_{\l,\mu}$, we only need to write down the maps $c_{V,\l}^{\coh}$, $c_{W,\mu}^{\coh}$, $b_{\geo}^{\coh}$, and $\sw_{V,W}$. The determination of these maps will be the subject of the rest of this section.

\subsubsection{Description of $M_V$}\label{mv}
We first describe $M_V\in \Mod(R_G\otimes_{k[\hbar]}R_G)$ defined in \eqref{MV}.

Recall that the parabolic subgroup $P_{G,\l}\sub G$ in \eqref{Vparabolic}. Similarly, we have the parabolic subgroup $P_{H,\l}\sub H$.

We define $W_{G,\l}:=W_{P_{G,\l}}\sub W_G$ and $W_{H,\l}:=W_{P_{H,\l}}\sub W_H$ which are the Weyl groups of the Levi groups of the corresponding parabolic subgroups. 

In this case, we have \begin{equation} M_V\isom \Gamma(B(\cP_{G,\l}\rtimes\Aut(D)),\uk)\langle d_{G,\l}\rangle\isom \cO(\frt)^{W_{G,\l}}[\hbar]\langle d_{G,\l}\rangle\in\Mod(R_G\otimes_{k[\hbar]} R_G)  .\end{equation} The $R_G$-bimodule structure can be seen as follows: The right $R_G$-module structure on $M_V$ is induced by the natural inclusion map $\rh_{\loc,V}:B(\cP_{G,\l}\rtimes\Aut(D))\to B(L^+G\rtimes\Aut(D)) $ which is the natural inclusion \begin{equation} i_{G,\l}:\cO(\frt)^{W_G}[\hbar]\to \cO(\frt)^{W_{G,\l}}[\hbar]. \end{equation} The left $R_G$-module structure on $M_V$ is induced by the map $\lh_{\loc,V}:B(\cP_{G,\l}\rtimes\Aut(D))\to B(L^+G\rtimes\Aut(D))$ which is identified with the map\footnote{This is because on homotopy types $\lh_{\loc,V}$ can be identified with $B(P_{G,\l}\times \Gm)\to B(G\times\Gm)$ such that $(p,a)\mapsto (p\l(a^{-1}),a)$, which is an easy consequence of \eqref{dfccordiag}.} \begin{equation}\label{twistincl} \tili_{G,\l}: \cO(\frt)^{W_G}[\hbar]\to\cO(\frt)^{W_{G,\l}}[\hbar] \end{equation} induced from the pull-back along $\frt\times \A^1\to \frt\times\A^1$ given by $(X,a)\mapsto (X-a\cdot d\l,a)$.

\subsubsection{Description of $c_{V,\l}^{\coh}$}\label{cv}

Now we describe the map $c_{V,\l}^{\coh}:M_V\otimes_{R_G}R_H\to R_H$.

Let $r:\cO(\frt)[\hbar]\to\cO(\frt_H)[\hbar]$ be the map induced by the natural inclusion. Then it is easy to see from definition of $c_{V,\l}^{\loc}$ in \eqref{dfcloc} that $c_{V,\l}^{\coh}$ can be identified with 
\begin{equation}
\begin{split}
    &\cO(\frt)^{W_{G,\l}}[\hbar]\otimes_{\cO(\frt)^{W_G}[\hbar]}\cO(\frt_H)^{W_H}[\hbar] \langle d_{G,\l}\rangle \\ \xrightarrow{r\otimes i_{H,\l}} &\cO(\frt_H)^{W_{H,\l}}[\hbar]\langle d_{G,\l}\rangle \\ \xrightarrow{\lh^{+,Y}_{\loc,!}} &\cO(\frt_H)^{W_H}[\hbar]\langle d_{G,\l}-2d_{H,\l}+2\d_{\l,Y}\rangle
    \end{split}
\end{equation}
where the map $\lh^{+,Y}_{\loc}$ is the top arrow in \eqref{derivedrelativehecke}, and by $\lh^{+,Y}_{\loc,!}$ we mean integration along the derived fundamental class $[\lh^{+,Y}_{\loc}]$. The map $r:\cO(\frt)^{W_{G,\l}}[\hbar]\to \cO(\frt_H)^{W_{H,\l}}[\hbar]$ is induced from the natural inclusion $\frt_H\to \frt$, and the map $i_{H,\l}:\cO(\frt_H)^{W_H}[\hbar]\to \cO(\frt_H)^{W_{H,\l}}[\hbar]$ is defined in the same way as $i_{G,\l}$.

Furthermore, consider $\ch(D_{\l,Y})\in \ZZ_{\geq 0}\cdot X^*(T_H\times\Gm)$ which is the character of the defect module $D_{\l,Y}$ defined in \thref{defect}. We can take its top Chern class \begin{equation}\label{defectchern}
    c_{\top}(\ch(D_{\l,Y}))\in \cO(\frt_H)[\hbar]
\end{equation}  Let us also consider the map \begin{equation}\label{parabolicpush} p_{H,\l}:\cO(\frt_H)^{W_{H,\l}}\to \cO(\frt_H)^{W_{H}}\langle -2d_{H,\l}\rangle \end{equation} given by integration along $BP_{H,\l}\to BH$. It is easy to see that the map 
\begin{equation}
    \lh^{+,Y}_{\loc,!}:\cO(\frt_H)^{W_{H,\l}}[\hbar]\to \cO(\frt_H)^{W_H}\langle-2d_{H,\l}+2\d_{\l,Y}\rangle.
\end{equation}
can be identified with the composition
\begin{equation}
    \cO(\frt_H)^{W_{H,\l}}[\hbar]\xrightarrow{\cdot c_{\top}(\ch(D_{\l,Y}))}\cO(\frt_H)^{W_{H,\l}}[\hbar]\langle 2\d_{\l,Y}\rangle \xrightarrow{p_{H,\l}} \cO(\frt_H)^{W_H}[\hbar]\langle-2d_{H,\l} +2\d_{\l,Y}\rangle.
\end{equation}

\begin{remark}\thlabel{parabolicpushdes}
    The map $p_{H,\l}$ is the unique graded $\cO(\frt_H)^{W_{H}}$-linear map sending $c_{\top}(\ch(\frh/\Lie(P_{H,\l})))\in \cO(\frt_H)^{W_{H,\l}}\langle 2d_{H,\l}\rangle$ to $|W_H/W_{H,\l}|\in \ZZ \sub \cO(\frt_H)^{W_H}$.\footnote{Since we are regarding $\frg$ as a \emph{right} $G$-module, our $\ch(\frh/\Lie(P_{H,\l}))$ is different from the usual convention by a minus sign.} Here, the $\cO(\frt_H)^{W_H}$-linear structure on $\cO(\frt_H)^{W_{H,\lambda}}$ is given by $\tili_{H,\lambda}$, which is defined in the same way as $\tili_{G,\lambda}$.
\end{remark}

\subsubsection{Description of $b_{\geo}^{\coh}$}\label{geocoh}
Now we describe the map $b_{\geo}^{\coh}:M_V\otimes_{R_G}M_W\to R_G$.

Choose $w\in W_G$ such that $w(\mu)=-\l$. Consider the map $a_{G,\l,\mu}:\cO(\frt)^{W_{G,\l}}[\hbar]\to \cO(\frt)^{W_{G,\mu}}[\hbar]$ induced by pull-back along the map $\frt\times\A\to\frt\times\A^1$ given by $(X,a)\mapsto (wX-a\cdot d\l,a)$. Then the map $b_{\geo}^{\coh}$ can be identifies with 
\begin{equation}
    \cO(\frt)^{W_{G,\l}}[\hbar]\langle d_{G,\l}\rangle\otimes_{\cO(\frt)^{W_G}[\hbar]}\cO(\frt)^{W_{G,\mu}}[\hbar]\langle d_{G,\l}\rangle\xrightarrow{\id\otimes a_{G,\l,\mu}}\cO(\frt)^{W_{G,\l}}[\hbar]\langle 2d_{G,\l}\rangle \xrightarrow{p_{G,\l}} \cO(\frt)^{W_{G}}[\hbar]
\end{equation}
where the map $p_{G,\l}$ is defined similarly as \eqref{parabolicpush}.

\subsubsection{Description of $\sw_{V,W}$}\label{sw}
Now we describe the map $\sw_{V,W}:M_V\otimes_{R_G} M_W\isom M_W\otimes_{R_G}M_V$. 

By \cite[\S6.2, Lemma\,13]{BF}, we know that $\sw_{V,W}$ is the unique graded map lifting $\overline{M}_V\otimes_{\overline{R}_G}\overline{M}_W\isom \overline{M}_W\otimes_{\overline{R}_G}\overline{M}_V $ given by the natural commutativity constraint in $\Mod(\overline{R}_G)$ (i.e. $x\otimes y\mapsto (-1)^{|x||y|}y\otimes x$)

\begin{remark}
    Although the above gives a characterization of the map $\sw_{V, V^*}$, we do not know a general way to write it down for all minuscule $V$.
\end{remark}

\section{Higher period integrals and derivatives of \texorpdfstring{$L$}{L}-functions}\label{higherperiodintegrals}
In this section, we combine the results in previous sections to get a higher period integral formula (\thref{higherformulageneral}).
\begin{itemize}
    \item In \autoref{geometricisotypicpart}, we recall the geometric isotypic part of the $X$-period, which has been considered in many previous works. 
    \item In \autoref{heckeactionsection}, we discuss the Hecke action on the geometric isotypic part.
    \item In \autoref{middimcase}, we restrict to the Poisson-pure case and get the main result \thref{higherformulageneral}.
\end{itemize}

\subsection{Geometric isotypic part of period integral}\label{geometricisotypicpart}
For any affine smooth $G$-variety $X$, we have the period sheaf \[\cP_X\in\Shv(\Bun_G)\] constructed in  \autoref{periodsheaf}. We would like to construct some ``higher period integrals", which can be related to higher derivatives of $L$-functions after restricting the ``integrals" to the ``$\pi$-isotypic part" for every (everywhere unramified) cuspidal automorphic representation $\pi\in\Irr(G(\A_F))$. For this purpose, we need to make sense of the ``isotypic part of the period" and reasonably define higher integrals.

\subsubsection{Geometric period integral}
We now come to the definition of geometric period integral, which is a geometric enhancement of the period integral.

\begin{defn}\thlabel{geoperiod}
Consider $\pi:\Bun_G^X\to \Bun_G$. We define the \emph{geometric $X$-period integral} to be the functor $\int_X:\Shv(\Bun_G)\to \Vect$ such that $\int_X=\Gamma_c\circ \pi^*$.
\end{defn}

\begin{remark}
    Here we are using $!$-push for morphisms locally of finite type between Artin stacks locally of finite type. We did not include this in our $\Shv$ due to laziness in writing the article. However, by \thref{openpresen}, our sheaf theory is the same as the usual sheaf theory on Artin stacks locally of finite type. Therefore, we can use the six-functor formalism developed in \cite{liu2017enhancedadicformalismperverse}.
\end{remark}

\begin{remark}
The functor $\int_X$ can be understood using the period sheaf as follows: Consider the functor $\Gamma_c\circ\D_{\Bun_G}^*:\Shv(\Bun_G\times \Bun_G)\to \Vect$, where $\D_{\Bun_G}:\Bun_G\to\Bun_G\times\Bun_G$ is the diagonal map. By base change, we have $\int_X=\Gamma_c\circ\D_{\Bun_G}^*(\cP_X\boxtimes-)$. By \cite[Corollary\,3.3.7]{arinkin2022dualityautomorphicsheavesnilpotent}, we know that $\Gamma_c\circ\D_{\Bun_G}^*|_{\Shv_{\Nilp}(\Bun_G)^{\otimes 2}}=\ev_{\Mir}$ is the counit map for the Miraculous duality on $\Shv_{\Nilp}(\Bun_G)$. By \cite[Proposition\,3.4.6]{arinkin2022dualityautomorphicsheavesnilpotent}, we know $\int_X|_{\Shv_{\Nilp}(\Bun_G)}=\ev_{\Mir}(\cP_X^r\otimes-)$, where $\cP_X^r=\mathrm{P}_{\Nilp}*\cP_X$ is the image of $\cP_X$ under the right adjoint (to the forgetful functor) $\mathrm{P}_{\Nilp}:\Shv(\Bun_G)\to\Shv_{\Nilp}(\Bun_G)$ which is given by the Beilinson's spectral projector. Note that it is conjectured in \cite[\S12.4]{BZSV} that the period sheaf after spectral projection $\cP_X^r$ should correspond to the $L$-sheaf defined in \textit{loc.cit.}. This gives a way to understand $\int_X|_{\Shv_{\Nilp}(\Bun_G)}$ in terms of miraculous duality and the global conjecture made in \cite{BZSV}.
\end{remark}

\begin{remark}
    Note that the functor $\int_X|_{\Shv_{\Nilp}(\Bun_G)}:\Shv_{\Nilp}(\Bun_G)\to\Vect$ preserves compact objects. This observation will be used in \cite{trace} to relate fake special cycle classes and isotypic part of special cycle classes on Shtukas via the functoriality of categorical trace.
\end{remark}

\subsubsection{Hecke eigensheaves}

We propose a definition of higher integrals using the period sheaf and the Hecke eigensheaves. For this purpose, we recall the notion of Hecke eigensheaf following \cite[\S13]{arinkin2022stacklocalsystemsrestricted}:

\begin{defn}\thlabel{heckeeigen}
    For $\s\in\Loc_{\Gc}^{\res}(k)$ (i.e. a map $i_{\s}:\pt\to\Loc_{\Gc}^{\res}$), a Hecke eigensheaf for $\s$ is an object in \[\Hecke_{\s}:= \Shv_{\Nilp}(\Bun_G)\otimes_{\QCoh(\Loc_{\Gc}^{\res})}\Vect\] where $\QCoh(\Loc_{\Gc}^{\res})$ acts on $\Shv_{\Nilp}(\Bun_G)$ via the spectral action \cite[Theorem\,0.7.4]{arinkin2022stacklocalsystemsrestricted}, and the action of $\QCoh(\Loc_{\Gc}^{\res})$ on $\Vect$ is given by $i_{\s}^*:\QCoh(\Loc_{\Gc}^{\res})\to\Vect$.
\end{defn}

We summarize some useful properties relating to Hecke eigensheaves:
\begin{prop}\thlabel{eigensheafproperties}
The following holds:
    \begin{enumerate}
        \item The map \[\id\otimes i_{\s,*}:\Hecke_{\s}=\Shv_{\Nilp}(\Bun_G)\otimes_{\QCoh(\Loc_{\Gc}^{\res})}\Vect\to \Shv_{\Nilp}(\Bun_G)\] admits left adjoint $\id\otimes i_{\s}^*$;
        \item \label{eigenprojector}\cite[\S13.1.10]{arinkin2022stacklocalsystemsrestricted} The map \[\Hecke_{\s}\xrightarrow{\id\otimes i_{\s,*}}\Shv_{\Nilp}(\Bun_G)\sub\Shv(\Bun_G)\] admits left adjoint $\rP_{\s}:\Shv(\Bun_G)\to\Hecke_{\s}$.
    \end{enumerate}
\end{prop}

Given \itemref{eigensheafproperties}{eigenprojector}, we make the following definition:
\begin{defn}\thlabel{whiteigen}
    For $\s\in\Loc_{\Gc}^{\res}(k)$, we define the \emph{Whittaker normalized} Hecke eigensheaf of $\s$ to be \[\LL_{\s}^{\Whit}:=\rP_{\s}(\cP_{\Whit}^{\norm})\in\Hecke_{\s},\] where $\cP_{\Whit}^{\norm}$ is the normalized Whittaker period sheaf defined in  \autoref{supergeometriclanglands}.
\end{defn}

In concrete terms, giving a Hecke eigensheaf $\LL_{\s}\in\Shv(\Bun_G)$ is equivalent to giving the object $\LL_{\s}\in\Shv(\Bun_G)$ together with compatible isomorphisms \[\IC_{c^*V}\boxstar (\LL_{\s}\boxtimes\uk_{C})\cong\rh_{\fone,\glob,!}(\lh_{\glob}^* \LL_{\s}\otimes \IC_V)\isom V_{\s}\boxtimes  \LL_{\s}\in\Shv(C\times\Bun_G)\] for $V\in\Rep(\Gc)$. Here the two maps $\lh_{\glob},\rh_{\fone,\glob}$ are defined in \autoref{modulispaces}, and we recall that they form the diagram \[\begin{tikzcd}
    & \hk^{\glob}_{G,\fone} \ar[dl, "\lh_{\glob}"'] \ar[dr, "\rh_{\fone,\glob}"] &\\
    \Bun_G&&C\times \Bun_G
\end{tikzcd}.\] The sheaf $V_{\s}\in\Lisse(C)$ is the value of $\s:\Rep(\Gc)\to \QLisse(C)$ at $V\in\Rep(\Gc)$. Here, $\Lisse(C)$ (resp. $\QLisse(C)$) is the category of lisse sheaves (resp. left completion of ind-lisse sheaves) on $C$ defined in \cite[\S1.2]{arinkin2022stacklocalsystemsrestricted}.
In particular, we know that for each $V\in\Rep(\Gc)$ \begin{equation}
    \IC_{c^*V}\boxstar \LL_{\s}\isom  \Gamma(C,V_{\s}) \otimes \LL_{\s}
\end{equation}
where the action $\boxstar$ is the action of $\Sat_{G,\hbar}$ on $\Shv(\Bun_G)$ defined in  \autoref{shvandcat}.

\subsubsection{Geometric isotypic part}

\begin{defn}\thlabel{geoisotypic}
We define \emph{geometric $\s$-isotypic $X$-period integral} to be the vector space \[\int_X\LL_{\s}\in\Vect.\]
\end{defn}

\subsection{Hecke actions and cocommutator relations}\label{heckeactionsection}
In this section, we study the Hecke action on $\int_X\LL_{\s}$ coming from special cohomological correspondences on $\cP_X$. In the Language of \cite{BZSV}, these Hecke actions come from the action of the \emph{RTF-algebra} on the period sheaf. For a more detailed discussion on this perspective, we refer to \cite[\S16.3]{BZSV}.

\subsubsection{Hecke actions}
Suppose we are given a cohomological correspondence \[c_V^{\glob}\in\Cor_{\hk^{\glob}_{G,\fone},\IC_V}(\cP_X,\cP_X\langle d_V-2\rangle )=\Hom^0(\IC_V\boxstar \cP_X\langle -d_V+2\rangle,\cP_X),\] we can define the Hecke action map as the following composition
\begin{equation}\label{heckeactioneigen}
    \begin{split}
        a_{V,\s}:\Gamma(C,V_{\s})\langle-d_V+2\rangle\otimes\int_X\LL_{\s}&=\Gamma(C,V_{\s})\langle-d_V+2\rangle\otimes \Gamma_c(\Bun_G,\cP_X\otimes \LL_{\s}) \\
        &\isom \Gamma_c(\Bun_G, \cP_X\otimes (\IC_{c^*V}\boxstar \LL_{\s}))\langle-d_V+2\rangle\\
        &\isom \Gamma_c(\Bun_G, (\IC_{V}\boxstar\cP_X)\otimes \LL_{\s})\langle-d_V+2\rangle\\
        &\to \Gamma_c(\Bun_G,\cP_X\otimes\LL_{\s})\\
        &=\int_X\LL_{\s}
    \end{split}
\end{equation}

Moreover, given a local special cohomological correspondence datum $(I,V^I,\{c_{V^x}^{\loc}\}_{x\in I})$ for $I=[r]$ as in \thref{loctovarious}, we can form special global cohomological correspondences $\{c_{V^x}^{\glob}\}_{x\in I}$ as in \eqref{removeleg} and define \begin{equation}
    a_{V^I,\s}=a_{V^1,\s}\circ\cdots\circ (\underbrace{\id\otimes \cdots\otimes\id}_{(r-1)~ \mathrm{times}}\otimes a_{V^r,\s}):\bigotimes_{x\in I}\Gamma(C,V^{x}_{\s})\langle -d_I+2r\rangle\otimes \int_X\LL_{\s}\to \int_X\LL_{\s}.
\end{equation}

\subsubsection{Commutator relations}
Now consider the special cohomological correspondences $c_V^{\loc},c_W^{\loc},c_{V\otimes W}^{\loc}$ in \thref{commsetup}. Note that we have a natural map induced by the cup product \begin{equation}
\begin{split}
    \cup:\Gamma(C,V_{\s})\otimes\Gamma(C,W_{\s})\to\Gamma(C,(V\otimes W)_{\s})
    \end{split}.
\end{equation}
We have the following corollary of \thref{geokolycor}:
\begin{cor}\thlabel{cocommutatorrelation}
    Under \thref{commsetup}, assume \thref{geokolycor} is true. We have the following identity of maps \[a_{V\boxtimes W,\s}-a_{W\boxtimes V,\s}=a_{V\otimes W,\s}\circ(\cup \otimes \id):\Gamma(C,V_{\s})\otimes \Gamma(C,W_{\s})\langle -d_V-d_W+4\rangle \otimes\int_X\LL_{\s}\to  \int_X\LL_{\s}.\]
\end{cor}

The proof is straightforward, and we omit it.

\subsection{Middle-dimensional case}\label{middimcase}
In this section, we keep working under \thref{commsetup} and assuming further \thref{middimassum}, which gives us a $G$-invariant bilinear form $b: V\otimes W\to \triv$. We assume this bilinear form is non-degenerate.

\subsubsection{Clifford relations}
Consider the natural map \begin{equation}
\begin{split}
    \ev_{b,\s}:&\Gamma(C,V_{\s})\otimes\Gamma(C,W_{\s}) \langle 2\rangle \\
    \to &\Gamma(C,(V\otimes W)_{\s})\langle 2\rangle \\
    \to &\Gamma(C,\uk)\langle 2\rangle \\
    \isom & k
    \end{split}.
\end{equation} Here, the first map is the cup product, the second map is induced by $b:V\otimes W\to \triv$, the third map is the natural isomorphism given by the fundamental class of $C$.

We have the following reformulation of \thref{cocommutatorrelation}:

\begin{cor}\thlabel{eigenclifrelation}
Under \thref{commsetup} and \thref{middimassum}, assume \thref{geokolycor} is true. We have the following identity of maps \[a_{V\boxtimes W,\s}-a_{W\boxtimes V,\s}= \ev_{b,\s}\otimes\id:\Gamma(C,V_{\s})\otimes \Gamma(C,W_{\s})\langle 2\rangle\otimes\int_X\LL_{\s}\to  \int_X\LL_{\s} .\]
\end{cor}

\subsubsection{Clifford algebra}\label{clifaction}
Continue with the setting in the previous section. We now construct an action of a Clifford algebra on the geometric isotypic part $\int_{X}\LL_{\s}$.

We first construct a symplectic $\Gc$-representation $(K,\om_K)$ as follows: When $V=W$, it is easy to see that $b$ is a symplectic form, we can take $(K,\om_K):=(V,b)$; when $V$ and $W$ are not necessarily isomorphic, we can take $(K,\om_K):=(V\oplus W,\om_{\can})$ where $\om_{\can}$ is the canonical symplectic form constructed in \eqref{symcanb}. Here we use $b$ to identify $V\cong W^*$.

Define $M:=\Gamma(C,K_{\r})\langle 1\rangle$. It is equipped with a symplectic pairing $\om_M$ coming from the symplectic pairing $\om_K$ and cup product. From the map \eqref{heckeactioneigen} (take $I$=[1]), we get an action of the free tensor algebra $M^{\otimes}$ on $\int_X\LL_{\s}$ denoted by \begin{equation}\label{tensoractioneigen}
    a_{M,\s}:M^{\otimes}\otimes \int_X\LL_{\s}\to \int_X\LL_{\s}.
\end{equation} 

Suppose we are further under \thref{minassum}, we can also use $b_{\geo}$ in \eqref{geoform} instead of $b=b_{\l,\mu}$ in \eqref{Xform}. The same construction as above gives us a symplectic form $\om_{K,\geo}$ on $K$ and a symplectic form $\om_{M,\geo}$ on $M$. We have $\om_M=\kappa_{\l,\mu}\cdot\om_{M,\geo}$.
We make the following assumption for simplicity:
\begin{assumption}\thlabel{concentrationassum}
$M\in\Vect^{\he}$.
\end{assumption}

Then we get the following reformulation of \thref{eigenclifrelation}:

\begin{cor}\thlabel{eigenkolyvagin}
    Suppose we are in \thref{commsetup} and further assume \thref{middimassum}, \thref{minassum}, and \thref{concentrationassum}. Suppose \thref{geokolycor} is true, then the action \eqref{tensoractioneigen} factors through the Clifford algebra $\Cl(M)=\Cl(M,\om_M)$ introduced in  \autoref{linearalgebrasetup}, hence gives us an action map \begin{equation}\label{cliffordactioneigen}
    a_{M,\s}:\Cl(M)\otimes \int_X\LL_{\s}\to \int_X\LL_{\s}.\end{equation}
\end{cor}

\begin{remark}
    \thref{concentrationassum} made in \thref{eigenkolyvagin} can be weakened to only require that $\Gamma(C,V_{\s})$ and $\Gamma(C,W_{\s})$ are concentrated in a single degree respectively without difficulty.
 \end{remark}

\subsubsection{Producing Kolyvagin system}
Now we would like to apply the machinery in  \autoref{linearalgebra}. We make the following assumption:
\begin{assumption}\thlabel{findimstalk}
    The geometric $\s$-isotypic $X$-period $\int_X\LL_{\s}\in\Vect$ is a perfect complex.
\end{assumption}

We need Frobenius automorphisms on $M$ and $\int_X\LL_{\s}$. We make the following setup:
\begin{setting}\thlabel{heckeeigenweil}
    We choose $\s \in\Loc_{\Gc}^{\arith}(k)$ (i.e. a Weil $\Gc$-local system on $C$) and a Weil Hecke eigensheaf $\LL_{\s}\in\Hecke_{\s}^{\Frob^*}$, where $\Hecke_{\s}$ is the category of Hecke eigensheaves defined in \thref{heckeeigen}.
\end{setting}

Note that for each $\s\in\Loc_{\Gc}^{\arith}(k)$, the Whittaker normalized Hecke eigensheaf $\LL_{\s}^{\Whit}$ defined in \thref{whiteigen} gives a particular choice of Weil Hecke eigensheaf.

Given a Weil Hecke eigensheaf $\LL_{\s}$, continuing with the assumptions in last section, we get a triple $(M,\om, F)$ as in \thref{Fsymp} where $F$ is given by the Frobenius endomorphism of $M=\Gamma(C, K_{\s})\langle 1\rangle$, and $\int_X\LL_{\s}$ is an $F$-equivariant $\Cl(M)$-module. We can construct the Kolyvagin system \begin{equation}\label{eigenkoly}
    \{z_{\LL_{\s},r}\in M^{*\otimes r}\}_{r\in\Z_{\geq 0}}
\end{equation}
following \thref{kolydefrep}. Then \thref{kolylfun} gives us the following result, which can be regarded as a higher period formula:
\begin{thm}\thlabel{higherformulageneral}
    Keep the same assumptions as \thref{eigenkolyvagin}. The following identity holds for any $r\geq 0$:
    \begin{equation}
        \om_{r,\geo}(z_{\LL_{\s},r},z_{\LL_{\s},r})=\b_{\l,\mu,X}(\log q)^{-r}\left(\frac{d}{ds}\right)^{r}\Big|_{s=0}(q^{(g-1)\dim(K)s}L(K_{\s},s+1/2))
    \end{equation}
    where $\b_{\l,\mu,X}:=\kappa_{\l,\mu}^{-r}\l(\int_X\LL_{\s})(-1)^{r/2}$. The bilinear form $\om_{r,\geo}:(M^{\otimes r})^{\otimes 2}\to k$ is induced from $\om_{M,\geo}:M^{\otimes 2}\to k$ as in \eqref{tensorbformula}, the factor $\l(\int_X\LL_{\s})$ is defined in \eqref{kolynormrep}.
\end{thm}
\begin{proof}
    One just note that we have $\dim(M)=\dim\Gamma(C,K_{\s})\langle 1\rangle=(0|2(g-1)\dim(K))$, \[L(K_{\s},s+1/2) = L(M,F,s),\] and $z_{\LL_{\s},r}$ is non-zero only if $r$ is even. Then the formula follows from \thref{kolylfun}.
\end{proof}

\section{Example: higher Rankin--Selberg convolutions}\label{higherrs}
In this section, we consider the example $G=\GL_n\times \GL_{n-1}$ and $X=\GL_n$, where $G$ acts on $X$ via \[x\cdot(g_n,g_{n-1})=g_{n-1}^{-1}xg_n\] for $(g_n,g_{n-1})\in\GL_n\times\GL_{n-1}$ and $x\in X=\GL_n$, where we regard $\GL_{n-1}$ as a subgroup of $\GL_n$ via $g_{n-1}\mapsto \diag(g_{n-1},1)$. In this case, we have $\Gc=\GL_n\times \GL_{n-1}$ and $\Mc=T^*(\std_n\boxtimes\std_{n-1})$, where the $\Ggr$-action on $\Mc$ is the scaling action (i.e. weight one on both $\std_n\boxtimes\std_{n-1}$ and $(\std_n\boxtimes\std_{n-1})^*$). 
\begin{itemize}
    \item In \autoref{localinput}, we discuss the local aspect of this example.
    \item  In \autoref{globalinput}, we discuss the global aspect of this example.
    \item  In \autoref{higherformulaglnsection}, we combine the input and get the higher Rankin--Selberg integral formula \thref{higherformulagln}.
\end{itemize}

\subsection{Local input}\label{localinput}
We first discuss the local aspect of this example. The key result is \thref{geometricscalargln}.
\subsubsection{Construction of local special cohomological correspondences}\label{consloccohcorgln}
\thref{bzsvloc} in this case is proved in \cite{BFGT}. In particular, it tells us that \thref{middimassum} holds for $V:=\std_n\boxtimes\std_{n-1}$, $W:=\std_n^*\boxtimes\std_{n-1}^*$, and we are going to introduce the local special cohomological correspondences $c_V^{\loc}$ and $c_W^{\loc}$ used in \textit{loc.cit}.
We directly apply the construction in \autoref{dfc}. Using the notations in \textit{loc.cit}, we have $H=\GL_{n-1}$ regarded as a subgroup of $G=\GL_n\times\GL_{n-1}$ via diagonal embedding. Under the standard identification $X_*(T)=\ZZ^{n}\oplus\ZZ^{n-1}$. We have \[\l_V=((1,0,\dots,0),(1,0,\dots,0))\in X_*(T)\] and \[\l_W=((0,\dots,0,-1),(0,\dots,0,-1))\in X_*(T).\] Note also that $X_*(T_H)=\ZZ^{n-1}$. We take \[\l=\l_{\std_{n-1}}:=(1,0,\dots,0)\in X_*(T_H)^+,\] and  $\mu=-\l$. Then we have $\l^+=\l_V$ and $\mu^+=\l_W$. The construction \eqref{dfcloc} gives us \begin{equation}\label{dfclocgln}
    c_V^{\loc}:=c_{V,\l}^{\loc}\in\Cor_{\hk^{\Xp,\loc}_{G},\IC_V}(\d^X_{\loc},\d^X_{\loc}\langle 1\rangle), c_W^{\loc}:=c_{W,-\l}^{\loc}\in \Cor_{\hk^{\Xp,\loc}_{G},\IC_W}(\d^X_{\loc},\d^X_{\loc}\langle 1\rangle).
\end{equation}

\subsubsection{Determine the number $\kappa_{\l,-\l}$ in \eqref{geometricscalar}}\label{geometricscalarglnsection}

Now we can apply the method in  \autoref{verclifrel} to determine the number $\kappa_{\l,-\l}$ in \eqref{geometricscalar}. The result is the following:
\begin{prop}\thlabel{geometricscalargln}
    We have $\kappa_{\l,-\l}=(-1)^{n-1}$. In other word, we have \[c_V^{\loc}\circ c_W^{\loc} - c_V^{\loc}\circ c_W^{\loc}=(-1)^{n-1} \hbar\cdot b_{\geo}\in\Hom^0(\IC_{V\otimes W}*\d^X_{\loc},\d^X_{\loc}\langle 2\rangle).\] Here $b_{\geo}:V\otimes W\to \triv$ is defined in \eqref{geoform}.
\end{prop}

\begin{remark}
    The number $\kappa_{\l,-\l}$ can also be read off from the proof of \thref{bzsvloc} in this case given in \cite{BFGT}. Since we are using different conventions on Hecke actions, to avoid confusion, we prefer to provide a direct approach to determine this number.
\end{remark}

\begin{proof}[Proof of \thref{geometricscalargln}]
Using the notation in \autoref{verclifrel}, we have $\cO(\frt)=k[x_1,\dots,x_n,y_1,\dots,y_{n-1}]$, $\cO(\frt_H)=k[z_1,\dots,z_{n-1}]$, such that the restriction along $\frt_H\to\frt$ is given by $x_i\mapsto z_i$, $y_i\mapsto z_i$.

Let \begin{equation} c_i:=\sum_{J\sub [n],|J|=i}\prod_{j\in J}x_j \end{equation} 
\begin{equation}d_i:=\sum_{J\sub [n-1],|J|=i}\prod_{j\in J}y_j\end{equation} 
\begin{equation}e_i:=\sum_{J\sub [n-1],|J|=i}\prod_{j\in J}z_j\end{equation} 
\begin{equation} x:=x_1 \end{equation} 
\begin{equation} y:=y_1\end{equation} 
\begin{equation} \ovc_i:=\sum_{J\sub [n]-\{1\},|J|=i}\prod_{j\in J}x_j \end{equation} 
\begin{equation} \ovd_i:=\sum_{J\sub [n-1]-\{1\},|J|=i}\prod_{j\in J}y_j \end{equation}
\begin{equation} \ove_i:=\sum_{J\sub [n-1]-\{1\},|J|=i}\prod_{j\in J}z_j \end{equation}

We have \begin{equation} R_G=\cO(\frt)^{W_G}[\hbar]=k[c_1,\dots,c_n,d_1,\dots,d_{n-1},\hbar]\end{equation} \begin{equation}
    R_H=\cO(\frt_H)^{W_H}[\hbar]=k[e_1,\cdots,e_{n-1},\hbar]
\end{equation} and \begin{equation} M_V\isom M_W=k[x,\ovc_1,\dots,\ovc_{n-1},y,\ovd_1,\dots,\ovd_{n-2},\hbar]\langle 2n-3\rangle.\end{equation}

We first describe the $(R_G,R_G)$-bimodule structure on $M_V$ following  \autoref{mv}. The right $R_G$-module structure on $M_V$ is given by the ring map
\begin{equation}\label{glnmvleft}
    \begin{split}
    i_{G,\l}:k[c_1,\dots,c_n,d_1,\cdots,d_{n-1},\hbar]&\to k[x,\ovc_1,\dots,\ovc_{n-1},y,\ovd_1,\dots,\ovd_{n-2},\hbar] \\
        c_i&\mapsto x\ovc_{i-1}+\ovc_{i} \\
        d_i&\mapsto y\ovd_{i-1}+\ovd_{i} \\
        \hbar&\mapsto \hbar
    \end{split}.
\end{equation}
The left $R_G$-module structure on $M_V$ is given by the ring map
\begin{equation}\label{glnmvright}
    \begin{split}
    \tili_{G,\l}:k[c_1,\dots,c_n,d_1,\dots,d_{n-1},\hbar]&\to k[x,\ovc_1,\dots,\ovc_{n-1},y,\ovd_1,\dots,\ovd_{n-2},\hbar] \\
        c_i&\mapsto (x-\hbar)\ovc_{i-1}+\ovc_{i} \\
        d_i&\mapsto (y-\hbar)\ovd_{i-1}+\ovd_{i} \\
        \hbar&\mapsto \hbar
    \end{split}.
\end{equation}

We can similarly describe the $(R_G,R_G)$-bimodule structure on $M_W$: the right $R_G$-module structure on $M_W$ in still given by the same formula \eqref{glnmvleft}, while the left $R_G$-module structure on $M_W$ is given by the ring map
\begin{equation}\label{glnmvright-}
    \begin{split}
    \tili_{G,-\l}:k[c_1,\dots,c_n,d_1,\dots,d_{n-1},\hbar]&\to k[x,\ovc_1,\dots,\ovc_{n-1},y,\ovd_1,\dots,\ovd_{n-2},\hbar] \\
        c_i&\mapsto (x+\hbar)\ovc_{i-1}+\ovc_{i} \\
        d_i&\mapsto (y+\hbar)\ovd_{i-1}+\ovd_{i} \\
        \hbar&\mapsto \hbar
    \end{split}.
\end{equation}

Following \autoref{cv}, the map $c_V^{\coh}:M_V\langle -1\rangle \otimes_{R_G}R_H\to R_H$ can be identified with the composition
\begin{equation}\label{glncv}
\begin{split}
    c_V^{\coh}:&k[x,\ovc_1,\dots,\ovc_{n-1},y,\ovd_1,\dots,\ovd_{n-2},\hbar]\otimes_{R_G} k[e_1,\dots,e_{n-1},\hbar]\langle 2n-4\rangle\\
    \to & k[z,\ove_1,\dots,\ove_{n-2},\hbar]\langle 2n-4\rangle \\
    \to & k[e_1,\dots,e_{n-2},\hbar]
    \end{split}
\end{equation}
The first map in \eqref{glncv} is $k[x,\ovc_1,\dots,\ovc_{n-1},y,\ovd_1,\dots,\ovd_{n-2},\hbar]$-linear on the left and $k[e_1,\dots,e_{n-1},\hbar]$-linear on the right (both consider the most natural module structure). The second map in \eqref{glncv} is the unique $k[e_1,\dots,e_{n-1},\hbar]$-linear map satisfies $z^{n-2}\mapsto (-1)^n$ (see \thref{parabolicpushdes}). Here, the linearity is defined via $\tili_{H,\lambda}$.

Similarly, the map $c_W^{\coh}:M_W\langle -1\rangle \otimes_{R_G}R_H\to R_H$ can be identified with the composition
\begin{equation}\label{glncv-}
\begin{split}
    c_W^{\coh}:&k[x,\ovc_1,\dots,\ovc_{n-1},y,\ovd_1,\dots,\ovd_{n-2},\hbar]\otimes_{R_G} k[e_1,\dots,e_{n-1},\hbar]\langle 2n-4\rangle\\
    \to & k[z,\ove_1,\dots,\ove_{n-2},\hbar]\langle 2n-4\rangle \\
    \to & k[e_1,\dots,e_{n-2},\hbar]
    \end{split}
\end{equation}
The first map in \eqref{glncv-} is $k[x,\ovc_1,\dots,\ovc_{n-1},y,\ovd_1,\dots,\ovd_{n-2},\hbar]$-linear on the left and $k[e_1,\dots,e_{n-1},\hbar]$-linear on the right (both consider the most natural module structure). The second map in \eqref{glncv-} is the unique $k[e_1,\dots,e_{n-1},\hbar]$-linear map satisfies $z^{n-2}\mapsto 1$. Here, the linearity is defined via $\tili_{H,-\lambda}$.

Following \autoref{geocoh}, the map $b_{\geo}^{\coh}$ can be identified with the composition
\begin{equation}\label{glngeo}
    \begin{split}
        b_{\geo}^{\coh}:&k[x,\ovc_1,\dots,\ovc_{n-1},y,\ovd_1,\dots,\ovd_{n-2},\hbar]\otimes_{R_G}k[x,\ovc_1,\dots,\ovc_{n-1},y,\ovd_1,\dots,\ovd_{n-2},\hbar]\langle 4n-8 \rangle \\
        \to& k[x,\ovc_1,\dots,\ovc_{n-1},y,\ovd_1,\dots,\ovd_{n-2},\hbar]\langle 4n-8 \rangle \\
        \to& k[c_1,\dots,c_n,d_1,\dots,d_{n-1},\hbar]
    \end{split}
\end{equation}
The first map in \eqref{glngeo} is the unique $k[x,\ovc_1,\dots,\ovc_{n-1},y,\ovd_1,\dots,\ovd_{n-2},\hbar]$-bilinear map where the left module structure on the second line comes from the twisted ring map $x\mapsto x-\hbar$, $y\mapsto y-\hbar$ (and no twist on other generators). The second map in \eqref{glngeo} is the unique $k[x,\ovc_1,\dots,\ovc_{n-1},y,\ovd_1,\dots,\ovd_{n-2},\hbar]$-linear map such that $x^{n-1}y^{n-2}\mapsto -1$. Here, the linearity is defined via $\tili_{G,\lambda}$.

The only thing needs extra work is the description of $\sw_{V,W}$, which is given by the following lemma:
\begin{lemma}\thlabel{swapgln}
    The map $\sw_{V,W}:M_V\otimes_{R_G} M_W\to M_W\otimes_{R_G} M_V$ satisfies \[\begin{split}
        x^iy^s\otimes x^jy^t \mapsto x^jy^t\otimes x^iy^s
    \end{split}\] for $0\leq i,j\leq n-1$, $0\leq s,t\leq n-2$ such that $0\leq i+j\leq n-1$, $0\leq s+t\leq n-2$.
\end{lemma}
\begin{proof}
    This follows from the characterization of $\sw_{V, W}$ given in \autoref{sw} via elementary arguments. We leave it to the reader. 
\end{proof}

Now we can easily determine the scalar $\kappa_{\l,-\l}$. Consider $x^{n-1}\otimes y^{n-2}\otimes 1\in M_V\langle-1\rangle\otimes_{R_G} M_W\langle-1\rangle\otimes_{R_G} R_H $. We have
\begin{equation}\label{gln1}
    \begin{split}
        c_{V}^{\coh}\circ c_{W}^{\coh}(x^{n-1}\otimes y^{n-2}\otimes 1)&= c_{V}^{\coh}(x^{n-1}\otimes 1) \\
        &= (-1)^{n}(e_1+(n-1)\hbar)
    \end{split}
\end{equation}
\begin{equation}\label{gln2}
    \begin{split}
        c_{W}^{\coh}\circ c_{V}^{\coh}\circ\sw_{V,W}(x^{n-1}\otimes y^{n-2}\otimes 1)&= c_{W}^{\coh}\circ c_{V}^{\coh}(y^{n-2}\otimes x^{n-1}\otimes 1) \\
        &= c_{W}^{\coh}((-1)^{n}y^{n-2}\otimes (e_1+(n-1)\hbar)) \\
        &= (-1)^{n}(e_1+(n-2)\hbar)
    \end{split}
\end{equation}
\begin{equation}\label{gln3}
    \begin{split}
        \kappa_{\l,-\l}\hbar \cdot b_{\geo}^{\coh}(x^{n-1}\otimes y^{n-2}\otimes 1)=-\kappa_{\l,-\l}\hbar
    \end{split}
\end{equation}
Comparing \eqref{gln1}~\eqref{gln2}~\eqref{gln3} with \eqref{cohclif}, we get $\kappa_{\l,-\l}=(-1)^{n-1}$. This concludes the proof.

\end{proof}

\subsection{Global input}\label{globalinput}
Now we turn to the global aspect of the example. The main result in this section is \thref{higherrsbzsvglobal}.

We write $\Bun_n:=\Bun_{\GL_n}$, and we have a connected components decomposition $\Bun_n=\coprod_{d\in \ZZ}\Bun_n^d$ according to the degree. Then we have $\Bun_G^X=\Bun_{n-1}$ and the period sheaf $\cP_X=\pi_!\uk_{\Bun_{n-1}}$ where \begin{equation}\label{glnpi}
    \pi:\Bun_{n-1}\to\Bun_{n-1}\times\Bun_n
\end{equation}
is given by 
\[\pi(\cE)=(\cE,\cE\oplus\cO)\] for $\cE\in\Bun_{n-1}$, where $\cO$ is the structure sheaf on the curve $C$.

\subsubsection{Hecke eigensheaves for $\GL_n$}
In this section, we briefly recall some properties of the (Weil) Hecke eigensheaf $\LL_{\s}^{\FGV}$ for each $\s\in\Loc_{\GL_n}^{\res,\irr}(k)$ (i.e. $\s$ is a \emph{geometrically irreducible} rank $n$ local system on $C$) constructed in \cite{frenkel2002geometric}. Our normalization will be different from the normalization in \textit{loc.cit}. Since the numerical statement \thref{higherformulageneral} is very sensitive to the Weil structure on the Hecke eigensheaf, we add a superscript $\FGV$ in the notation $\LL_{\s}^{\FGV}$ to distinguish it from an arbitrary Hecke eigensheaf.

Fix a line bundle $\cL\in\Pic$. Define $\Bun_n'^{\cL}$ to be the moduli space whose $S$ points is given by \begin{equation}
    \Bun_n'^{\cL}(S)=\{\Om^{(n-1)/2}\otimes\cL\sub\cE|\cE\in\Bun_n(S)\}.
\end{equation}
where we always use $\sub$ to denote inclusion of \emph{coherent sheaves}.
Denote the moduli space of rank 0 coherent sheaves on $C$ by $\Coh_0$. Consider the moduli space $\Bun_{N,\rc(\Om)\otimes\cL}$ defined by \begin{equation}
    \Bun_{N,\rc(\Om)\otimes\cL}(S)=\{\cE_1\sub\cE_2\sub\cdots\sub\cE_n|\textup{$\cE_i\in\Bun_i(S)$, $\cE_i/\cE_{i-1}\isom \Om^{(n-1)/2-i+1}\otimes\cL$, $i\in[n]$}\}.
\end{equation}
Consider the moduli space \begin{equation}
    \rm{Q}_n^{\cL}:=\{\cE_1\sub\cdots\sub\cE_n\sub\cE_{n+1}|(\cE_1\sub\cdots\sub\cE_n)\in\Bun_{N,\rc(\Om)\otimes\cL}(S), \cE_{n+1}\in\Bun_n(S), \cE_{n+1}/\cE_n\in\Coh_0(S)\}
\end{equation}

We have similar degree decomposition $\Coh_0=\coprod_{n\in\ZZ}\Coh_0^d$ and $\rm{Q}_n^{\cL}=\coprod_{n\in\ZZ}\rm{Q}_n^{\cL,d}$, where $\rm{Q}_n^{\cL,d}$ is defined by $\deg(\cE_{n+1})=d$.

There is a natural decomposition $\LL_{\s}^{\FGV}=\bigoplus_{d\in\ZZ}\LL_{\s}^{\FGV,d}$ such that $\LL_{\s}^{\FGV,d}$ is supported on $\Bun_n^d$.

Recall the Laumon sheaf $\cL_{\s}=\bigoplus_{d\in\ZZ_{\geq 0}}\cL_{\s}^d\in\Shv(\Coh_0)$ whose construction is given in \cite[\S2.1]{frenkel2002geometric} (our $\cL_{\s}^d$ is the same as $\cL_{E}^d$ for $E=\s$ in \textit{loc. cit.}). The sheaf $\cL_{\s}^d$ is a shift of an irreducible perverse sheaf and lies in (naive) cohomological degree 0 on the semisimple locus of $\Coh_0$.

For each $d\in\ZZ$, consider diagram
\begin{equation}\label{fgvdiag1}
\begin{tikzcd}
    \rm{Q}^{\cL,d}_n \ar[r, "v_n^{\cL,d}"] \ar[d, "\ev_n^{\cL,d}\times \a_n^{\cL,d}"'] & \Bun_n'^{\cL,d} \ar[r,"\rho_n^{\cL,d}"] & \Bun_n^{\cL,d} \\
    \A^1\times \Coh_0^{d-n\deg(\cL)}&&
    \end{tikzcd}
\end{equation}
where the maps in \eqref{fgvdiag1} are defined by
\[v_n^{\cL,d}(\cE_1\sub\cdots\sub\cE_n\sub\cE_{n+1})=(\Om^{(n-1)/2}\otimes\cL=\cE_1\sub\cE_{n+1})\in\Bun_n'^{\cL,d}\] \[\ev_n^{\cL,d}(\cE_1\sub\cdots\sub\cE_n\sub\cE_{n+1})=q(\cE_1\otimes\cL^{-1}\sub\cdots\sub\cE_n\otimes\cL^{-1})\in\A^1\] \[\a_n^{\cL,d}(\cE_1\sub\cdots\sub\cE_n\sub\cE_{n+1})=\cE_{n+1}/\cE_n\in\Coh_0^{d-n\deg(\cL)}\] for $(\cE_1\sub\cdots\sub\cE_n\sub\cE_{n+1})\in \rm{Q}_n^{\cL,d} $ where $q:\Bun_{N,\rc(\Om)}\to\A^1$ is the map in \eqref{bunndiagram}, and \[\rho_n^{\cL,d}(\Om^{(n-1)/2}\otimes\cL\sub\cE)=\cE\] for $(\Om^{(n-1)/2}\otimes\cL\sub\cE)\in\Bun_n'^{\cL,d}$.

Define \begin{equation}\cW_{\s}^{\cL,d}:=\ev_n^{\cL,d,*}\AS\otimes \a_n^{\cL,d,*}\cL_{\s}\langle \dim(\rm{Q}^{\cL,d}_n) \rangle\end{equation} where $\AS$ is the Artin-Schreier sheaf in \eqref{assheaf} and \begin{equation}
    \dim(\rm{Q}^{\cL,d}_n)=\dim(\Bun_{N,\rc(\Om)})+(d-n\deg(\cL))n=(1-g)(\sum_{i\in[n-1]}i^2)+(d-n\deg(\cL))n
\end{equation} for $d-n\deg(\cL)\geq 0$.

The main input is the following:
\begin{prop}\cite[\S7.9,Corollary\,9.3,\S8.7]{frenkel2002geometric}\thlabel{fgvinput}
\begin{enumerate}
    \item \label{fgvcharacterize} There exists a unique (perverse) Hecke eigensheaf $\LL_{\s}^{\FGV}\in\Hecke_{\s}$ satisfying
    \[\rho_n^{\cO,d,*}\LL_{\s}^{\FGV}\isom v_{n,!}^{\cO,d}\cW_{\s}^{\cO,d}\langle -d+n^2(g-1)\rangle ,\] where $d-n^2(g-1)=\dim(\Bun_n'^{\cO,d})-\dim(\Bun_n)$ ($\dim$ is understood as expected dimension).
    \item \label{fgvcleanext} For every $d\in\ZZ$, $\LL_{\s}^{\FGV,d}:=\LL_{\s}^{\FGV}|_{\Bun_n^d}$ is a clean extension from a quasi-compact open substack of $\Bun_n^d$.
    \item \label{fgvselfdual} We have \[\DD^{\ver}(\LL_{\s}^{\FGV})\isom \LL_{\s^*}^{\FGV}\] where $\DD^{\ver}:\Shv(\Bun_G)\to\Shv(\Bun_G)$ is the Verdier duality functor.
\end{enumerate}
\end{prop}

\begin{remark}
    When $n=1$, one can easily check $\LL_{\s}^{\FGV}\isom\LL_{\s}$ where $\LL_{\s}$ is defined in  \autoref{gcft}.
\end{remark}

We have the following corollary of \thref{fgvinput}:
\begin{cor}\thlabel{fgvtwist}
For any $\cL\in\Pic(\F_q)$, we have
    \[\rho_n^{\cL,d,*}\LL_{\s}^{\FGV}\isom v_{n,!}^{\cL,d}\cW_{\s}^{\cL,d}\langle -(d-n\deg(\cL))+n^2(g-1)\rangle\otimes (\LL_{\det\s})_{\cL}\langle -(g-1)\rangle \] where $(\LL_{\det\s})_{\cL}$ is the stalk at $\cL\in\Pic(\F_q)$ of the Hecke eigensheaf $\LL_{\det\s}$ with eigenvalue $\det\s$ introduced in \autoref{gcft}.
\end{cor}
\begin{proof}[Proof of \thref{fgvtwist}]
    Consider $m_{\cL}:\Bun_n\to\Bun_n$ defined by \begin{equation}
        m_{\cL}(\cE)=\cE\otimes\cL\in\Bun_n
    \end{equation} for $\cE\in\Bun_n$. From the Hecke-eigen property of $\LL_{\s}^{\FGV}$, we know \begin{equation}\label{lbtwist}
        m_{\cL}^*\LL_{\s}^{\FGV}\isom \LL_{\s}^{\FGV}\otimes (\LL_{\det\s})_{\cL}\langle -(g-1)\rangle.
    \end{equation}
    Consider the diagram
    \begin{equation}
        \begin{tikzcd}
            \rm{Q}^{\cL,d}_n \ar[r, "m^{\rm{Q}}_{\cL^{-1}}","\sim"'] \ar[d, "v_n^{\cL,d}"']  & \rm{Q}^{\cO,d-n\deg(\cL)}_n \ar[d, "v_n^{\cO,d}"] \\
            \Bun_n'^{\cL,d} \ar[d, "\rho_n^{\cL,d}"'] \ar[r, "m'_{\cL^{-1}}","\sim"'] & \Bun'^{\cO,d-n\deg(\cL)}_n \ar[d, "\rho_n^{\cO,d}"] \\
            \Bun_n^d \ar[r, "m_{\cL^{-1}}","\sim"'] & \Bun_n^{d-n\deg(\cL)}
        \end{tikzcd}
    \end{equation}
    where \begin{equation}
        m^{\rm{Q}}_{\cL^{-1}}(\cE_1\sub\cdots\sub\cE_n\sub\cE_{n+1})=(\cE_1\otimes\cL^{-1}\sub\cdots\sub\cE_n\otimes\cL^{-1}\sub\cE_{n+1}\otimes\cL^{-1})\in \rm{Q}^{\cO,d-n\deg(\cL)}_n
    \end{equation} for $(\cE_1\sub\cdots\sub\cE_n\sub\cE_{n+1})\in \rm{Q}^{\cL,d}_n $, and \begin{equation}
        m'_{\cL^{-1}}(\Om^{(n-1)/2}\otimes\cL\sub\cE)=(\Om^{(n-1)/2}\sub\cE\otimes\cL^{-1}) \in \Bun'^{\cO,d-n\deg(\cL)}_n
    \end{equation} for $(\Om^{(n-1)/2}\otimes\cL\sub\cE)\in \Bun_n'^{\cL,d}$.
    It is easy to see that all horizontal maps are isomorphisms. Then our desired isomorphism follows from \itemref{fgvinput}{fgvcharacterize}, \eqref{lbtwist} and
    \[m^{\rm{Q},*}_{\cL^{-1}}\cW_{\s}^{\cO,d-n\deg(\cL)}\isom \cW_{\s}^{\cL,d}.\]
    \end{proof}

\begin{remark}
    Our normalization is different from \cite{frenkel2002geometric} as follows: One gets the sheaf $\Aut_E$ for $E=\s$ in \textit{loc.cit} when taking $\cL=\Om^{(n-1)/2}$ instead of $\cL=\cO$ in \itemref{fgvinput}{fgvcharacterize}. Moreover, \thref{fgvtwist} tells us that $\LL_{\s}^{\FGV}\isom \Aut_E\otimes(\LL_{\det E})_{\Om^{(n-1)/2}}\langle -(g-1)\rangle$ for $E=\s$.
\end{remark}

\subsubsection{Computing $\int_X\LL_{\s}^{\FGV}$}
From now on, we fix $\s=(\s_n,\s_{n-1})\in\Loc_{GL_n}^{\res,\irr}(k)\times\Loc_{GL_{n-1}}^{\res,\irr}(k)$. Consider the diagram
\begin{equation}\label{fgvdiag2}
    \begin{tikzcd}
        &\rm{Q}^{\Om^{-n/2},d}_{n-1} \ar[r, "s_{\rm{Q}}^d"] \ar[dl, "v_{n-1}^{\Om^{-n/2},d}"'] \ar[d, "p_{n-1}^{\Om^{-n/2},d}"] & \rm{Q}^{\Om^{-(n-1)/2},d}_{n} \ar[d, "v_n^{\Om^{-(n-1)/2},d}"] \\
        \Bun_{n-1}'^{\Om^{-n/2},d} \ar[r, "\rho_{n-1}^{\Om^{-n/2},d}"'] & \Bun_{n-1}^d \ar[r, "s^d"] \ar[dr, "r^d"'] & \Bun_{n}'^{\Om^{-(n-1)/2},d} \ar[d, "\rho_n^{\Om^{-(n-1)/2},d}"]\arrow[ul, phantom, "\lrcorner", shift right =2, very near end]\\
        && \Bun_n^d
    \end{tikzcd}
\end{equation}
where the right column is the top row of \eqref{fgvdiag1} (taking $\cL=\Om^{-(n-1)/2}$) and the left triangle is also the top row of \eqref{fgvdiag1} for $n-1$ (taking $\cL=\Om^{-n/2}$). The map $r^d$ is defined by \[r^d(\cE)=\cE\oplus\cO\in\Bun_n^d\] for $\cE\in\Bun_{n-1}^d$. The map $s^d$ is defined by \[s^d(\cE)=(\Om^{(n-1)/2}\otimes\Om^{-(n-1)/2}=\cO\sub \cE)\in \Bun_{n}'^{\Om^{-(n-1)/2},d} \] for $\cE\in\Bun_{n-1}^d$. The map $s^d_{\rm{Q}}$ is defined by \[s^d_{\rm{Q}}(\cE_1\sub\cdots\sub\cE_{n-1}\sub\cE_{n})=(\cO\sub\cE_1\oplus\cO\sub\cdots\sub\cE_{n-1}\oplus\cO\sub\cE_n\oplus\cO)\in\rm{Q}^{\Om^{-(n-1)/2},d}_{n} \] for $(\cE_1\sub\cdots\sub\cE_{n-1}\sub\cE_{n})\in\rm{Q}^{\Om^{-n/2},d}_{n-1}$. 

Note that the square in \eqref{fgvdiag2} is Cartesian.
Define $d_0:=d-(n-1)\deg(\Om^{-n/2})=d+n(n-1)(g-1)$. The geometric $\s$-isotypic $X$-period defined in \thref{geoisotypic} in this example is \begin{equation}
    \int_X\LL_{\s}^{\FGV}=\bigoplus_{d\in\ZZ}\int_{X,d}\LL_{\s}^{\FGV}
\end{equation}
where \begin{equation}\int_{X,d}\LL_{\s}^{\FGV}:=\Gamma(\Bun_{n-1}^d,\pi^*\LL_{\s}^{\FGV}|_{\Bun_{n-1}^d}),\end{equation} and the main result in this section is the following (essentially proved in \cite{lysenko2002local}):

\begin{thm}\thlabel{higherrsbzsvglobal}
    For each $d\in\ZZ$, there is a canonical isomorphism \[\int_{X,d}\LL_{\s}^{\FGV}\isom \Sym^{d_0}(\Gamma(C, \s_{n-1}\otimes\s_{n})\langle 1\rangle) \otimes (\LL_{\det\s_{n-1}})_{\Om^{-n/2}}\otimes (\LL_{\det\s_n})_{\Om^{-(n-1)/2}}\langle(n^2-2)(g-1) \rangle.\]
\end{thm}

\begin{proof}[Proof of \thref{higherrsbzsvglobal}]

\rotatebox{90}{%
\begin{minipage}{\textheight}
    \begin{align}
        &\int_{X,d}\LL_{\s}^{\FGV}\\
        =&\Gamma_c(\Bun_{n-1}, \LL_{\s_{n-1}}^{\FGV}\otimes r^{d,*}\LL_{\s_n}^{\FGV} ) \\
        \isom&\Gamma_c(\Bun_{n-1}, \LL_{\s_{n-1}}^{\FGV} \otimes s^{d,*} \rho_n^{\Om^{-(n-1)/2},d,*} \LL_{\s_n}^{\FGV}  ) \\
        \isom&\Gamma_c(\Bun_{n-1}, \LL_{\s_{n-1}}^{\FGV} \otimes s^{d,*} v_{n,!}^{\Om^{-(n-1)/2},d}\cW_{\s_n}^{\Om^{-(n-1)/2},d}\otimes (\LL_{\det\s_n})_{\Om^{-(n-1)/2}}\langle -d+(n-1)(g-1)\rangle) \\
        \label{stepbc}\isom&\Gamma_c(\Bun_{n-1}, \LL_{\s_{n-1}}^{\FGV} \otimes p_{n-1,!}^{\Om^{-n/2},d}s_{\rm{Q}}^{d,*}\cW_{\s_n}^{\Om^{-(n-1)/2},d}\otimes (\LL_{\det\s_n})_{\Om^{-(n-1)/2}}\langle -d+(n-1)(g-1)\rangle)\\
        \isom&\Gamma_c(\Bun_{n-1}, p_{n-1,!}^{\Om^{-n/2},d}(p_{n-1}^{\Om^{-n/2},d,*}\LL_{\s_{n-1}}^{\FGV} \otimes s_{\rm{Q}}^{d,*}\cW_{\s_n}^{\Om^{-(n-1)/2},d}\allowbreak \otimes (\LL_{\det\s_n})_{\Om^{-(n-1)/2}}\langle -d+(n-1)(g-1)\rangle))\\
        \isom&\Gamma_c(\Bun_{n-1}, p_{n-1,!}^{\Om^{-n/2},d}(v_{n-1}^{\Om^{-n/2},d,*}v_{n-1,!}^{\Om^{-n/2},d}\cW_{\s_{n-1}}^{\Om^{-n/2},d}\otimes (\LL_{\det\s_{n-1}})_{\Om^{-n/2}} \otimes s_{\rm{Q}}^{d,*}\cW_{\s_n}^{\Om^{-(n-1)/2},d}\otimes (\LL_{\det\s_n})_{\Om^{-(n-1)/2}}\langle -2d-(g-1)\rangle)) \\
        \isom&   \Gamma_c( \Bun_{n-1}'^{\Om^{-n/2},d}, v_{n-1,!}^{\Om^{-n/2},d}(v_{n-1}^{\Om^{-n/2},d,*}v_{n-1,!}^{\Om^{-n/2},d}\cW_{\s_{n-1}}^{\Om^{-n/2},d}\otimes (\LL_{\det\s_{n-1}})_{\Om^{-n/2}} \otimes s_{\rm{Q}}^{d,*}\cW_{\s_n}^{\Om^{-(n-1)/2},d}\otimes (\LL_{\det\s_n})_{\Om^{-(n-1)/2}}\langle -2d-(g-1)\rangle))   \\
        \isom& \Gamma_c( \Bun_{n-1}'^{\Om^{-n/2},d}, v_{n-1,!}^{\Om^{-n/2},d}\cW_{\s_{n-1}}^{\Om^{-n/2},d}\otimes (\LL_{\det\s_{n-1}})_{\Om^{-n/2}} \otimes v_{n-1,!}^{\Om^{-n/2},d}s_{\rm{Q}}^{d,*}\cW_{\s_n}^{\Om^{-(n-1)/2},d}\otimes (\LL_{\det\s_n})_{\Om^{-(n-1)/2}}\langle -2d-(g-1)\rangle))\\
        \label{lysenkomain}  \isom&\Gamma(C^{(d_0)}, (\s_{n-1}\otimes\s_n)^{(d_0)} )\langle d_0\rangle \otimes (\LL_{\det\s_{n-1}})_{\Om^{-n/2}}\otimes (\LL_{\det\s_n})_{\Om^{-(n-1)/2}}\langle(n^2-2)(g-1) \rangle \\
        \isom&\Sym^{d_0}(\Gamma(C, \s_{n-1}\otimes\s_{n})\langle 1\rangle) \otimes (\LL_{\det\s_{n-1}})_{\Om^{-n/2}}\otimes (\LL_{\det\s_n})_{\Om^{-(n-1)/2}}\langle(n^2-2)(g-1) \rangle
    \end{align}
    \end{minipage}%
    }
    \newpage

    We make some remarks on this seemingly burdensome computation. The step \eqref{stepbc} uses proper base change for the Cartesian square in \eqref{fgvdiag2}. The step \eqref{lysenkomain} is a direct consequence of \cite[\S2.1\,Main Local Theorem]{lysenko2002local}. All other steps are straightforward given \thref{fgvinput} and \thref{fgvtwist}.

\end{proof}

    \begin{remark}
        For the convenience of the reader, we compare our notations with notations used in \cite{lysenko2002local}. Our moduli space $\rm{Q}^{\Om^{-n/2},d}_{n-1}$ is isomorphic to $\leftindex_{n-1}\cQ_{d_0}$ in \textit{loc.cit} via tensoring $\Om^{n-1}$.\footnote{Without tensoring, our moduli $\rm{Q}^{\Om^{(n-2)/2},d+(n-1)(n-2)(g-1)}_{n-1}$ is literaly the same as $\leftindex_{n-1}\cQ_{d}$ in \textit{loc.cit}} Under this isomorphism, forgetting the parity, our sheaf $\cW_{\s_{n-1}}^{\Om^{-n/2},d}$ can be identified with $\leftindex_{n-1}\cF_{\s_{n-1}}^{d_0}$ in \textit{loc.cit}, and $s_{\rm{Q}}^{d,*}\cW_{\s_n}^{\Om^{-(n-1)/2},d}$ can be identified with $\leftindex_{n-1}\cF_{\s_{n}}^{d_0}\langle d_0 - (n-1)^2(g-1) \rangle$ in \textit{loc.cit}.
    \end{remark}

    \begin{remark}
        When $n=2$, the identity \eqref{lysenkomain} is obvious, and we recommend the reader work through this case to see what is happening here.
    \end{remark}

    \subsection{Obtaining higher period formulas}\label{higherformulaglnsection}
    Now we combine all the input from \autoref{localinput} and \autoref{globalinput} together, and prove the higher Rankin--Selberg period integral formula \thref{higherformulagln}(\thref{higherformulaglnfirst}).
    
    Recall we use $g$ to denote the genus of the curve $C$. Since there is no geometrically irreducible local system of rank $\geq2$ on $C$ when $g=1$, we can assume $g\neq 1$.
    Since \thref{bzsvloc} is true in this case, we know that \thref{genassum} and \thref{techassum} are true. Then \thref{geokolycorthm} implies that \thref{geokolycor} is true for the local special cohomological correspondences $c_V^{\loc}$ and $c_W^{\loc}$ constructed in \eqref{dfclocgln}.
    
    Now, we can apply the construction in \autoref{clifaction}.
    We take $K=V\oplus V^*$, where $V=\std_n\boxtimes\std_{n-1}$. Let $\l=\l_{\std_{n-1}}\in X_*(T_{\GL_{n-1}})$. Then we have two non-degenerate bilinear forms $b=b_{\l,-\l}:V\otimes V^*\to\triv$ and $b_{\geo}:V\otimes V^*\to \triv$ defined in \eqref{Xform} and \eqref{geoform} respectively. These two bilinear forms give two symplectic forms $\om_M:M^{\otimes 2}\to k$ and $\om_{M,\geo}:M^{\otimes 2}\to k$ as in \autoref{clifaction}. Then \thref{geometricscalargln} implies $\om_M=(-1)^{n-1}\om_{M,\geo}$.
    
    Using the local special cohomological correspondences $c_V^{\loc}$ and $c_{V^*}^{\loc}$, \thref{eigenkolyvagin} makes $\int_X\LL_{\s}^{\FGV}$ an $F$-equivariant $\Cl(M)$-module introduced in \thref{kolyrep}. To completely determine this module, we only need to compute the invariant $\l(\int_X\LL_{\s}^{\FGV})$ defined in \eqref{kolynormrep}. 

    Note that $\int_X\LL_{\s}^{\FGV}=\bigoplus_{d\in\ZZ}\int_{X,d}\LL_{\s}^{\FGV}$ is naturally $\ZZ$-graded on which $\Gamma(C,V_{\s})\langle 1\rangle$ acts by degree 1 endomorphisms and $\Gamma(C,V^*_{\s})\langle1\rangle$ acts by degree -1 endomorphisms. In particular, any non-zero vector in the one-dimensional subspace $\int_{X,2n(n-1)(g-1)}\LL_{\s}^{\FGV}\sub \int_X \LL_{\s}^{\FGV} $ will be a highest vector in the sense of \thref{kolynormstandard}. Then \thref{kolynormstandard} implies that \begin{equation}
    \begin{split}
        \l(\int_X\LL_{\s}^{\FGV})&=(F|_{\int_{X,2n(n-1)(g-1)}\LL_{\s}^{\FGV}})^2\e(\s_{n}\otimes\s_{n-1})^{-1}\\&=q^{-n^2(g-1)}\chi^{-n}_{\det\s_{n-1}}(\Om)\chi^{-n+1}_{\det\s_{n}}(\Om)\e(\s_{n}\otimes\s_{n-1})
        \end{split}
    \end{equation}
    The Hecke characters $\chi_{\det\s_{n-1}}$, $\chi_{\det\s_n}$ are introduced in \autoref{cft}.

    Form the Kolyvagin system \begin{equation}
    \{z_{\LL_{\s}^{\FGV},r}\in M^{*\otimes r}\}_{r\in\Z_{\geq 0}}
\end{equation} as in \eqref{eigenkoly}. We can now apply \thref{higherformulageneral} and get the following:
\begin{thm}\thlabel{higherformulagln}
    Under the setting above, the following identity holds for any $r\geq 0$:
    \begin{equation}
        \om_{r,\geo}(z_{\LL_{\s}^{\FGV},r},z_{\LL_{\s}^{\FGV},r})=\b_{\s}(\log q)^{-r}\left(\frac{d}{ds}\right)^{r}\Big|_{s=0}(q^{2n(n-1)(g-1)s}L(\s_{n}\otimes\s_{n-1},s+1/2)L(\s_{n}^*\otimes\s_{n-1}^*,s+1/2))
    \end{equation}
    where the bilinear forms $\om_{r,\geo}$ are induced from $\om_{\geo}$, and \[\b_{\s}:=(-1)^{r/2}q^{-n^2(g-1)}\chi^{-n}_{\det\s_{n-1}}(\Om)\chi^{-n+1}_{\det\s_n}(\Om)\e(\s_{n}\otimes\s_{n-1}).\] Here the Hecke characters $\chi_{\det\s_{n-1}}$, $\chi_{\det\s_n}$ are introduced in  \autoref{cft}.
\end{thm}

\begin{remark}
    \thref{fundtoglob} gives us a simple description of the special cohomological correspondences $c_V,c_{V^*}$ on the period sheaf $\cP_X$. This allows us to formulate \thref{higherformulagln} completely in terms of global constructions.
\end{remark}



\bibliographystyle{amsalpha}
\bibliography{Bibliography}

\end{document}

%% file: newintro.tex
\section{Introduction}

\subsection{Background and motivation}

\subsubsection{Period integrals and special values of $L$-functions}
A central theme in modern number theory and the Langlands program is the relationship between \emph{period integrals} of automorphic forms and special values of $L$-functions. Given a reductive group $G$ over a global field $F$, a subgroup $H\sub G$, and a cuspidal automorphic form $\ph$ on $G$, the period integral
\begin{equation}\label{numperiod}
\cP_H(\ph)=\int_{[H]}\ph(h)\,dh
\end{equation}
is expected to encode special values of $L$-functions attached to the automorphic representation generated by $\ph$. Prototypical examples include the Waldspurger formula \cite{waldspurger1985valeurs}, relating toric periods on $\PGL_2$ to $L(\pi,1/2)$, and the Rankin--Selberg integral, relating periods for $\GL_{n-1}\sub \GL_n\times\GL_{n-1}$ to $L(\pi_n\times\pi_{n-1},1/2)$.

A conceptual framework for such relations has emerged in recent years. Sakellaridis and Venkatesh \cite{sakellaridis2017periods} relate periods for spherical varieties to Plancherel decompositions and $L$-functions. The recent work of Ben-Zvi, Sakellaridis, and Venkatesh \cite{BZSV} further develops this formalism into \emph{relative Langlands duality}: to a hyperspherical Hamiltonian $G$-variety $M$ one associates a dual Hamiltonian $\Gc$-variety $\Mc$, and the period integral for $M$ is expected to be related to an $L$-value read off from $\Mc$. This framework unifies many known results, including the Waldspurger and Rankin--Selberg cases above, and predicts new ones.

\subsubsection{Special cycles and higher derivatives of $L$-functions}
While period integrals capture \emph{values} of $L$-functions, deeper arithmetic information lies in higher \emph{derivatives} of $L$-functions. A prototypical example is the Gross--Zagier formula \cite{gross1986heegner}, which relates the N\'eron--Tate height of Heegner points (special cycles) on modular curves to $L'(E,1)$ for an elliptic curve $E$. Combined with Kolyvagin’s Euler system method, this yields the Birch and Swinnerton-Dyer conjecture for elliptic curves of analytic rank at most one. More generally, the Beilinson--Bloch--Kato conjectures predict that the order of vanishing of $L(\pi,s)$ at the central point equals the rank of a Selmer group, and that the leading Taylor coefficient is governed by arithmetic invariants. These conjectures remain widely open over number fields.\par

Over number fields, the geometric meaning of higher derivatives of $L$-functions remains mysterious even for elliptic curves, and is far out of reach in higher rank. By contrast, the function field setting provides a fertile testing ground, largely due to the availability of geometric methods, in particular the existence of moduli of Shtukas (with multiple legs), which serve as function field analogues of Shimura varieties.

A first breakthrough came in the work of Yun and Zhang \cite{YZ1}\cite{YZ2}, who established a \emph{higher Gross--Zagier formula} generalizing the Waldspurger formula: for certain cuspidal automorphic representations $\pi$ of $\PGL_2$ over $\mathbb{F}_q(C)$, the self-intersection number of the $\pi$-isotypic component of the Heegner--Drinfeld cycle on the moduli of $\PGL_2$-Shtukas with $r$ legs is related to the $r$-th central derivative of $L(\pi,s)$. This result has since inspired further developments, including higher theta series and higher Siegel--Weil formulas \cite{FYZ1,FYZ2,FYZ3}, as well as arithmetic volume computations \cite{feng2026arithmeticvolumesmodulistacks}.

However, their approach relies on the compactness of Heegner--Drinfeld cycles, a feature specific to toric periods that rarely holds in higher-dimensional settings. Moreover, their argument is intrinsically tied to the relative trace formula and does not yield a conceptual mechanism applicable in general. In particular, for a general period attached to a $G$-variety $X$, where no relative trace formula is available, it remains unclear why higher derivatives should arise.

\subsubsection{A conceptual and uniform framework in the strongly tempered case}

In this article, we develop a framework addressing these questions in broad generality. The term ``higher'' in the title has a four-fold meaning: higher derivatives of $L$-functions, higher rank groups, higher categorical structures, and a higher vantage point that unifies recent advances. We extend the relative Langlands framework of \cite{BZSV} in the strongly tempered case to higher derivatives and also the work of Yun--Zhang to higher rank.

The idea can be summarized as follows. We replace the period integral \eqref{numperiod} by its sheaf-theoretic counterpart \eqref{geoperiodintro}, which recovers the classical period integral upon taking the trace of Frobenius. These geometric period integrals admit additional natural endomorphisms given by \emph{$L$-observables}. Taking the trace of these endomorphisms composed with Frobenius then produces special cycle classes and higher derivatives of $L$-functions. 

This perspective connects several developments in the (geometric) Langlands program. The study of geometric period integrals was initiated by Lysenko \cite{lysenko2002local}; the action of $L$-observables was proposed in \cite{BZSV}; and the idea of recovering Shtuka-type information via traces is motivated by \cite{arinkin2022automorphicfunctionstracefrobenius}. 
Combining these ingredients yields a uniform and conceptual way to generalize the work of Yun--Zhang within geometric Langlands theory and also extends relative Langlands framework.

\subsection{Main result: a general formulation}\label{mainresultgen}
In this section, we formulate our main result under a general setup.

\subsubsection{Geometric result}
We fix a smooth, projective, geometrically connected curve $C$ over $\FF_q$ of genus $g \neq 1$, and a split reductive group $G$ over $\FF_q$. Let $X$ be a \emph{strongly tempered} affine smooth $G$-variety, and consider the natural map
\[
\pi:\Bun_G^X \to \Bun_G.
\]

\begin{example}
If $X$ is $G$-homogeneous, so that $X=H\bs G$ for some subgroup $H\sub G$, we refer to $X$ by the group pair $(G,H)$. Then $\Bun_G^X=\Bun_H$.
\end{example}

Let $\Gc$ be the Langlands dual group over a finite extension $k/\Ql$. Given a $\Gc$-local system $\s$ and a Hecke eigensheaf $\LL_{\s}\in \Shv(\Bun_G)$ with eigenvalue $\s$, define the \emph{geometric period integral}
\begin{equation}\label{geoperiodintro}
\int_X \LL_{\s}:=H^*_{c}(\Bun_G^X,\pi^*\LL_{\s})\in \Vect.
\end{equation} This recovers the numerical period \eqref{numperiod} by taking the trace of Frobenius.

Let $K$ be the dual $\Gc$-Hamiltonian space following \cite{BZSV} (denoted $\Mc$ in \textit{loc.\,cit.}). The strongly tempered condition means that $K$ is a symplectic $\Gc$-representation. Let $K_{\s}$ be the local system on $C$ associated to $K$ and $\s$. The symplectic form $\omega_K$ on $K$ induces an orthogonal form $\om_M$ on
$
M=H^1(K_{\s}):=H^1(C, K_{\s}),
$
and hence gives a Clifford algebra $\Cl(M):=\Cl(M,\om_M)$.

We now state the main geometric result (see \thref{eigenkolyvagin} for a precise formulation), asserting that the geometric period integral \eqref{geoperiodintro} carries a Clifford module structure:
\begin{thm}\thlabel{clifgen}
    Assuming \cite[Conjecture 8.5.1]{BZSV}, there exists an action of $\Cl(M)$ on $\int_X \LL_{\s}$.
\end{thm}

\begin{remark}
	The Clifford algebra action in \thref{clifgen} is predicted in \cite[\S18]{BZSV} as an action by \emph{$L$-observables}, arising from the conjectural description of the factorizable Plancherel algebra, which is widely open. However, here we only assume \cite[Conjecture\,8.5.1]{BZSV}, which only involves $\mathrm{PL}_{X,\hbar}$ the (non-factorizable) Plancherel algebra with loop rotation and is known in several important cases, including the Rankin--Selberg case $(G,H)=(\GL_n\times\GL_{n-1},\GL_{n-1})$ \cite{BFGT}, and the Gross--Prasad case $(G,H)=(\SO_n\times\SO_{n-1},\SO_{n-1})$ in \cite{Braverman_2022}. In particular, \thref{clifgen} is unconditional in these cases of central interest in the relative Langlands program.
\end{remark}

\subsubsection{Arithmetic result}

We now state our formula for the derivatives of $L$-functions. Let $F\in \End(\int_X\LL_{\s})$ denote the Frobenius endomorphism, and retain the assumptions of \thref{clifgen}. Let
\[
a:\Cl(M)\to \End(\int_X\LL_{\s})
\]
be the action map.  Define elements $z_{\LL_{\s},r}\in (M^{\otimes r})^*$ by
\begin{equation}\label{kolygenintro}
z_{\LL_{\s},r}(m_1\otimes\cdots\otimes m_r)
=\tr\!\big(a(m_1)\circ\cdots\circ a(m_r)\circ F,\;\int_X\LL_{\s}\big).
\end{equation}
We assume that $\int_X\LL_{\s}$ is finite-dimensional so that the trace is well-defined, and that $H^0(K_{\s})=0$.

We state a rough form of the main arithmetic theorem (see \thref{higherformulageneral} for a precise and more canonical version), which follows from the geometric result and Kolyvagin system formalism in \autoref{linearalgebra}.

\begin{thm}\thlabel{higherformulagen}
Under the above assumptions,
\begin{equation}\label{higherformulagenformula}
\om_{r}(z_{\LL_{\s},r},z_{\LL_{\s},r})
=\b_{\s,X}(\log q)^{-r}\left(\frac{d}{ds}\right)^{r}\Big|_{s=0}
\big(q^{(g-1)\dim(K)s}L(K_{\s},s+1/2)\big),
\end{equation}
where $\b_{\s,X}$ is a constant depending only on $X$ and $\s$, and $\om_{r}$ is the bilinear form on $(M^{\otimes r})^*$ induced by $\om_M$.
\end{thm}

\begin{remark}
The elements $z_{\LL_{\s},r}$ are expected to correspond to the $\s$-isotypic components of special cycle classes in the cohomology of Shtukas. This relation is made precise in the companion work \cite{trace}, under suitable assumptions.
\end{remark}

This main arithmetic theorem establishes a new framework to understand the connection between special cycle (classes) and higher derivatives of $L$-functions, and the explication of the constants boils down to explication of the Clifford module structure (i.e. understanding of $L$-observables). As an example to show the power of this framework on higher rank case, we will work out explicitly Rankin--Selberg case.

\subsection{Example: Rankin--Selberg case}\label{mainresultrs}

We specialize the results of \autoref{mainresultgen} to the Rankin--Selberg case $(G,H)=(\GL_n\times\GL_{n-1},\GL_{n-1})$ and explicate the constants. In this setting, one has
\[
K=T^*(\std_{n}\boxtimes\std_{n-1})\in \Rep(\Gc).
\]

\subsubsection{Geometric result}

Fix a geometrically irreducible $\Gc$-local system $\s=(\s_n,\s_{n-1})$, where $\s_n$ and $\s_{n-1}$ have ranks $n$ and $n-1$, respectively. By construction in \cite{frenkel2002geometric}, we have the Hecke eigensheaf
\[
\LL_{\s}^{\FGV}=\LL_{\s_n}^{\FGV}\boxtimes\LL_{\s_{n-1}}^{\FGV}\in \Shv(\Bun_G)
\]
 with eigenvalue $\s_n\boxtimes\s_{n-1}$.
By \thref{clifgen}, there is a canonical action of $\Cl(M)$ on $\int_X\LL_{\s}^{\FGV}$. In this case, we can explicate this Clifford module structure completely:
\begin{thm}[\thref{higherrsbzsvglobal}]\thlabel{cliffirst}
The Clifford module $\int_X\LL_{\s}^{\FGV}$ is the unique irreducible representation of $\Cl(M)$.
\end{thm}

\subsubsection{Arithmetic result}

We now turn to derivatives of $L$-functions. In this case,
\[
L(K_{\s},s)=L(\s_{n}\otimes\s_{n-1},s)\,L(\s_{n}^*\otimes\s_{n-1}^*,s).
\]

\begin{thm}[\thref{higherformulagln}]\thlabel{higherformulaglnfirst}
We have
\begin{equation}\label{higherformulaglnformula}
\om_{r}(z_{\LL_{\s}^{\FGV},r},z_{\LL_{\s}^{\FGV},r})
=\b_{\s}(\log q)^{-r}\left(\frac{d}{ds}\right)^{r}\Big|_{s=0}
\big(q^{2n(n-1)(g-1)s}L(\s_{n}\otimes\s_{n-1},s+1/2)L(\s_{n}^*\otimes\s_{n-1}^*,s+1/2)\big),
\end{equation}
where $\om_{r}$ is the bilinear form on $M^{\otimes r}$ induced by $\om_M$, and
\[
\b_{\s}=(-1)^{r/2}q^{-n^2(g-1)}\chi^{-n}_{\det\s_{n-1}}(\Om)\chi^{-n+1}_{\det\s_n}(\Om)\e(\s_{n}\otimes\s_{n-1}).
\]
Here $\chi_{\det\s_{n-1}},\chi_{\det\s_n}:\Pic(\F_q)\to k^{\times}$ are the Hecke characters attached to $\det\s_{n-1}$ and $\det\s_n$, $\Om\in\Pic(\F_q)$ is the canonical line bundle on $C$, and $\e(\s_n\otimes\s_{n-1})=\det(\Frob, H^1(C,\s_n\otimes\s_{n-1})(1/2))$ is the root number.
\end{thm}

\begin{remark}
\thref{higherformulaglnfirst} depends on the choices of Weil structures on $\LL^{\FGV}_{\s}$, the symplectic form on $K$, and the map $M\to \End(\int_X\LL_{\s}^{\FGV})$, all taken to be canonical. Choosing a Hecke eigensheaf should be viewed as the geometric analogue of choosing a test vector in classical period formulas, with $\LL_{\s}^{\FGV}$ corresponding to a Whittaker-normalized choice.
\end{remark}

\subsection{Outline of the proof}

We now outline the proof of the main theorems. It proceeds in five steps, corresponding to \autoref{linearalgebra}, \autoref{commutatorrelations}, \autoref{automorphiccliffordrelations}, \autoref{higherperiodintegrals}, and \autoref{higherrs}, which we briefly summarize below.

\subsubsection{Linear algebra: Kolyvagin systems}
The first step (\autoref{linearalgebra}) is to introduce a general mechanism to deduce arithmetic results (\thref{higherformulagen}) from geometric results (\thref{clifgen}). The abstract situation can be summarized as follows: Given a Clifford algebra $\Cl(M)$ and a finite-dimensional module $T\in\Mod(\Cl(M))$ equipped with compatible automorphisms $F$, one can construct a sequence of elements $\{z_{r}\in M^{*\otimes r}\}_{r\in\ZZ}$ called \emph{Kolyvagin system (over function fields)} as in \eqref{kolygenintro}, and get a formula similar to \eqref{higherformulagenformula}. The Kolyvagin system over function fields is first considered in the work \cite{YZ3} under a different formulation. Compared to their original formulation, our construction uses the theory of Clifford algebra, which simplifies and conceptualizes the entire picture. 


The most important result in this part is \thref{kolylfun}, which is a prototype of \eqref{higherformulagenformula}.

\subsubsection{Commutator relations on cohomological correspondences: local-to-global}
Section \autoref{commutatorrelations} provides the first step toward constructing an action of $\Cl(M)$ on $\int_X \LL_{\s}$: define the action of generators and verify the relations by a local-to-global argument. The generators act via \emph{special cohomological correspondences} $c_V$ on the period sheaf $\cP_X=\pi_!\uk$, which underlie special cycles. We construct $c_V$ and verify relations by studying their local counterparts $c_V^{\loc}\in\Hom(V,\PL_{X,\hbar})$.

This local-to-global procedure is the focus of \autoref{commutatorrelations}. The main result, \thref{geokolycorthm}, establishes the \emph{commutator relation for special cohomological correspondences} (\thref{geokolycor}) under certain technical assumptions (\thref{techassum}).

\subsubsection{Automorphic Clifford relations}
Section \autoref{automorphiccliffordrelations} provides the second step toward constructing an action of $\Cl(M)$ on $\int_X \LL_{\s}$. We single out the minuscule and \emph{Poisson-pure case} (\thref{minassum} and \thref{middimassum}), which includes most situations where a Clifford action is expected. \thref{minassum} allows us to construct special cohomological correspondences $c_V$ via derived fundamental classes. Then the local-to-global procedure of the previous section applies, and the resulting commutator relation, called the \emph{automorphic Clifford relation} (\thref{geokolycormidmin}), gives a universal form of the Clifford algebra relations. This section provides a geometric origin of the Clifford action, which is crucial for relating $c_V$ to special cycle classes in \cite{trace}.

\subsubsection{Higher period integrals and derivatives of $L$-functions}
In this section \autoref{higherperiodintegrals}, we combine all results in previous sections to finally construct an action of $\Cl(M)$ on $\int_X\LL_{\s}$ and apply \autoref{linearalgebra} to deduce the main arithmetic result \thref{higherformulageneral}, which is the precise formulation of \thref{higherformulagen}.

\subsubsection{Example: higher Rankin--Selberg convolutions}

In this section \autoref{higherrs}, we specialize to the Rankin--Selberg case and explicate all constants. This requires a refined analysis of the Plancherel algebra (see \autoref{localinput}) and a computation of the geometric period integral (see \autoref{globalinput}). Combining these ingredients, we complete the proofs of \thref{cliffirst} (\thref{higherrsbzsvglobal}) and \thref{higherformulaglnfirst} (\thref{higherformulagln}).

\subsection*{Acknowledgment}
Z.\,Wang would like to thank his advisor Zhiwei Yun for suggesting this topic, generously sharing his thoughts and ideas. This work would never have come out without his help. Also, he would like to thank Guoquan Gao, Weixiao Lu, Liang Xiao, Xiangqian Yang, Wei Zhang, Zhiyu Zhang, and Daming Zhou for helpful discussions. S.\,Liu would like to thank his advisor Xinwen Zhu for his inspiring suggestions. He would also like to thank Brian Conrad and Dongryul Kim for valuable discussions.